\documentclass[12pt]{amsart}
\usepackage{amsmath}
\usepackage[backref=page]{hyperref}
\usepackage{tikz,float,caption}
\usetikzlibrary{arrows.meta,calc,decorations.markings,patterns,cd,patterns.meta}

\usepackage[margin=1.25 in,top=1.2in, bottom=1.2 in]{geometry}

\usepackage[bb=ams,scr=euler]{mathalpha}

\usepackage{setspace}
\usepackage{enumitem}
\setenumerate{
  wide=10pt,
  leftmargin=0pt,
  labelwidth=*,
  label=(\roman*),
  topsep=0pt,
  listparindent=0pt,
  parsep=4pt}

\usepackage{thmtools}
\declaretheoremstyle[headfont=\normalsize\normalfont\bfseries,notefont=\mdseries, notebraces={(}{)},bodyfont=\normalfont,postheadspace=0.5em]{basicstyle}

\declaretheorem[style=basicstyle,name=Theorem,numberwithin=section]{theorem}
\declaretheorem[name=Definition,style=basicstyle,sibling=theorem]{defn}
\declaretheorem[name=Remark,style=basicstyle,sibling=theorem]{remark}

\declaretheorem[style=basicstyle,name=Corollary,sibling=theorem]{cor}
\declaretheorem[style=basicstyle,name=Claim,sibling=theorem]{claim}
\declaretheorem[style=basicstyle,name=Proposition,sibling=theorem]{prop}
\declaretheorem[style=basicstyle,name=Lemma,sibling=theorem]{lemma}
\declaretheorem[style=basicstyle,name=Proof,numbered=no,qed=$\square$]{preproof}
\renewenvironment{proof}{\preproof}{\endpreproof}
\newenvironment{dmatrix}{\left[\,\begin{matrix}}{\end{matrix}\,\right]}

\makeatletter

\newcommand{\abs}[1]{\left|#1\right|}
\newcommand{\bd}{\partial}

\newcommand{\cl}[1]{\bar{#1}}

\newcommand{\C}{\mathbb{C}}

\renewcommand{\d}{\mathrm{d}}

\newcommand{\intprod}{\mathbin{{\tikz{\draw(-0.1,0)--(0.1,0)--(0.1,0.2)}\hspace{0.5mm}}}}
\newcommand{\ip}[1]{\langle#1\rangle}
\newcommand{\norm}[1]{\left\lVert#1\right\rVert}

\newcommand{\pd}[2]{\frac{\partial #1}{\partial #2}}
\newcommand{\pr}{\mathrm{pr}}

\newcommand{\qs}{\partial_s Q_n}
\newcommand{\qt}{\partial_t Q_n}
\newcommand{\R}{\mathbb{R}}
\renewcommand{\setminus}{\mathbin{\tikz[baseline={(0,0.033)}]{\draw[line width=0.4, rounded corners=0.5,line cap=round,scale=0.7] (0,0.3)--(0.25,0.1);}}}
\newcommand{\set}[1]{\left\{#1\right\}}
\renewcommand\section{\@startsection{section}{1}{0pt}{-3.5ex \@plus -1ex \@minus -.2ex}{2.3ex \@plus.2ex}{\centering\bfseries}}
\renewcommand{\subsection}{\@startsection{subsection}{2}%
  \z@{.5\linespacing\@plus.7\linespacing}{-.5em}%
  {\normalfont\itshape}}

\newcommand{\Z}{\mathbb{Z}}

\makeatother

\usetikzlibrary{cd,decorations.markings,arrows.meta,intersections,calc,decorations.pathmorphing,patterns,backgrounds,positioning,patterns.meta}

\usepackage{thmtools,thm-restate}

\title[Adiabatic compactness for Nearby Lagrangians]{Adiabatic compactness for holomorphic curves with boundary on nearby Lagrangians}
\author{Dylan Cant}
\author{Daren Chen}
\date{\today}
\begin{document}
\maketitle
\begin{abstract}
  In the paper \cite{FloerW}, Floer established a connection between holomorphic strips with boundary on a Lagrangian $L$ and a small Hamiltonian push-off $L_{f}$, and gradient flow lines for the function $f$. The present paper studies the compactness theory for holomorphic curves $u_{n}$ whose boundary components lie on Hamiltonian perturbations $L_{n}^{1},\dots,L^{N}_{n}$ of a fixed Lagrangian $L$, where each sequence of nearby Lagrangians $L^{j}_{n}$ converges to $L$ as $n\to\infty$. Generalizing earlier work of Oh, Fukaya, Ekholm, and Zhu, we prove that the limit of a sequence of such holomorphic maps is a configuration consisting of holomorphic curves with boundary on $L$ joined by gradient flow lines connecting points on the boundary of holomorphic pieces. The key new result is an exponential estimate analyzing the interface between the holomorphic parts and the gradient flow line parts.
\end{abstract}

\tableofcontents

\section{Introduction}
\label{sec:intro}
The main result of this paper is a compactness result for sequences of holomorphic curves taking boundary values on a particular kind of degenerating sequence of Lagrangians. The limit object in the compactness statement will be a nodal holomorphic curve together with gradient flow lines joining points on the boundary of the holomorphic curve. The compactness result is similar to those in \cite{fukaya-oh, ekholm} which assert that holomorphic curves converge to gradient flow trees. Our result allows for more general symplectic manifolds, and non-constant holomorphic curves may appear in the limit. Holomorphic curves connected by gradient flow lines have been studied in \cite{oh_relative,cornea_lalonde,biran-cornea,biran_cornea_CRM,biran_cornea_lagrangian_topology,seidel_homological,sheridan_homological,charest_2012,alston,charest_woodward_2017,alston_bao}. The ``adiabatic'' phenomenon of holomorphic curves (and, more generally, solutions to certain non-linear elliptic PDE defined on surfaces) degenerating to gradient flow lines has been studied in the closed string case in \cite{oh_duke_05,MT,ct,OZ1,OZ2,cant2021perturbed}.

To set the stage, suppose that $(W,\omega)$ is a symplectic manifold, and $L^{1},\dots,L^{k}$ are embedded compact Lagrangians in $W$ so that $L^{i}\cap L^{j}$ is an isolated set for each $i\ne j$.

A \emph{degenerating sequence of Lagrangian boundary conditions} for this data is a collection of sequences $L^{1}_{n},\dots,L^{\ell}_{n}$, $n\in \mathbb{N}$ and a function $\pi:\set{1,\dots,\ell}\to \set{1,\dots,k}$ so that each $L^{i}_{n}$ converges to $L^{\pi(i)}$. For example, when $k=1$, all the Lagrangians in the sequence converge to the same Lagrangian $L$.

Given $i\ne j$ so that $\pi(i)=\pi(j)$, the two sequences $L^{i}_{n},L^{j}_{n}$ are said to be of \emph{adiabatic type} relative to their common limit $L$ provided there exists a Weinstein neighbourhood of $L$ so that:
\begin{equation*}
  L^{i}_{n}=\text{graph of $\epsilon_{n}\mathfrak{a}_{n}$, and }
  L^{j}_{n}=\text{graph of $\epsilon_{n}(\mathfrak{a}_{n}+\d f_{n})$},
\end{equation*}
where $\mathfrak{a}_{n}$ is a closed one-form, $\mathfrak{a}_{n}$ converges, $\d f_{n}$ converges to $\d f_{\infty}$ for a \emph{Morse} function $f_{\infty}:L\to \R$, and $\epsilon_{n}\to 0$.

A $C^{\infty}$-convergent sequence of $\omega$-tame complex structures $J_{n}$ is said to be \emph{admissible} provided that $J_{n}|L^{i}=J|L^{i}$ is $\omega$-compatible for each $i=1,\dots,k$. Denote the limit of $J_{n}$ by $J$. The data of the fixed metric $g=\omega(-,J-)$ on each Lagrangian $L_{i}$ induces a gradient vector field for any function $f:L^{i}\to \R$. 

For the next definition, fix the following data: $L^{1}_{n},\dots,L^{\ell}_{n}$ is a degenerating sequence of Lagrangians, with possible limits $L^{1},\dots,L^{k}$ and associated function $\pi$, as above, and suppose that for every $i\ne j$ with $\pi(i)=\pi(j)$, the sequences $L^{i}_{n},L^{j}_{n}$ are of adiabatic type, with limiting Morse function $f^{ij}$.

\begin{defn}\label{defn:generalized_curve}
  A \emph{generalized $J$-holomorphic curve} for the data of $L^{1},\dots,L^{k}$, $\pi$, and $\set{f^{ij}:\pi(i)=\pi(j)}$, is the data of a finite graph\footnote{The graph is allowed to have multiple edges between the same vertices, self-edges, and edges with a free end. Edges with a free end are called exterior.} $\Gamma$ with the following labels:
  \begin{enumerate}
  \item a vertex $v$ of $\Gamma$ is labeled by a punctured holomorphic curve $(u_{v},\Sigma_{v})$, together with an identification of the incident edges of $\Gamma$ with the punctures on the domain.\footnote{Self-edges, which join $v$ to $v$, count twice, i.e., count for two of the punctures on the domain of $v$.} Each connected component of $\bd\Sigma_{v}$ is mapped onto one of the Lagrangians in the collection $L^{1},\dots,L^{k}$. Each holomorphic curve has (continuously) removable singularities at its punctures.
  \item an edge $e$ of $\Gamma$ is labelled by a \emph{type}: which can be either \emph{interior node}, \emph{boundary node}, \emph{intersection}, or \emph{adiabatic}. Each type has a different kind of label. Interior nodes are labeled by a point in $W$, boundary nodes are labeled by a point in $L^{1}\sqcup \dots \sqcup L^{k}$, intersection points are labeled by an unordered pair $i\ne j$ and a point in the finite set $L^{i}\cap L^{j}$, and adiabatic type edges are labeled by an unordered pair $i\ne j$ satisfying $\pi(i)=\pi(i)$ together with a \emph{(potentially broken) Morse flow line} for the function $f^{ij}$. These flow lines lie on the Lagrangian $L^{\pi(i)}=L^{\pi(j)}$, and do not necessarily start or end at critical points.
  \end{enumerate}

  \begin{figure}[H]
  \centering
  \begin{tikzpicture}
    \begin{scope}[shift={(6.5,0)}]

      \draw (1,1) coordinate (Pz) (Pz) circle (0.5) (Pz)+(43:0.5) coordinate(XL) (Pz)+(-40:0.5)coordinate(X) (Pz)+(170:0.5) coordinate(DD);
      \draw[decoration={random steps,segment length=2mm},decorate] (DD)--+(150:0.8) coordinate(DA);
      \draw[decoration={random steps,segment length=2mm},decorate] (DA)--+(110:0.8)node[above]{$x$} (DA)node[draw,circle,black,inner sep=1pt,fill]{}--+(210:0.8)node[below]{$y$};
      
      
      \draw (3,0.4) circle (0.8); \coordinate (Y) at ($(3,0.4)+(-0.8,0)$);
      \coordinate (Z) at ($(3,0.4)+(0.8,0)$);
      \path (3,0.4)--+(60:1.6) coordinate (B);
      \path (B)--+(120:0.4) coordinate (C) --+(0:0.5) coordinate (D); 
      \draw (B) circle (0.2) (C) circle (0.2) (D) circle (0.3);

      \draw[decoration={random steps,segment length=2mm},decorate] (X)--(Y);

      \draw (5,0.3)circle(0.5);
      \draw (5,0.3)+(110:0.5) coordinate(AL) (D)+(-20:0.3) coordinate(BL) (C)+(180:.2) coordinate (YL);
      \coordinate (W) at (4.5,0.3);\coordinate (P) at (5.5,0.3);
      \draw[decoration={random steps,segment length=2mm},decorate] (Z)--(W) (AL)--(BL) (XL)--(YL);

      \node[shift={(0.1,0)},draw,circle,shading=ball,ball color=white,inner sep=3pt] at (D){};

      \node[shift={(0.1,0)},draw,circle,shading=ball,ball color=white,inner sep=5pt] at (3.2,0.1){};
      \node[shift={(0.1,0)},draw,circle,shading=ball,ball color=white,inner sep=3pt] at (3.3,-0.2){};
      \node[shift={(0.1,0)},draw,circle,shading=ball,ball color=white,inner sep=3pt] at (3.5,-0.3){};
      \node[shift={(0.1,0)},draw,circle,shading=ball,ball color=white,inner sep=3pt] at (3.1,-0.3){};

      \node[shift={(0.1,0)},draw,circle,shading=ball,ball color=white,inner sep=4pt] at (1.0,1.1){};

      \node[shift={(0.1,0)},draw,circle,shading=ball,ball color=white,inner sep=4pt] at (4.6,0){};
    \end{scope}
  \end{tikzpicture}
  \caption{A generalized holomorphic curve.} 
  \label{fig:genus2-boundary}
\end{figure}
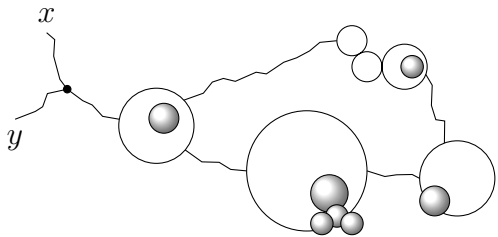
  
  The labels should be following compatible in the obvious ways:
  \begin{enumerate}[label=(\alph*)]
  \item\label{item:requirement-A} For a non-exterior edge $e$, call the two incident vertices' holomorphic curves $u_{0}$ and $u_{1}$ (it is possible that $u_{0}=u_{1}$). If $e$ has interior node type, the requirement is that $e$ corresponds to interior punctures $p_{0},p_{1}$ on the domains of $u_{0},u_{1}$, and
    \begin{equation}\label{eq:incident}
      u_{0}(p_{0})=u_{1}(p_{1})=p(e),
    \end{equation}
    where $p(e)\in W$ is the label corresponding to $e$. For boundary nodes $e$, suppose that $p(e)\in L^{i}$; the edge must correspond to boundary punctures $p_{0},p_{1}$, on $u_{0},u_{1}$, the four incident Lagrangian boundary conditions must be all $L^{i}$, and the incidence relation \eqref{eq:incident} must hold. The requirement is the same when $e$ has intersection point type, except, if $p(e)\in L^{i}\cap L^{j}$, then the incident boundary conditions at $p_{0}$ and $p_{1}$ must be $L^{i}$ and $L^{j}$; moreover, the order in which $L^{i}$ and $L^{j}$ appear is opposite at the two punctures $p_{0},p_{1}$ (i.e., if $L^{i}$ comes before $L^{j}$ at $p_{0}$, then $L^{j}$ comes before $L^{i}$ at $p_{1}$). When $e$ has adiabatic type, corresponding to a pair $i\ne j$ with $\pi(i)=\pi(j)$, then $p_{0}$ and $p_{1}$ are boundary punctures on $u$ with all four incident Lagrangians equal to $L^{\pi(i)}=L^{\pi(j)}$, and $u(p_{0})$ and $u(p_{1})$ are the endpoints of the Morse flow line which $e$ is labeled by. 
  \item For an exterior edge, the labels of $e$ should be consistent with the corresponding holomorphic curve, similarly to \ref{item:requirement-A}. The only non-standard aspect is when $e$ has adiabatic type: in this case, we require that the flow line joins the removable singularity on the holomorphic curve to a critical point for the Morse function. 
  \end{enumerate}
  Examples of such a generalized holomorphic curve is shown in Figures \ref{fig:genus2-boundary} and \ref{fig:convergence_illustration}.
\end{defn}
\begin{figure}[H]
  \centering
  {
    \def\svgwidth{2in}
    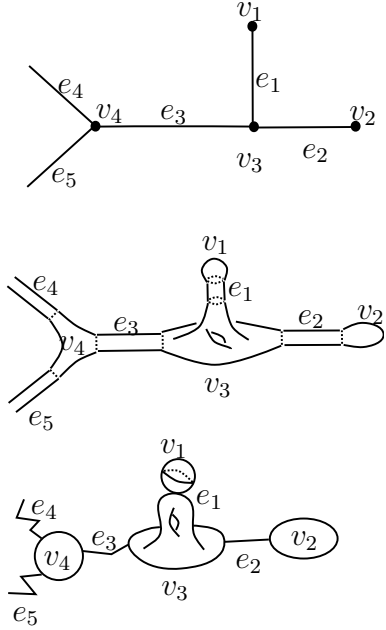
  }     
  \caption{Convergence illustration. Bear in mind that the holomorphic maps may be constant, e.g., $v_{4}$ could be a constant map. For exact Lagrangians, all the holomorphic curves will be constant, and one recovers the gradient flow trees studied in \cite{fukaya-oh,ekholm}}
  \label{fig:convergence_illustration}
\end{figure}

The next definition explains how a sequence of maps $u_{n}$ to converges to a generalized holomorphic curve.
\begin{defn}\label{defn:convergence_defn}
  Let $(u_{n},\Sigma_{n},\bd\Sigma_{n})$ be a sequence of maps defined on punctured Riemann surfaces with boundary $(\Sigma_{n},\bd\Sigma_{n})$, and let $\Gamma$ be a labeled graph describing a generalized holomorphic curve. Say that $u_{n}$ converges to $\Gamma$ provided the following data exists:
  \begin{enumerate}
  \item A decomposition of the underlying domain $(\Sigma_{n},\bd\Sigma_{n})$ into \emph{compact partial domains}, as defined in \S\ref{sec:compact-partial}. Briefly, this involves embedding long necks into the Riemann surface, and cutting the domain along these necks. Let $\set{\Sigma^{a}_{n} \vert a\in A_{n}}$ denote the resulting pieces. 
  \item A bijection between $V(\Gamma)\sqcup E(\Gamma)$ and $A_{n}$; in other words, each piece of the cut-up surface corresponds to either a vertex or edge of $\Gamma$.
  \end{enumerate}
  This data is required to satisfy the following properties:
  \begin{enumerate}[label=(\alph*)]
  \item The bijection between $V(\Gamma)\sqcup E(\Gamma)$ and $A$ should respect incidence relations; i.e., if two pieces $\Sigma^{a}_{n}$ and $\Sigma^{b}_{n}$, $a,b\in A$ are connected in $\Sigma_{n}$, then the corresponding vertex and edge should be attached. In particular, the pieces $\Sigma_{n}^{a}$ are always ``alternating'' between vertex pieces and edge pieces.
  \item Each piece $\Sigma^{a}_{n}$ corresponding to an edge is required to be a strip or cylinder (biholomorphic to $[a_{n},b_{n}]\times S$, where $S=[0,1]$ or $S=\R/\Z$, depending on the type of the edge). The modulus $b_{n}-a_{n}$ is also required to converge to infinity. Moreover, if $e$ is not adiabatic, then $u_{n}|_{\Sigma_{n}^{a}}$ converges uniformly to the point $p(e)$. If $e$ is adiabatic, then $u_{n}$ converges to the broken flow line which $e$ is labeled by, in the sense described in \S\ref{sec:digression-flow-line} and \S\ref{sec:low-energy-strips}.
  \item Each piece $(u_{n}|_{\Sigma_{n}^{a}},\Sigma^{a}_{n})$ corresponding to a vertex $v$ converges uniformly to $(u_{v},\Sigma_{v})$, in the sense explained in \S\ref{sec:compact-partial}. Briefly, uniform convergence asserts the existence of approximately holomorphic embeddings $\psi_{n}:\Sigma_{n}^{a}\to \Sigma_{v}$, which exhaust all of $\Sigma_{v}$ as $n\to\infty$, so that the uniform distance between $u_{v}\circ \psi_{n}$ and $u_{n}\vert_{\Sigma^a_n}$ converges to zero. 
  \end{enumerate}  
\end{defn}

Then main result is:
\begin{theorem}\label{theorem:main_theorem}
  Suppose that $L_{n}^{i}\to L^{\pi(i)}$, $i=1,\dots,\ell$, is a degenerating sequence of Lagrangians in a symplectic manifold $(W,\omega)$, so that whenever $\pi(i)=\pi(j)$, $i\ne j$, the pair $L^{i}_{n},L^{j}_{n}$ are of adiabatic type with limiting functions $f^{ij}$. Let $J_{n}\to J$ be an admissible sequence of $\omega$-tame complex structures.

  If $(u_{n},\Sigma_{n},\bd\Sigma_{n})$ is a sequence of $J_{n}$-holomorphic curves defined on boundary punctured domains which satisfies:
  \begin{enumerate}
  \item $u_{n}(\Sigma_{n})\subset K$ for some compact set $K\subset W$,
  \item $\sup_{n}\int_{n} u_{n}^{*}\omega<\infty$, 
  \item the domains $\Sigma_{n}$ have bounded topology (i.e., the number of punctures and Euler characteristic is bounded), and
  \item each component of $\bd\Sigma_{n}$ is mapped by $u_{n}$ onto one of the Lagrangians $L^{i}_{n}$,
  \end{enumerate}
  then, after passing to a subsequence, $u_{n}$ converges to a generalized holomorphic curve with boundary for the data $(L^{i},\pi, f^{ij})$.
\end{theorem}

\subsubsection*{Outline of paper}
\label{sec:outline}
The overall strategy of the paper is rather straightforward: prove that low energy holomorphic strips with adiabatic boundary conditions converge uniformly to broken Morse flow lines; paying close attention to the endpoints of the interval. Roughly speaking, the convergence of a holomorphic strip to a flow line is only seen after rescaling, i.e., the reparametrization $v_{n}(s,t)=u_{n}(s/\epsilon_{n},t/\epsilon_{n})$ is what converges. Since $\epsilon_{n}\to 0$, gradient bounds on $u_{n}$ do not imply gradient bounds on $v_{n}$ and hence we cannot apply Arzel\`a-Ascoli to $v_{n}$. The technical core of the paper, \S\ref{sec:exponential-estimates}, concerns a delicate exponential decay estimate needed to bound the derivatives of $v_{n}$; the hardest place to control $v_{n}$ is at the ends of the strip, which is where the limit flow line will connect to the holomorphic curves in the limiting configuration.

In \S\ref{sec:lccoordinates}, we review the Levi-Civita coordinate system for cotangent bundles, and in \S\ref{sec:apriori} we review the basic a priori estimates for holomorphic curves; both sections can be safely skipped by any experts.

In \S\ref{sec:compact-partial}, we give a fairly detailed account of the compactness theory for Riemann surfaces, essentially to make precise the informal idea that, in order to analyze the compactness phenomena of holomorphic curves, one only needs to understand the behaviour of strips/cylinders (necks) with low energy and large modulus. These low energy regions are analyzed in \S\ref{sec:low-energy-regions}. 

\subsubsection*{Acknowledgements}
\label{sec:ack}
This work was completed while the authors were graduate students at Stanford University. Both authors benefited greatly from interactions with Yasha Eliashberg, Eleny Ionel, Umut Varolgunes, and the other graduate students in our field. The authors also wish to thank Ke Zhu and Octav Cornea for useful comments during the preparation of the text.

\subsection{Floer's 1989 paper}
\label{sec:floers-argument}
The analytical arguments in this paper are inspired by Floer's paper \cite{FloerW}. Floer's argument can be summarized as follows: consider a single Morse function $f$ and holomorphic strips $u:\R\times [0,1]\to T^{*}L$ for a compact Lagrangian $L$ with $u(\R\times \set{0})\in L_{0}$ and $u(\R\times \set{1})\in L_{f}$. Here $L_{f}$ is the graph of $\d f$. The map $u$ is asymptotic to intersection points $x$ and $y$ between $L$ and $L_{f}$, as in Figure \ref{fig:FloerStrip}.
\begin{figure}[H]
  \centering
  \begin{tikzpicture}
    \draw (0,1)--(5,1);
    \draw (0,0)--(5,0);
    \draw[->] (2.5,-0.1)--(2.5,1.5) node [right] {$t$};
    \draw[->] (1.9,0.5)--(3.5,0.5) node [right] {$s$};
    \draw[->] (5.5,0.5)--node[above]{$u$}(6.5,0.5);
    \begin{scope}[shift={(9,0.3)}]    
      \path[name path={P1}] (-2,0.2)--(2,0.2);
      \path[name path={P2}] (-2,-0.2)to[out=45,in=135](2,-.2);    
      \path[name intersections={of=P1 and P2}];
      \begin{scope}
        \clip (intersection-1) rectangle ($(intersection-2)+(0,0.6)$);
        \fill[pattern={Lines[angle=90]},pattern color=black!50!white] (-2,-0.2)to[out=45,in=135](2,-.2)--cycle;
      \end{scope}
      \draw (-2,0.2)--node[below]{$L_{0}$}(2,0.2) ;
      \draw (-2,-0.2)to[out=45,in=135]node[pos=0.5,above]{$L_{f}$}(2,-.2);

      \node[draw,circle,fill=red,inner sep=1pt] (A)at (intersection-1) {};
      \node[draw,circle,fill=green!60!black,inner sep=1pt] (B)at (intersection-2) {};    
      \node at (A)[below]{$x$};
      \node at (B)[below]{$y$};
    \end{scope}
  \end{tikzpicture}
  \caption{A holomorphic strip $u$ in the cotangent bundle asymptotic to intersection points $x,y$. The lower boundary of $u$ maps to the zero section $L_{0}=L$, and the upper boundary of $u$ maps to the Lagrangian section $L_{f}$ (the graph of $\d f$). The intersection points $x,y$ are critical points of $f$.}
  \label{fig:FloerStrip}
\end{figure}
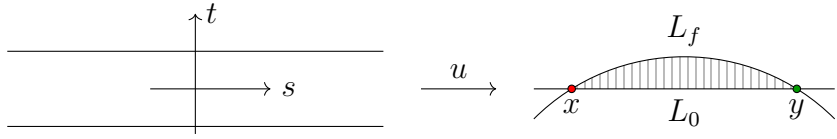
Floer considered maps $u:\R\times [0,1]\to T^{*}L$ which were holomorphic for a particular $t$-dependent almost complex structure $J_{f}$ associated with $f$. This complex structure $J_{f}$ is defined by conjugating a fixed complex structure $J_{0}$ with the flow $\varphi_{f}$ generated by the Hamiltonian vector field $X_{f}$:\footnote{In words, $J_{f,t}(u)$ is defined by first flowing the tangent space $T(T^{*}L)_{u}$ backwards by time $t$, then applying $J_{0}$, then flowing back up. Our convention for the Hamiltonian vector field $X_{f}$ is that it satisfies the equation $X_{f}\intprod \omega:=\omega(X_{f},-)=-\pr^{*}\d f$, where $\pr:T^{*}L\to L$ denotes the projection.}
\begin{equation}\label{eq:floersJf}
  J_{f,t}(u)=\d\varphi_{t}(\varphi_{-t}(u))J_{0}(\varphi_{-t}(u))\d \varphi_{-t}(u).
\end{equation}

Floer took $J_{0}$ to be the \emph{Levi-Civita complex structure} associated to an arbitrary Riemannian metric $g$ on $L$ (see \S \ref{sec:J0} for the definition of $J_{0}$).

A straightforward computation showed that if $u$ is $J_{f}$ holomorphic, in the sense that $\bd_{s}u+J_{f}(t,u)\bd_{t}u=0$, then the modified map $\tilde{u}(s,t)=\varphi_{-t}(u(s,t))$ satisfies:
\begin{equation}\label{eq:floersholoequation}
  \bd_{s}\tilde{u}+J_{0}(\tilde{u})\bd_{t}\tilde{u}+J_{0}(\tilde{u})X_{f}(\tilde{u})=0.
\end{equation}
Notice that, by construction, this modified map has both boundary components lying on the zero section. 

Floer then analyzed $\tilde{u}$ by considering the integral quantity $\gamma(s)$ defined by
\begin{equation}\label{eq:floers-integral-quantity}
  \gamma(s)=\int_{0}^{1}g(\tilde{u}(s,t),\tilde{u}(s,t))\d t
\end{equation}
Notice that $\tilde{u}(s,t)$ lives in \emph{some} fiber of the vector bundle $T^{*}L$, and so it makes sense to insert $\tilde{u}$ into the metric $g$.

Using \eqref{eq:floersholoequation}, Floer showed that $\gamma$ satisfies the differential inequality
\begin{equation}
  \label{eq:differential-inequality-1}
  \gamma^{\prime\prime}-\delta^{2}\gamma\ge 0,
\end{equation}
for some constant $\delta>0$, provided that the $C^{2}$ size of the Morse function $f$ is sufficiently small. In particular, $\gamma^{\prime\prime}$ is always non-negative, and hence $\gamma$ cannot attain a positive maximum. 
The asymptotic behaviour of $\gamma$ implies it is identically zero and so $\tilde{u}$ lies entirely in the zero section.

Having established this, we now reconsider the differential equation \eqref{eq:floersholoequation}. Observe that $X_{f}(\tilde{u})$ points in the $T^{*}L$ directions and $\bd_{s}\tilde{u},\bd_{t}\tilde{u}$ both point in the $TL$ directions with respect to the decomposition of $T(T^{*}L)|_{L}=TL\oplus T^{*}L$. One of the properties of the Levi-Civita complex structure $J_{0}$ is that $J_{0}(T^{*}L)=TL$ (this is built into the definition of $J_{0}$). Therefore both sides of the equation
\begin{equation*}
  \bd_{s}\tilde{u}+J_{0}(\tilde{u})X_{f}(\tilde{u})=-J_{0}(\tilde{u})\bd_{t}\tilde{u}
\end{equation*}
are zero, as the left hand side is valued in $TL$ while the right side is valued in $T^{*}L$. Therefore $\tilde{u}(s,t)=Q(s)$ where $Q:\R\to L$ is a smooth curve.

Another property of the Levi-Civita complex structure $J_{0}$ is that $J_{0}(\tilde{u})X_{f}(\tilde{u})$ is the negative gradient of $f$ with respect to the chosen metric $g$. This follows from the equation $X_{f}\intprod \omega=-\pr^{*}\d f$ defining $X_{f}$, and the \emph{almost K\"ahler} relationship $\omega(-,J_{0}-)=g$ established in \S \ref{sec:compatibility}). In particular we conclude that
\begin{equation*}
  \bd_{s}Q=\text{gradient of $f$},
\end{equation*}
i.e., $s\mapsto Q(s)$ is a gradient flow line for the function $f$. 

We return to the original holomorphic map $u$. As $\tilde{u}(s,t)=\varphi_{-t}(u(s,t))$ we conclude that $u(s,t)=\varphi_{t}(\tilde{u}(s,t))=\varphi_{t}(Q(s))$.

We can summarize the conclusion of Floer's argument in the following theorem:
\begin{theorem}[Floer's Theorem]\label{theorem:floer}
  For every Riemannian metric $g$ on a compact Lagrangian $L$, there is a constant $\epsilon_{0}>0$ with the following property: if $f$ is a Morse function such that $f,\d f$ and $\nabla \d f$ are all less than $\epsilon_{0}$ and $u:\R\times [0,1]\to T^{*}L$ is a $J_{f}$-holomorphic curve with boundary conditions as in Figure \ref{fig:FloerStrip}, then
  \begin{equation}\label{eq:exact-formula}
    u(s,t)=\varphi_{t}(Q(s))
  \end{equation}
  where $Q(s)$ is a gradient flow line for $f$ and $\varphi_{t}$ is the flow of the Hamiltonian vector field for $f$.
\end{theorem}
\begin{remark}
  One idea which motivated this paper is that, in order to analyze more general sequences $u_{n}$, we should understand (i) what happens for general almost complex structures $J$ (or at least complex structures which are allowed to be independent of $n$ and $t$) and (ii) what happens if we have a sequence of finite-length strips $[-R_{n},R_{n}]\times [0,1]$ (necks) or half-infinte strips $[0,\infty)\times [-1,1]$ (ends), rather than only infinite strips.
\end{remark}

\section{The Levi-Civita coordinate system near a Lagrangian}
\label{sec:lccoordinates}

In this section, we define a special class of coordinate systems on tubular neighbourhoods of compact Lagrangian submanifolds $L$ called \emph{Levi-Civita coordinates}. We explain how a metric on $L$ induces an almost K\"ahler structure on $T^{*}L$ extending the canonical symplectic form. We prove certain standard PDE estimates for holomorphic curves using these coordiantes. Similar coordinate systems and results can be found in \cite{FloerW,oh_relative,fukaya-oh,ekholm}.

\subsection{Levi-Civita connection and the cotangent bundle}
Consider a smooth manifold $L$ as the Lagrangian zero section inside its cotangent bundle $\pr:T^{*}L\to L$. Given a Riemannian metric $g$ on $L$, we obtain a Levi-Civita connection $\nabla$ on $TL$, uniquely determined by the properties that $\nabla$ is torsion-free and $g$-compatible. This connection determines a coordinate system and almost complex structure on $T^{*}L$ which are well-adapted to computations involving holomorphic curves.

\subsubsection{The splitting of $T^{*}L$ induced by the connection}
\label{sec:splitting}
The Levi-Civita connection on $TL$ induces a connection on $T^{*}L$ (also denoted by $\nabla$) uniquely determined by the property that for every pair of sections $\mu,X$ of $T^{*}L$ and $TL$, respectively, we have
\begin{equation*}
  \d\ip{\mu,X}=\ip{\nabla \mu,X}+\ip{\mu,\nabla X}.
\end{equation*}
This connection on $T^{*}L$ induces a splitting of the short-exact sequence
\begin{equation}\label{eq:connection-split}
  \begin{tikzcd}
    {0}\arrow[r,"{}"] &{\pr^{*}T^{*}L}\arrow[r,"{}"] &{T(T^{*}L)}\arrow[r,"{\d\pr}"] &{\pr^{*}TL}\arrow[l,out=230,in=-50]\arrow[r,"{}"] &{0}
  \end{tikzcd}
\end{equation}
where the summand $\pr^{*}TL\subset T(T^{*}L)$ is called the \emph{horizontal distribution}. A section $\sigma$ passing through $(p,v)\in T^{*}L$ is tangent to the horizontal distribution at $(p,v)$ if and only if
\begin{equation*}
  \nabla \sigma(p)=0;
\end{equation*}
this property uniquely determines the splitting of \eqref{eq:connection-split}. We denote by $$\Pi:T(T^{*}L)\to \pr^{*}T^{*}L$$ the induced projection onto the vertical sub-bundle (i.e., $\Pi$ is the unique projection which vanishes on the horizontal distribution).

\begin{remark}
  Let $E\to L$ be a vector bundle over a smooth manifold $L$, let $\nabla$ be a linear connection on $E$, and let $\Pi$ be the induced vertical projection.
  
  Then $\nabla$ can be recovered from the vertical projection $\Pi$ by the formula
  \begin{equation}\label{eq:nablaproj}
    \nabla s = \Pi(s)\d s
  \end{equation}
  for all sections $s:L\to E$.
\end{remark}

\subsubsection{The induced Riemannian metric on $T(T^{*}L)$}
\label{sec:metricontotalspace}
Let $g_{*}:TL\to T^{*}L$ be the linear isomorphism:
\begin{equation}\label{eq:giso}
  g_{*}:X\mapsto g(X,-).
\end{equation}
This induces a metric $g$ on $T^{*}L$ uniquely determined by the property that \eqref{eq:giso} is a linear isometry. Define a metric $g$ on $T(T^{*}L)$ by the property:
\begin{equation*}
  g=\begin{dmatrix}
    {g}&{0}\\
    {0}&{g}
  \end{dmatrix}\text{ with respect to the splitting }T(T^{*}L)=\pr^{*}(TL)\oplus \pr^{*}(T^{*}L).
\end{equation*}

\subsection{The Levi-Civita complex structure}
\label{sec:J0}
Introduce the complex structure $J_{0}$ on $T(T^{*}L)$ by the requirement\footnote{Here the $\cdot$ notation appearing in $\Pi\cdot J_0$ signifies the bilinear composition of linear homomorphisms (i.e., $J_0$ and $\Pi$ are both sections of homomorphism bundles, and hence can be composed). We will use the $\cdot$ notation throughout the paper to indicate various bilinear operations.}
\begin{equation}\label{eq:cx1}
  \Pi\cdot J_{0}=g_{*}\d \pr.
\end{equation}
The other component $\d\pr J_{0}$ is determined by the constraint that $J_{0}^{2}=0$.

\subsubsection{Compatibility with the symplectic form}
\label{sec:compatibility}
The metric from \S\ref{sec:metricontotalspace}, the Levi-Civita almost complex structure from \S\ref{sec:J0}, and the canonical symplectic structure on the cotangent bundle always form an \emph{almost K\"ahler triple} in the sense that $\omega(-,J_{0}-)=g$. The verification of this fact is left to the reader. The key step is showing that the horizontal distribution in $T^{*}L$ determined by the Levi-Civita connection is a Lagrangian distribution, which is a consequence of the symmetry of the connection.

\subsubsection{Decomposing a map into horizontal and vertical components}
\label{sec:PandQ}
Let $\Sigma$ be a smooth manifold, and let $u:\Sigma\to T^{*}L$ be a smooth map. We define $Q:\Sigma\to L$ as the projection $Q=\pr(u)$. Then $u$ can be recovered by the section $P\in \Gamma(Q^{*}T^{*}L)$ defined so that $P(z)=u(z)$. We call $Q$ the \emph{horizontal component} of $u$, and $P$ the \emph{vertical component} of $u$.


The derivatives of $u$ can be expressed in terms of the derivatives of $Q$ and the covariant derivatives of $P$. In order to state the precise relationship, recall the definition of the pullback connection:
\begin{defn}\label{defn:pbc}
  Given any vector bundle $E\to L$ with a connection $\nabla$ and any smooth map $Q:\Sigma\to L$, the \emph{pullback connection} on $Q^{*}E$, also denoted $\nabla$, is the unique linear connection which satisfies the following property: if $\sigma:L\to E$ is a section, and $s(z)=\sigma(Q(z))$ is the induced ``pulled back'' section of $Q^{*}E$, then
  \begin{equation*}
    \nabla s(z)\cdot v=\nabla\sigma(Q(z))\cdot \d Q_z(v),
  \end{equation*}
  for all $z\in \Sigma$, $v\in T\Sigma_{z}$. Note that some may write $\nabla s(z)\cdot v$ as $\nabla_{v}s(z)$.
\end{defn}

\begin{lemma}[Derivatives of $u$ in terms of $Q$ and $P$]\label{lemma:horvertp}
  Let $u:\Sigma\to T^{*}L$ be a smooth map. Considering $\d u$ as a section of $$\mathrm{Hom}(T\Sigma,u^{*}T(T^{*}L))\simeq \mathrm{Hom}(T\Sigma,u^{*}\pr^{*}TL\oplus u^{*}\pr^{*}T^{*}L),$$ we have
  \begin{equation}\label{eq:horvertp}
    \d \pr\cdot \d u=\d Q\text{ and }\Pi(u)\cdot \d u=\nabla P,
  \end{equation}
  where $\nabla P$ is computed using the pullback connection defined in Definition \ref{defn:pbc}.
\end{lemma}
The proof of Lemma \ref{lemma:horvertp} is an exercise in differential calculus and is left to the reader.

\subsubsection{The holomorphic curve equation in Levi-Civita coordinates}
\label{sec:holocurveeq}
Let $\Sigma$ be a Riemann surface with complex structure $j$. Combining \eqref{eq:horvertp} and \eqref{eq:cx1} shows that a map $u:\Sigma\to T^{*}L$ is $J_{0}$-holomorphic if and only if
\begin{equation*}
  g_{*}\d Q=\nabla P\cdot j.
\end{equation*}
In particular, if $\Sigma$ carries holomorphic coordinates $s+it$, then $u$ is $J_{0}$-holomorphic if and only if
\begin{equation}\label{eq:CReqs}
  g_{*}\pd{Q}{s}=\nabla_{t}P\text{ and }g_{*}\pd{Q}{t}=-\nabla_{s}P.
\end{equation}

\subsubsection{Interchanging derivatives}
\label{sec:interchange}
The following lemma will play a key role in later computations.
\begin{lemma}\label{lemma:interchange}
  Let $\Sigma$ be a Riemann surface with holomorphic coordinates $s+it$. Let $Q:\Sigma\to L$ be a smooth map. Then
  \begin{equation}\label{eq:interchange}
    \nabla_{s}(g_{*}\pd{Q}{t})=\nabla_{t}(g_{*}\pd{Q}{s}).
  \end{equation}  
\end{lemma}
\begin{proof}
  This follows easily from the axioms for the Levi-Civita connection, and is left to the reader.
\end{proof}

\subsubsection{The energy density in Levi-Civita coordinates}
\label{sec:energy-density}
Let $u:\Sigma\to (W,\omega)$ be a smooth map valued in a symplectic manifold. We define the \emph{$\omega$-energy} of $u$ to be the integral of $\omega$:
\begin{equation*}
  \text{$\omega$-energy of $u$}=\int_{\Sigma}u^{*}\omega.
\end{equation*}
If $(\omega,J,g)$ is an almost K\"ahler triple, in the sense that $\omega(-,J-)=g$ is a Riemannian metric, $u:\Sigma\to W$ is $J$-holomorphic, and $s+it$ are holomorphic coordinates on $\Sigma$, then it is straightforward to see that
\begin{equation}\label{eq:omega-energy}
  \text{$\omega$-energy of $u$}=\int_{\Sigma}\abs{\bd_{s}u}^{2}\d s\wedge \d t=\int_{\Sigma}\abs{\bd_{t}u}^{2}\d s\wedge \d t.
\end{equation}
where the norms $|-|$ are measured using $g$. More generally, if $(W,J,g)$ is an almost-complex manifold with a Riemannian metric $g$, then we define the \emph{$g$-energy} of a holomorphic curve $u:\Sigma\to W$ to be the integral
\begin{equation*}
  \text{$g$-energy of $u$}=\int_{\Sigma}\abs{\bd_{s}u}^{2}\d s\wedge \d t=\int_{\Sigma}\abs{J\bd_{t}u}^{2}\d s\wedge \d t.
\end{equation*}
In this setting, we call the function $\abs{\bd_{s}u}^{2}$ the $g$-\emph{energy density}. If $J$ acts by isometries of $J$, then we can replace $\abs{\bd_{s}u}$ by $\abs{\bd_{t}u}$ in the above integrand.

Note that \eqref{eq:omega-energy} states that the $\omega$-energy equals the $g$-energy when we use an almost K\"ahler triple $(\omega,J,g)$; in this case we will just call the quantity in \eqref{eq:omega-energy} the \emph{energy}, denoted by $E(u)$.

Suppose that $u:\Sigma\to T^{*}L$ is a $J_{0}$-holomorphic curve, and $s+it$ are holomorphic coordinates on $\Sigma$. In \S\ref{sec:compatibility} we established that $(\omega,J_{0},g)$ formed an almost K\"ahler triple. We can use the orthogonal splitting of $T(T^{*}L)$ to express the energy density of $u$ in terms of its $Q,P$ components:
\begin{equation*}
  \abs{\bd_{s}{u}}^{2}=\abs{\d\pr(\bd_{s}u)}^{2}+\abs{\Pi (\bd_{s}u)}^{2}=\abs{\bd_{s}Q}^{2}+\abs{\nabla_{s}P}^{2}.
\end{equation*}

\subsubsection{Comparison with other complex structures}
\label{sec:other-complex-structures}
Let $J$ be a complex structure on $T^{*}L$ so that $J$ agrees with $J_{0}$ along $L$, where $J_{0}$ is the Levi-Civita almost complex structure for some Riemannian metric $g$. The first goal in this section is to write down the equation for a map $u$ being $J$-holomorphic in terms of the $(Q,P)$ coordinates of $u$. 
\begin{remark}
  Note that the condition that $\omega(-,J-)$ is a Riemannian metric on $TW|_{L}$ says that $J$ is $\omega$-\emph{compatible} along $L$. Clearly $\omega$-compatibility along $L$ is a necessary condition in order for $J$ to agree with $J_{0}$ along $L$, since $J_{0}$ is known to be $\omega$-compatible.
\end{remark}
We can generalize this slightly and suppose that $J_{n}\to J$ is a convergent sequence of complex structures so that $J_{n}|_{L}=J|_{L}$ is constant. The more general scenario where $J_{n}|_{L}$ is allowed to vary is more complicated, as our methods would require $n$-dependent families of Weinstein neighbourhoods $\iota_{n}$. It is highly likely that our methods work without too much additional difficulty to include the case where $J_{n}|_{L}$ is non-constant, although we do not pursue this in this paper.

Suppose that $\Omega$ be a domain with holomorphic coordinates $s+it$, and $u:\Omega\to T^{*}L$ satisfies $\bd_{s}u+J_{n}(u)\bd_{t}u=0,$ and $J_{n}|_{L}=J_{0}|_{L}$ for all $n$. Our goal is to express this equation as a perturbation of the $J_{0}$-holomorphic curve equation. Rearranging yields
\begin{equation*}
  \bd_{s}u+J(u)\bd_{t}u=0\implies \bd_{s}u+J_{0}(u)\bd_{t}u=(J_{n}(u)-J_{0}(u))\bd_{t}u.
\end{equation*}
Applying the two projections $g_{*}\d\pr$ and $\Pi$:
\begin{equation*}
  \begin{aligned}
    g_{*}\bd_{s}Q-\nabla_{t} P&=g_{*}\d\pr\cdot (J_{n}(u)-J_{0}(u))\bd_{t}u\\
    \nabla_{s}P+g_{*}\bd_{t} Q&=\Pi\cdot (J_{n}(u)-J_{0}(u))\bd_{t}u,
  \end{aligned}
\end{equation*}
using the fact that $\Pi \cdot J_{0}=g_{*}\d\pr$ (see \eqref{eq:cx1}). Since $J_{n}=J_{0}$ along $L$ there are smooth sections $u\mapsto A_{n}(u),B_{n}(u)$ so that\footnote{To be precise, $A_{n},B_{n}$ are smooth sections of $\mathrm{Hom}(\pr^{*}T^{*}L\otimes T(T^{*}L),\pr^{*}T^{*}L))$, which is a bundle over $T^{*}L$.}
\begin{equation*}
  \begin{aligned}
    g_{*}\d\pr\cdot (J_{n}(u)-J_{0}(u))\bd_{t}u&=A_{n}(u)\cdot P\cdot \bd_{t}u\\
    \Pi\cdot (J_{n}(u)-J_{0}(u))\bd_{t}u&=B_{n}(u)\cdot P\cdot \bd_{t}u.
  \end{aligned}
\end{equation*}
Moreover, $A_{n}$ and $B_{n}$ converge to maps $A$, $B$. As a consequence the $Q,P$ coordinates of a $J$-holomorphic curve $u$ satisfy the following system of equations:
\begin{equation}\label{eq:perturbJ}
  \begin{aligned}
    g_{*}\bd_{s}Q-\nabla_{t} P&=A_{n}(u)\cdot P\cdot \bd_{t}u\\
    \nabla_{s}P+g_{*}\bd_{t} Q&=B_{n}(u)\cdot P\cdot \bd_{t}u.
  \end{aligned}  
\end{equation}
Consider \eqref{eq:perturbJ} as a perturbation of \eqref{eq:CReqs}.

\subsubsection{Tubular neighborhoods adapted to a choice of complex structure}
\label{sec:tubnub_thm}

The next lemma shows that the condition that $J_{n}$ agrees with $J_{0}$ along $L$ is not very restrictive. Indeed, it can always be achieved by the correct choice of Weinstein neighbourhood and Riemannian metric $g$.
\begin{restatable}{lemma}{tubnub}\label{lemma:tubnub}
  Let $L$ be a closed Lagrangian in a symplectic manifold $(W,\omega)$. Suppose that $J|_{L}\in \mathrm{End}(TW|_{L})$ is an $\omega$-compatible almost complex structure along $L$. Then there is an tubular neighbourhood: $$\iota:T^{*}L\supset N\to W$$ so that $\iota^{*}\omega=-\d\lambda_{\mathrm{can}}$ and $\d\iota^{-1}J\d\iota$ agrees with $J_{0}$ along $L$ for the metric $g=\omega(-,J-)$ restricted to $L$.
\end{restatable}
The proof is a straightforward application of the Moser deformation argument, and is left to the reader.

\subsection{The Poincar\'e inequality for arcs of $L$.}
\label{sec:poincare}
Recall that an arc of $L$ is a map $\gamma:[0,1]\to T^{*}L$ with both boundary points on $L$. Consider the horizontal $Q=\pr(\gamma)$ and vertical $P\in \Gamma(Q^{*}T^{*}L)$ components. Then $P$ is a section which vanishes at both endpoints.
\begin{figure}[H]
  \centering
  \begin{tikzpicture}
    \draw[->] (-3,0)--(3,0) node[right]{$Q$};
    \draw[->] (0,-0.5)--(0,2) node[right]{$P$};
    \draw[line width=1pt,postaction={decorate,decoration={
        markings,
        mark=between positions 0.1 and .9 step 0.2 with {\arrow{>};},
      },
    }] plot[smooth,tension=1.6] coordinates {(-1,0) (-.8,1.2) (0,0.5) (0.8,0.9) (1,0)};
    \node at (0.5,0.9)[above right]{$\gamma$};
    \node at (-1,0) [below] {$x$};
  \end{tikzpicture}
  \caption{An arc of $L$ inside of $T^{*}L$.}
  \label{fig:1}
\end{figure}
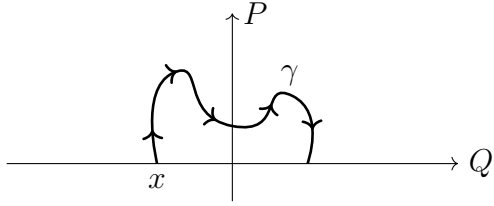

\begin{lemma}[The Poincar\'e inequality]\label{lemma:poincare}
  Let $\gamma$ be an arc of $L$, and let $Q(t),P(t)$ be its horizontal and vertical components. There is a universal constant $c_{\mathrm{pc}}>0$ with the following property:
  \begin{equation*}
    \int_{0}^{1} \abs{P(t)}^{2}\d t\le c_{\mathrm{pc}}\int_{0}^{1}\abs{\nabla_{t}P(t)}^{2}\d t,
  \end{equation*}
  where $\abs{-}$ is measured with respect to the Levi-Civita metric on $T^{*}L$.
\end{lemma}
\begin{proof}
  Let $x=Q(0)$ denote the starting basepoint of $\gamma$. Pick an orthonormal frame $F$ for the fiber $(T^{*}L)_{x}$, and extend $F$ to an frame $F(t)$ for the fiber over $Q(t)$ by the requirement that $\nabla_{t} F(t)=0$. 

  Write $P(t)=\sum_{i=1}^{n}p_{i}(t)F_{i}(t)$ where $F_{i}(t)$ are the basis vectors of $F(t)$. Then:
  \begin{equation*}
    \abs{P(t)}^{2}=\sum_{i=1}^{n} \abs{p_{i}(t)}^{2},\text{ and }\abs{\nabla_{t}P(t)}^{2}=\sum_{i=1}^{n} \abs{\pd{p_{i}(t)}{t}}^{2}.
  \end{equation*}
  In particular, if $c_{\mathrm{pc}}$ is a constant with the property that:
  \begin{equation}\label{eq:normalpc}
    \int_{0}^{1}\abs{p_{i}(t)}^{2}\d t\le c_{\mathrm{pc}}\int_{0}^{1}\abs{\pd{p_{i}(t)}{t}}^{2}\d t,
  \end{equation}
  for all $p_{i}$ satisfying $p_{i}(0)=p_{i}(1)=0$, then the lemma will hold with the same constant $c_{\mathrm{pc}}$. However \eqref{eq:normalpc} only involves $\R$-valued functions and the usual absolute value function. It is a standard fact of real analysis that there is some constant $c_{\mathrm{pc}}$ satisfying \eqref{eq:normalpc} (one can take $c_{\mathrm{pc}}=1/4$). This completes the proof.
\end{proof}

\subsubsection{Elliptic estimates for the Levi-Civita Laplacian}
\label{sec:elliptic-estimates}
Let $\Omega(r)=[r,r]\times [0,1]$, and let $Q:\Omega(r)\to L$ and let $P\in \Gamma(Q^{*}T^{*}L)$. Define the \emph{Levi-Civita Laplacian} of $P$ by the formula:
\begin{equation*}
  \Delta^{\mathrm{lc}}P:=\nabla_{s}\nabla_{s}P+\nabla_{t}\nabla_{t}P,
\end{equation*}
where $\nabla$ is the pullback connection in Definition \ref{defn:pbc}. Our goal in this section is to prove the following elliptic estimate:
\begin{lemma}\label{lemma:elliptic estimate}
  Pick $r>0$, $\delta>0$, $k\ge 1$ and $N> 0$. There exist $c=c(r,\delta)>0$ and $C_{k}=C(r,\delta,k,N)>0$ such that, if $Q:\Omega(r+\delta)\to L$ satisfies:
  \begin{equation*}
    \abs{\bd_{s}Q}+\abs{\bd_{t}Q}\le c\text{ and }\sum_{\ell=0}^{k}\abs{\nabla^{\ell}\bd_{s}Q}+\sum_{\ell=0}^{k}\abs{\nabla^{\ell}\bd_{t}Q}\le N,
  \end{equation*}
  and $P\in \Gamma(Q^{*}T^{*}L)$ satisfies $P(s,0)=P(s,1)=0$ then:
  \begin{equation}\label{eq:elliptic-estimate}
    \norm{P}_{W^{k+1,2}(\Omega(r))}\le C_{k}(\norm{\Delta^{\mathrm{lc}}P}_{W^{k-1,2}(\Omega(r+\delta))}+\norm{P}_{W^{k,2}(\Omega(r+\delta))}).
  \end{equation}
\end{lemma}
\begin{proof}
  The proof bears some similarities with our proof of the Poincar\'e inequality (Lemma \ref{lemma:poincare}); we express $P$ in terms of a travelling frame: $$P(s,t)=\sum_{i=1}^{n} p_{i}(s,t)F_{i}(s,t)\in \Gamma^{*}(Q^{*}T^{*}L).$$
  On this two-dimensional domain, one cannot expect to be able to find a parallel frame $F_{i}(s,t)$ due to curvature obstructions. For this proof, we prefer to use \emph{coordinate frames}, i.e., $F_{i}=\d x_{i}$ where $x_{1},\cdots,x_{d}$ form a coordinate chart.

  Pick finitely many coordinate charts $x:\cl{U}\to \cl{B}(1)^{n}$ whose interiors $U$ cover $L$. Let $\lambda>0$ be a Lebesgue number for the open cover induced by these charts.

  By choosing $c$ sufficiently small we can ensure that $\mathrm{diam}(Q(\Omega(r+\delta)))<\lambda$, and hence the image of $Q$ lies inside one of our finitely many coordinate charts, say $x:U\to B(1)$.

  Therefore we may write:
  \begin{equation*}
    P=\sum_{i=1}^{n}p_{i} \d x_{i}\circ Q.
  \end{equation*}
  We directly compute the derivatives of $P$. Recalling the definition of the Hessian: $$\mathrm{H}(f):=\nabla \d f\in \Gamma(Q^{*}\mathrm{Hom}(TL,T^{*}L)),$$  
  and the definition of the pull-back connection, we conclude that:
  \begin{equation*}
    \begin{aligned}
      \nabla_{s} P&=\sum_{i=1}^{n}\pd{p_{i}}{s} \d x_{i}\circ Q+\sum_{i=1}^{n}p_{i} \cdot(\mathrm{H}(x_{i})\circ Q)\pd{Q}{s},\\
      \nabla_{t} P&=\sum_{i=1}^{n}\pd{p_{i}}{t}\d x_{i}\circ Q+\sum_{i=1}^{n}p_{i} \cdot(\mathrm{H}(x_{i})\circ Q)\pd{Q}{t}.
    \end{aligned}
  \end{equation*}
  Differentiating once again, we obtain the impressive looking, but rather straightforward, formulas for the second derivatives appearing in $\Delta^{\mathrm{lc}}$.
  \begin{equation}\label{eq:2nd-derivs-lc}
    \begin{aligned}
      \nabla_{s}\nabla_{s}P&=\sum_{i=1}^{n}\pd{^{2}p_{i}}{s^{2}}\cdot \d x_{i}\circ Q+2\sum_{i=1}^{n}\pd{p_{i}}{s}\cdot (\mathrm{H}(x_{i})\circ Q)\cdot \pd{Q}{s}\\ &\hspace{10pt}+\sum_{i=1}^{n}p_{i} \cdot ((\nabla \mathrm{H}(x_{i}))\circ Q)\cdot \pd{Q}{s}\cdot \pd{Q}{s}+\sum_{i=1}^{n}p_{i} \cdot (\mathrm{H}(x_{i})\circ Q)\cdot \nabla_{s}\pd{Q}{s}.\\
      \nabla_{t}\nabla_{t}P&=\sum_{i=1}^{n}\pd{^{2}p_{i}}{t^{2}}\cdot \d x_{i}\circ Q+2\sum_{i=1}^{n}\pd{p_{i}}{t}\cdot (H(x_{i})\circ Q)\cdot \pd{Q}{t}\\ &\hspace{10pt}+\sum_{i=1}^{n}p_{i} \cdot ((\nabla \mathrm{H}(x_{i}))\circ Q)\cdot \pd{Q}{t}\cdot \pd{Q}{t}+\sum_{i=1}^{n}p_{i} \cdot (\mathrm{H}(x_{i})\circ Q)\cdot \nabla_{t}\pd{Q}{t}.
    \end{aligned}
  \end{equation}
  Since there are only finitely many coordinate charts, there are $K_{k}>0$ so that
  \begin{equation*}
    \sum_{\ell=0}^{k}\abs{\nabla^{\ell}\d x_{i}}+\sum_{\ell=0}^{k}\abs{\nabla^{\ell}\bd x_{i}}\le K_{k}
  \end{equation*}
  holds regardless which coordinate chart $Q$ lies inside. Note that we bound the frame $F_{i}=\d x_{i}$ and its dual frame $F_{i}^{\vee}=\bd x_{i}$.

  Then a straightforward computation shows that there is a constant $L_{k}>0$ (depending on $K_{k}$ and $N$) so that whenever $P=\sum p_{i}\cdot \d x_{i}$ we have\footnote{This involves taking $k$th derivatives of $P=\sum p\cdot \d x\circ Q$ or $p_{i}=P\cdot \bd x_{i}\circ Q$ and bounding the result.}
  \begin{equation*}
    \begin{aligned}
      L_{k}^{-1}\norm{p}_{W^{k,2}(\Omega(\rho))}&\le\norm{P}_{W^{k,2}(\Omega(\rho))}\le L_{k}\norm{p}_{W^{k,2}(\Omega(\rho))}.
    \end{aligned}
  \end{equation*}
  To simplify the notation, let us abbreviate $W^{k,2}(\Omega(\rho))=W^{k,2}(\rho)$.

  Apply the elliptic estimate for the regular Laplacian with Dirichlet boundary conditions (see \cite[Lemma C.2]{robbinsalamon}), so there is a constant $E=E(r,\delta,k)$ so that:
  \begin{equation*}
    \begin{aligned}
      \norm{p}_{W^{k+1,2}(r)}&\le E(\norm{\Delta p}_{W^{k-1,2}(r+\delta)}+\norm{p}_{W^{k,2}(r+\delta)})\\&\le E(L_{k-1}\norm{\sum \Delta p_{i}\cdot \d x_{i}\circ Q}_{W^{k-1,2}(r+\delta)}+L_{k}\norm{P}_{W^{k,2}(r+\delta)}).
    \end{aligned}
  \end{equation*}
  From \eqref{eq:2nd-derivs-lc} conclude that:
  \begin{equation*}
    \norm{\Delta^{\mathrm{lc}}P-\sum\Delta p_{i}\cdot \d x_{i}\circ Q}_{W^{k-1,2}(r+\delta)}\le \kappa(N+N^{2})K_{k+1}L_{k}\norm{P}_{W^{k,2}(r+\delta)},
  \end{equation*}
  where $\kappa$ is a combinatorial constant related to the number of terms appearing in \eqref{eq:2nd-derivs-lc}. Therefore:
  \begin{equation*}
    \norm{p}_{W^{k+1,2}(r)}\le E(L_{k-1}\norm{\Delta^{\mathrm{lc}}P}_{W^{k-1,2}(r+\delta)}+(1+\kappa(N+N^{2})K_{k+1})L_{k}\norm{P}_{W^{k,2}(r+\delta)}).
  \end{equation*}
  In particular, setting:
  \begin{equation*}
    L^{-1}_{k+1}C=\max\set{EL_{k-1},(1+\kappa(N+N^{2})K_{k+1})L_{k}},
  \end{equation*}
  then we conclude:
  \begin{equation*}
    \norm{P}_{W^{k+1,2}(r)}\le L_{k+1}\norm{p}_{W^{k+1,2}(r)}\le C(\norm{\Delta^{\mathrm{lc}}P}_{W^{k-1,2}(r+\delta)}+\norm{P}_{W^{k,2}(r+\delta)}),
  \end{equation*}
  as desired.
\end{proof}

\section{A priori estimates for holomorphic curves}
\label{sec:apriori}
This section concerns certain a priori estimates for holomorphic curves with boundary on a totally real submanifold $L$. The first estimate, \S\ref{sec:mvp}, ensures that low energy curves satisfy a gradient bound. The second estimate, \S\ref{sec:elliptic-bootstrapping}, ensures that all the higher derivatives can be bounded in terms of the first derivatives.

\subsection{The mean-value property for the $g$-energy density}
\label{sec:mvp}
\begin{lemma}\label{lemma:mvp}
  Let $(W,J)$ be an almost complex manifold with a Riemannian metric $g$, and let $L$ be a compact totally real submanifold. Let $K\subset W$ be a compact neighbourhood of $L$. Then there exist constants $\epsilon_{0}>0$, $c>0$ depending only on $(K,J,g,L)$ with the following properties:

  If $J_{n}$ be a sequence of almost complex structures which converge to $J$ in the $C^{\infty}$ topology, and $L_{n}\subset K$ a sequence of $J_{n}$-totally real submanifolds which converge to $L$ in the $C^{\infty}$ topology, and
  \begin{equation*}
    u_{n}:D(z,r)\cap \cl{\mathbb{H}}\to (K,L_{n})
  \end{equation*}
  is a sequence of $J_{n}$-holomorphic curves then
  \begin{equation}\label{eq:mvp-inequality}
    \int_{D(z,r)\cap \cl{\mathbb{H}}} \abs{\bd_{s} u_{n}}^{2}\d s\wedge \d t<\epsilon_{0}\implies \abs{\bd_{s}u_{n}(z)}^{2}\le \frac{c}{r^{2}}\int_{D(z,r)\cap \cl{\mathbb{H}}} \abs{\bd_{s}u_{n}}^{2}\d s\wedge\d t,
  \end{equation}
  for $n$ sufficiently large. Here $D(z,r)$ is the disk of radius $r$ centered at $z$ in $\mathbb{C}$, and $\cl{\mathbb{H}}$ is the upper half plane. 
\end{lemma}
\begin{remark}
  Recall that $L_{n}$ converges to $L$ provided $L_{n}$ can be expressed as a $C^{\infty}$ section $\alpha_{n}$ of the normal bundle of $L$ and $\alpha_{n}$ converges to $0$.
\end{remark}
\begin{proof}
  See \cite[Appendix A,B]{robbinsalamon} and \cite[\S4.3]{mcduffsalamon} for the proof in the case when $L_{n}=L$, $J_{n}=J$, $u_{n}=u$. Their argument works with minor modifications in our setting. We should note that we can reduce to the case when $L_{n}=L$ by applying a small $n$-dependent diffeomorphism $\varphi_{n}$ taking $L_{n}$ to $L$ (which changes the complex structure $J_{n}$). Arranging things so that $\varphi_{n}\to \mathrm{id}$ as $n\to\infty$ will preserve the fact that $J_{n}\to J$. As in \cite{robbinsalamon} and \cite{mcduffsalamon}, the crux of the argument is to show that the $g_{n}$-energy density $e_{n}:=\abs{\d u_{n}}_{g_{n}}^{2}$ satisfies $\Delta e_{n}\ge -ce_{n}^{2}$ for some constant $c$ (for $n$ sufficiently large). Here the metric $g_{n}$ can be chosen to be a convergent sequence $g_{n}\to g_{\infty}$. The choice of metric $g_{n}$ is specially adapted to $(L,J_{n})$, and is, in particular, arranged so that the normal derivative to $e_{n}=\abs{\d u_{n}}^{2}_{g_{n}}$ along the boundary vanishes. Consequently, $e_{n}$ can be extended to the full disk $D(z,r)$. Then a generalized mean-value theorem shows that $e_{n}$ satisfies a mean-value property of the type:
  \begin{equation}\label{eq:almost-there}
    \int_{D(z,r)} e_{n}\d s\d t<\epsilon'\implies e_{n}(z)\le \frac{c^{\prime}}{r^{2}}\int_{D(z,r)}e_{n}\d s\d t.
  \end{equation}
  Finally, using the fact that the ratio $g_{x,\infty}(v,v)/g_{x}(v,v)$ is bounded for $x\in K$ and $v\ne 0$, we conclude from \eqref{eq:almost-there} the desired result \eqref{eq:mvp-inequality}.

  The exposition of the argument in \cite[\S4.3]{mcduffsalamon} is very clear, and so we simply refer to them for the rest of the details.
\end{proof}

\subsection{Elliptic bootstrapping to bound higher derivatives}
\label{sec:elliptic-bootstrapping}
\begin{lemma}\label{lemma:bootstrappingone}
  Let $(W,J)$ be an almost complex manifold with a Riemannian metric $g$, and let $L$ be a compact totally real submanifold. Let $K\subset W$ be a compact neighbourhood of $L$.

  Suppose that $J_{n}$ is a sequence of almost complex structures which converge to $J$ in the $C^{\infty}$ topology, and $L_{n}\subset K$ is a sequence of $J_{n}$-totally real submanifolds which converge to $L$. Fix $r>0$. If $u_{n}:D(z_{n},r)\cap \cl{\mathbb{H}}\to (K,L_{n})$ is a sequence of $J_{n}$-holomorphic curves then
  \begin{equation*}
    \limsup_{n\to\infty}\int_{D(z_{n},r)\cap \cl{\mathbb{H}}} \abs{\bd_{s}u_{n}}^{2}\d s\wedge \d t=0\implies \limsup_{n\to\infty}\abs{\nabla^{\ell}\d u_{n}(z_{n})}=0,
  \end{equation*}
  for all $\ell\ge 0$.
\end{lemma}
The proof of Lemma \ref{lemma:bootstrappingone} is given in \S\ref{sec:proof-of-bootstrappingone}. Briefly, the argument uses the mean value property to conclude that $\abs{\d u_{n}(z_{n})}$ is bounded, and then uses elliptic estimates and bootstrapping to bound the higher derivatives.

We will also find the following corollaries of Lemma \ref{lemma:bootstrappingone} useful.
\begin{cor}
  \label{corollary:apriori-estimates}
  Assume the setup of Lemma \ref{lemma:bootstrappingone}. Pick constants $M>0$, $r>0$, and an integer $\ell\ge 0$. Then there is a constant $\epsilon_{\ell}=\epsilon_{\ell}(M,r)>0$ so that
  \begin{equation*}
    \limsup_{n\to\infty}\int_{D(z_{n},r)\cap \cl{\mathbb{H}}} \abs{\bd_{s}u_{n}}^{2}\d s\wedge \d t\le \epsilon_{\ell}\implies \limsup_{n\to\infty}\abs{\nabla^{\ell}\d u_{n}(z_{n})}\le M.
  \end{equation*}  
\end{cor}
\begin{proof}
  Suppose not, then we can find a (diagonal) sequence of curves with energy tending to zero but with $\abs{\nabla^{\ell}\d u_{n}(z_{n})}>M$ for all $n$, contradictng Lemma \ref{lemma:bootstrappingone}.
\end{proof}
\begin{cor}
  \label{corollary:apriori-estimates-2}
  Assume the setup of Lemma \ref{lemma:bootstrappingone}. Then for all $\ell\ge 1$ we have
  \begin{equation*}
    \limsup_{n\to\infty}\sup_{z\in D(z_{n},r)}\abs{\bd_{s}u_{n}(z)}<\infty\implies \limsup_{n\to\infty}\abs{\nabla^{\ell}\d u_{n}(z_{n})}<\infty.
  \end{equation*}
\end{cor}
\begin{proof}
  Suppose not. Then we can find $u_{n},z_{n},r,C,\ell$ and $\delta_{n}\to 0$ so that $\abs{\bd_{s}u_{n}(z)}$ remains bounded by $C$ for all $z\in D(z_{n},r)$ but
  \begin{equation*}
    \limsup_{n\to\infty}\delta_{n}^{\ell+1}\abs{\nabla^{\ell}\d u_{n}(z_{n})}>0.
  \end{equation*}
  Indeed this holds whenever $\abs{\nabla^{\ell}\d u_{n}(z_{n})}$ is unbounded. Then we set
  \begin{equation*}
    w_{n}(z)=u_{n}(z_{n}+\delta_{n}z),
  \end{equation*}
  so that $w_{n}(z)$ is defined on $D(\delta_{n}^{-1}r)$. Clearly the first derivative of $w_{n}$ is bounded by $\delta_{n}C$, which tends to $0$. On the other hand $\nabla^{\ell}\d w_{n}(0)$ does not tend to $0$, by our assumption on $\delta_{n}^{\ell+1}\nabla^{\ell}\d u_{n}(z_{n})$.

  For $n$ large enough, $D(1)\subset D(\delta_{n}^{-1}r)$, and clearly the energy of $w_{n}$ on $D(1)$ is bounded by the energy of $u_{n}$ on $D(z_{n},\delta_{n}r)$ which tends to zero. Thus we can apply Lemma \ref{lemma:bootstrappingone} to conclude that $\abs{\nabla^{\ell}\d w_{n}(0)}\to 0$, contradicting our earlier deduction. This completes the proof.
\end{proof}

\section{Exponential estimates for low-energy holomorphic strips.}
\label{sec:exponential-estimates}

Throughout this section, we fix the following data:
\begin{enumerate}
\item a symplectic manifold $(W,\omega,J)$ with a tame almost complex structure $J$,
\item a compact Lagrangian $L\subset W$, so that $\omega(-,J-)$ is a Riemannian metric along $L$,
\item a compact neighbourhood $K$ of $L$, and
\item\label{data:4} a symplectic tubular neighbourhood $\iota:T^{*}L\supset N\to K$, with the property that the pullback $\d \iota^{-1}J\d\iota=:J$ agrees with $J_{0}$ along $L$ for some Riemannian metric $g$ on $L$. See \S\ref{sec:J0} for the definition of $J_0$. Recall that the induced metric $g$ on $L$ is determined by $g=\omega(-,J-)$.
\end{enumerate}
The results of \S\ref{sec:other-complex-structures} guarantee that there always exists a tubular neighbourhood satisfying (iv) given the data of (i), (ii), (iii). 

Recall that if $L'$ is any submanifold $C^{1}$ close enough to $L$ then (by definition) $L'$ can be expressed as the graph of some one-form $\mathfrak{a}\in \Gamma(T^{*}L)$. We introduce the following notation:
\begin{equation*}
  L_{\mathfrak{a}}:=\text{graph of }\mathfrak{a}\in \Gamma(T^{*}L).
\end{equation*}
Note that $L_{\mathfrak{a}}$ can be considered as a submanifold of $K$ using the tubular neighbourhood \ref{data:4} if $\mathfrak{a}$ is sufficiently small. 

The main result of this section is an exponential $C^1$-bound for holomorphic strips whose derivatives satisfy $\abs{\d u_n} \to 0$ and which take boundary values on Lagrangians $L_{\mathfrak{a}_n}$, $L_{\mathfrak{b}_n}$ which converge to $L$. The idea is similar to the one in Floer's paper, \S\ref{sec:floers-argument}. Our weakened assumptions on the complex structure lead to additional terms in the calculation, but using the low energy assumption and the bootstrapping Lemma \ref{lemma:bootstrappingone}, we show they do not make a significant difference in the limit. The exponential bound is essential in proving the holomorphic strips converging to broken Morse flow lines, as will be shown in the proof of Lemma \ref{lemma:finite-number-case}.

\subsection{$W^{1,2}$ exponential estimates}
This section concerns an exponential estimate for the $W^{1,2}$ size of holomorphic strips with boundary conditions lying on submanifolds $L_{\mathfrak{a}}$ and $L_{\mathfrak{b}}$ nearby $L$.

To set-up the notation, recall that for a sequence of maps $u_{n}$ valued in $T^{*}L$ we have the horizontal and vertical coordinate decomposition $(Q_{n},P_{n})$, as in \S\ref{sec:PandQ}. The key quantity is a modification of the $P$ coordinate. Consider the following boundary value problem, where $\mathfrak{a}_{n},\mathfrak{b}_{n}$ are $1$-forms and $\epsilon_{n}>0$ is a parameter,
\begin{equation}\label{eq:bvp}
  \left\{
    \begin{aligned}
      &u_{n}:[-r_{n},r_{n}]\times [0,1]\to K\text{ (a compact neighbourhood of $L$)},\\
      &u_{n}(s,0)\in L_{\epsilon_{n}\mathfrak{a}_{n}},\hspace{1cm}u_{n}(s,1)\in L_{\epsilon_{n}\mathfrak{b}_{n}}.\\
    \end{aligned}\right.
\end{equation}
Define:
\begin{equation}\label{eq:perturbedP}
  \tilde{P}_{n}(s,t)=P_{n}(s,t)-\epsilon_{n}t(\mathfrak{b}_{n}\circ Q_{n}(s,t))-\epsilon_{n}(1-t)(\mathfrak{a}_{n}\circ Q_{n}(s,t)).
\end{equation}
By construction, $\tilde{P}_{n}(s,0)=\tilde{P}_{n}(s,1)=0$. This fact will be important for integration by parts. Denote by $\gamma_{n}(s)$ the $L^{2}$ size of $\tilde{P}_{n}$,
\begin{equation}\label{eq:gamma_quantity}
  \gamma_n(s) := \dfrac{1}{2}\int_{t=0}^{1} |\tilde{P}_n(s,t)|^2dt.
\end{equation}
Our first technical result is an exponential estimate for $\tilde{P}_{n}$ and its first derivatives.

\begin{lemma}\label{lemma:statementA}
  Let $\epsilon_{n}\to 0$ and $\mathfrak{a}_{n},\mathfrak{b}_{n}$ be convergent sequences of one-forms on $L$. Let $J_{n}\to J$ be a $C^{\infty}$ convergent family of almost complex structures so that $J_{n}|_{L}=J|_{L}$ for all $n$. Suppose that $u_{n}$ is a sequence of $J_{n}$-holomorphic curves satisfying the boundary value problem \eqref{eq:bvp} and $\limsup_{n\to\infty} \norm{\d u_{n}}_{C^{0}}=0$ for some Riemannian metric on $K$. Set $R_{n}+1=r_{n}$. Then:
  \begin{enumerate}[label=(\alph*)]
  \item\label{statementA:a} There exist constants $\delta$ and $K_{n}=K_{n}(u_{n},J,\mathfrak{a}_{n},\mathfrak{b}_{n})$ so that $K_{n}\to 0$ and
    \begin{equation*}
      \ddot{\gamma}_n(s) - \delta^2 \gamma_n(s) \geq \frac{1}{3}\int_{0}^{1}\abs{\nabla_{s}\tilde{P}_{n}}^{2}+\abs{\nabla_{t}\tilde{P}_{n}}^{2}\d t-K_{n}\epsilon_{n}^{2},
    \end{equation*}
    for $s\in [-R_{n}-0.75,R_{n}+0.75]$ and $n$ sufficiently large. 
  \item\label{statementA:b} There exists constants $\theta_{n}=\theta_{n}(u_{n},J,\mathfrak{a}_{n},\mathfrak{b}_{n})$ so that $\theta_{n}\to 0$ and
    \begin{equation}\label{eq:w12-estimate} \int_{s'=s-0.5}^{s+0.5}\int_{t=0}^{1}\abs{\tilde{P}_{n}}^{2}+\abs{\nabla_{s}\tilde{P}_{n}}^{2}+\abs{\nabla_{t}\tilde{P}_{n}}^{2}\d t\d s'\le \theta_{n}(e^{-\delta(R_{n}+s)}+e^{-\delta(R_{n}-s)}+\epsilon_{n}^{2}),
    \end{equation}
    for $s\in [-R_{n},R_{n}]$ and $n$ sufficiently large.
  \end{enumerate}
\end{lemma}
\begin{proof}
  We begin by proving the differential inequality in part \ref{statementA:a}.

  Let $\mathfrak{c}_n = \mathfrak{b}_n - \mathfrak{a}_n$. The covariant derivatives of $\tilde{P}_n$ and $P_n$ are related by the following expressions:
  \begin{equation}
    \label{eq:daren2}
    \begin{aligned}
      P_{n}&=\tilde{P}_{n}+\epsilon_{n} \mathfrak{a}_{n}\circ Q_{n}+t\epsilon_{n}\mathfrak{c}_{n}\circ Q_{n}\\
      \nabla_tP_n & = \nabla_t \tilde{P}_n + \epsilon_{n}(\nabla \mathfrak{a}_{n}\circ Q_{n})\cdot \bd_{t}Q_{n}+\epsilon_{n}\mathfrak{c}_n\circ Q_n + t\epsilon_{n}(\nabla\mathfrak{c}_n\circ Q_{n})\cdot\qt \\
      \nabla_sP_n & = \nabla_s \tilde{P}_n + \epsilon_{n}(\nabla \mathfrak{a}_{n}\circ Q_{n})\cdot \qs+ t\epsilon_{n}(\nabla\mathfrak{c}_n\circ Q_{n})\cdot\qs.
    \end{aligned}
  \end{equation}
  Recall that the $\cdot$ symbol denotes bilinear multiplication of tensors. We reprint the holomorphic curve equations derived in \S\ref{sec:other-complex-structures}:
  \begin{equation*}\tag{\ref{eq:perturbJ}}
    \begin{aligned}
      g_{*}\bd_{s}Q_{n}-\nabla_{t} P_{n}&=A_{n}(u)\cdot P_{n}\cdot \bd_{t}u_{n}\\
      \nabla_{s}P_{n}+g_{*}\bd_{t} Q_{n}&=B_{n}(u)\cdot P_{n}\cdot \bd_{t}u_{n}.
    \end{aligned}
  \end{equation*}
  Recall that $A_{n}$ and $B_{n}$ are $C^{\infty}$-convergent to limits $A,B$. Combining these with \eqref{eq:daren2} yields
  \begin{equation}
    \label{eq:daren3}\left.
      \begin{aligned}
        g_*\qs &= \nabla_t \tilde{P}_n + \epsilon_{n}(\nabla \mathfrak{a}_{n}\circ Q_{n})\cdot \bd_{t}Q_{n}+\epsilon_{n}\mathfrak{c}_n\circ Q_n \\&\hspace{2cm}+ t\epsilon_{n}(\nabla\mathfrak{c}_n\circ Q_{n})\cdot\qt +A_{n}(u_{n})\cdot P_{n}\cdot \bd_{t}u_{n}\\
        g_*\qt &= -\nabla_s \tilde{P}_n - \epsilon_{n}(\nabla \mathfrak{a}_{n}\circ Q_{n})\cdot \qs\\&\hspace{2cm}- t\epsilon_{n}(\nabla\mathfrak{c}_n\circ Q_{n})\cdot\qs + B_{n}(u_{n})\cdot P_{n}\cdot \bd_{t}u_{n}.
      \end{aligned}\right.
  \end{equation}

  Compute the first and second derivatives of $\gamma_n(s)$:
  \begin{equation*}
    \dot{\gamma}_n(s) = \int_{t=0}^{1} \langle\nabla_s \tilde{P}_n,\tilde{P}_n \rangle \d t,
  \end{equation*}
  \begin{equation*}\label{eq:daren4}
    \ddot{\gamma}_n(s) = \int_{t=0}^{1} |\nabla_s\tilde{P}_n|^2 + \langle\nabla_s\nabla_s \tilde{P}_n,\tilde{P}_n \rangle dt
  \end{equation*}
  From equation \eqref{eq:daren3}, we get
  \begin{equation}\label{eq:dylan1}
    \begin{aligned}
      \nabla_{s}\nabla_{s}\tilde{P}_{n} &=-\nabla_{s}(g_{*}\bd_{t}Q_{n})-\epsilon_{n}\nabla_{s}((\nabla \mathfrak{a}_{n}\circ Q_{n})\cdot \qs+t\epsilon_{n}(\nabla\mathfrak{c}_n\circ Q_{n})\cdot\qs) \\&\hspace{8cm}+ \nabla_{s}(B_{n}(u_{n})\cdot P_{n}\cdot \bd_{t}u_{n})\\
                                        &= -\nabla_{s}(g_{*}\bd_{t}Q_{n})+\epsilon_{n}R_{n}^{1}+\nabla_{s}(B_{n}(u_{n})\cdot P_{n}\cdot \bd_{t}u_{n})\\
                                        &=-\nabla_{t}(g_{*}\bd_{s}Q_{n})+\epsilon_{n}R_{n}^{1}+\tilde{P}_{n}\cdot R_{n}^{2}+\nabla_{s}\tilde{P}_{n}\cdot R_{n}^{3}+\epsilon_{n}R_{n}^{4},
    \end{aligned}
  \end{equation}
  where we use Lemma \ref{lemma:interchange} to switch the order of differentiation at the second line and where the remainder terms are given by:
  \begin{equation*}
    \begin{aligned}
      R_{n}^{1}&=-\nabla_{s}\left[(\nabla \mathfrak{a}_{n}\circ Q_{n})\bd_{s}Q_{n}-t(\nabla \mathfrak{c}_{n}\circ Q_{n})\cdot \bd_{s}Q_{n}\right]\\
      \tilde{P}_{n}\cdot R^{2}_{n}+\nabla_{s}\tilde{P}_{n}\cdot R^{3}_{n}+\epsilon_{n}R^{4}_{n}&=\nabla_{s}\left[B_{n}(u_{n})\cdot (\tilde{P}_{n}-\epsilon_{n}\mathfrak{a}_{n}\circ Q_{n}-t\epsilon_{n}\mathfrak{c}_{n}\circ Q_{n})\cdot \bd_{t}u_{n}\right].
    \end{aligned}
  \end{equation*}
  We have used \eqref{eq:daren2} to rewrite $P_{n}$ in terms of $\tilde{P}_{n}$. It is a straightforward consequence of the assumption that $\d u_{n}$ converges to $0$ and Lemma \ref{lemma:bootstrappingone} (ensuring decay on the higher derivatives of $u_{n}$) that one can take $R^{i}_{n}$ so that:
  \begin{equation*}
    \limsup_{n\to\infty}\sum_{i=1}^{4}\sum_{\ell=0}^{k}\abs{\nabla^{\ell}R^{i}_{n}(s,t)}=0,
  \end{equation*}
  uniformly for $(s,t)\in [-R_{n}-0.75,R_{n}+0.75]\times [0,1]$. This uses the fact that $\mathfrak{a}_{n}$, $\mathfrak{b}_{n}$, and $B_{n}$, are convergent to obtain uniform estimates on their derivatives. Simplify things by letting $R^{5}_{n}=R^{1}_{n}+R^{4}_{n}$. 

  Substitute the formula for $g_{*}\bd_{s}Q_{n}$ from \eqref{eq:daren2} into \eqref{eq:dylan1} to obtain:
  \begin{equation*}
    \begin{aligned}
      \nabla_{s}\nabla_{s}\tilde{P}_{n}     &=-\nabla_{t}(\nabla_{t}\tilde{P}_{n}+\epsilon_{n}(\nabla\mathfrak{a}_{n}\circ Q_{n})\cdot \bd_{t}Q_{n}+\epsilon_{n}\mathfrak{c}_{n}\circ Q_{n}+t\epsilon_{n}(\nabla \mathfrak{c}_{n}\circ Q_{n})\cdot \bd_{t}Q_{n})\\
                                            &\hspace{2cm}-\nabla_{t}(A_{n}(u_{n})\cdot P_{n}\cdot \bd_{t}u_{n})+\epsilon_{n}R^{5}_{n}+\tilde{P}_{n}\cdot R^{2}_{n}+\nabla_{s}\tilde{P}_{n}\cdot R_{n}^{3}.\\ &=-\nabla_{t}\nabla_{t}\tilde{P}_{n}+\epsilon_{n}R_{n}^{6}-\nabla_{t}(A_{n}(u_{n})\cdot P_{n}\cdot \bd_{t}u_{n})\\
                                            &\hspace{3cm}+\epsilon_{n}R_{n}^{5}+\tilde{P}_{n}\cdot R_{n}^{2}+\nabla_{s}\tilde{P}_{n}\cdot R_{n}^{3}\\     &=-\nabla_{t}\nabla_{t}\tilde{P}_{n}+\epsilon_{n}R_{n}^{6}+\tilde{P}_{n}\cdot R_{n}^{7}+\nabla_{t}\tilde{P}_{n}\cdot R_{n}^{8}+\epsilon_{n}R_{n}^{9}\\
                                            &\hspace{3cm}+\epsilon_{n}R_{n}^{5}+\tilde{P}_{n}\cdot R_{n}^{2}+\nabla_{s}\tilde{P}_{n}\cdot R_{n}^{3}.
    \end{aligned}
  \end{equation*}
  Here the remainder terms are given by: 
  \begin{equation*}
    \begin{aligned}
      R_{n}^{6}&=-\nabla_{t}\left[(\nabla\mathfrak{a}_{n}\circ Q_{n})\cdot \bd_{t}Q_{n}+\mathfrak{c}_{n}\circ Q_{n}+t(\nabla \mathfrak{c}_{n}\circ Q_{n})\bd_{t}Q_{n}\right],\\
      \tilde{P}_{n}\cdot R_{n}^{7}&+\nabla_{t}\tilde{P}_{n}\cdot R_{n}^{8}+\epsilon_{n}R_{n}^{9}=-\nabla_{t}(A_{n}(u_{n})\cdot P_{n}\cdot \bd_{t}u_{n}),
    \end{aligned}
  \end{equation*}
  using \eqref{eq:daren2} to rewrite $P_{n}$ in terms of $\tilde{P}_{n}$. The fact that all the derivatives of $u_{n}$ are converging to $0$ implies $R^{i}_{n}$ can be chosen so that:
  \begin{equation}\label{eq:remainder-goes-to-zero}
    \limsup_{n\to\infty}\sum_{i=1}^{9}\sum_{\ell=0}^{k}\abs{\nabla^{\ell}R^{i}_{n}(s,t)}=0
  \end{equation}
  uniformly for $(s,t)\in [-R_{n}-0.75,R_{n}+0.75]\times [0,1]$.

  Combining similar terms:
  \begin{equation*}
    R^{10}_{n}=R^{6}_{n}+R^{5}_{n}+R^{9}_{n}\hspace{1cm}R^{11}_{n}=R^{7}_{n}+R^{2}_{n},
  \end{equation*}
  whereby: 
  \begin{equation}\label{eq:dylan2} \nabla_{s}\nabla_{s}\tilde{P}_{n}=-\nabla_{t}\nabla_{t}\tilde{P}_{n}+\epsilon_{n}R^{10}_{n}+\tilde{P}_{n}\cdot R^{11}_{n}+\nabla_{s}\tilde{P}_{n}\cdot R^{3}_{n}+\nabla_{t}\tilde{P}_{n}\cdot R^{8}_{n}.
  \end{equation}
  It follows that:
  \begin{equation}\label{eq:dylan3}
    \begin{aligned}
      \ddot{\gamma}_{n}(s)&=\int_{0}^{1}\abs{\nabla_{s}\tilde{P}_{n}}^{2}\d t+\int_{0}^{1}\ip{-\nabla_{t}\nabla_{t}\tilde{P}_{n},\tilde{P}_{n}}\d t+\epsilon_{n}\int_{0}^{1}\ip{R_{n}^{10},\tilde{P}_{n}}\d t\\&\hspace{1cm}+\int_{0}^{1}\ip{\tilde{P}_{n}\cdot R^{11}_{n},\tilde{P}_{n}}\d t+\int_{0}^{1}\ip{\nabla_{s}\tilde{P}_{n}\cdot R^{3}_{n},\tilde{P}_{n}}\d t+\int_{0}^{1}\ip{\nabla_{t}\tilde{P}_{n}\cdot R^{8}_{n},\tilde{P}_{n}}\d t
    \end{aligned}
  \end{equation}
  Recalling the boundary condition $\tilde{P}_n(s,0) = \tilde{P}_n(s,1) =0$, integrate by parts to obtain:
  \begin{equation*}
    \int_{0}^{1} \langle - \nabla_t\nabla_t\tilde{P}_n,\tilde{P}_n\rangle \d t = \int_{0}^{1} \langle\nabla_t\tilde{P}_n,\nabla_t\tilde{P}_n\rangle \d t=\int_{0}^{1} |\nabla_t\tilde{P}_n|^2 \d t=\norm{\nabla_{t}\tilde{P}_{n}}^{2},
  \end{equation*}
  using the notation:
  \begin{equation*}
    \norm{f}^{2}=\int_{0}^{1}\abs{f}^{2}\,dt.
  \end{equation*}

  The remaining terms can be bounded as follows:
  \begin{equation}\label{eq:remainder-integral-estimate}
    \begin{aligned}
      \abs{\epsilon_{n}\int_{0}^{1}\ip{R_{n}^{10},\tilde{P}_{n}}\d t}&\le \epsilon_{n}\norm{R_{n}^{10}}\norm{\tilde{P}_{n}}\\
      \abs{\int_{0}^{1}\ip{\tilde{P}_{n}\cdot R^{11}_{n},\tilde{P}_{n}}\d t}&\le (\max \abs{R_{n}^{11}(s,t)})\norm{\tilde{P}_{n}}^{2}\\
      \abs{\int_{0}^{1}\ip{\nabla_{s}\tilde{P}_{n}\cdot R^{3}_{n},\tilde{P}_{n}}\d t}&\le (\max \abs{R_{n}^{3}(s,t)})\norm{\nabla_{s}\tilde{P}_{n}}\norm{\tilde{P}_{n}}\\
      \abs{\int_{0}^{1}\ip{\nabla_{s}\tilde{P}_{n}\cdot R^{8}_{n},\tilde{P}_{n}}\d t}&\le (\max \abs{R_{n}^{8}(s,t)})\norm{\nabla_{t}\tilde{P}_{n}}\norm{\tilde{P}_{n}}.
    \end{aligned}
  \end{equation}
  As in \eqref{eq:remainder-goes-to-zero}, for $s\in [-R_{n}-0.75,R_{n}+0.75]$, one can arrange that:
  \begin{equation*}
    \lim_{n\to\infty}\left[\norm{R^{10}_{n}}+\max\abs{R^{11}_{n}}+\max\abs{R^{3}_{n}}+\max\abs{R^{8}_{n}}\right]=0.
  \end{equation*}
  Invoke the trick that $2ab\le a^{2}+b^{2}$, and use equations \eqref{eq:dylan3} with the estimates \eqref{eq:remainder-integral-estimate} to obtain:
  \begin{equation*}
    \begin{aligned}
      \ddot{\gamma}_{n}(s)&\ge \norm{\nabla_{s}\tilde{P}_{n}}^{2}+\norm{\nabla_{t}\tilde{P}_{n}}^{2}-K_{n}(\epsilon_{n}^{2}+\norm{\tilde{P}_{n}}^{2}+\norm{\nabla_{s}\tilde{P}_{n}}^{2}+\norm{\nabla_{t}\tilde{P}_{n}}^{2}),
    \end{aligned}
  \end{equation*}
  for some sequence $K_{n}=K_{n}(u_{n},J,\mathfrak{a}_{n},\mathfrak{b}_{n})\to 0$. Note that any quantity involving $J_{n}$ and its derivatives (e.g., $A_{n},B_{n}$ and their derivatives) can be bounded in terms of $J$ and its derivatives.
  
  Another input is the version of Poincar\'{e} inequality as stated in Lemma \ref{lemma:poincare}, which implies:
  \begin{equation*}
    \norm{\tilde{P}_{n}}^{2}\le c_{\mathrm{pc}}\norm{\nabla_{t} \tilde{P}_{n}}^{2}.
  \end{equation*}
  Deduce the following estimate for $\ddot{\gamma_n}(s)$ for $n$ sufficiently large:
  \begin{equation*}
    \begin{aligned}
      \ddot{\gamma_n}
      &\geq \frac{2}{3}(\norm{\nabla_{s}\tilde{P}_{n}}^{2}+\norm{\nabla_{t}\tilde{P}_{n}}^{2})+(\frac{1}{3}-K_{n})(\norm{\nabla_{s}\tilde{P}_{n}}^{2})\\&\hspace{4cm}+(\frac{1}{3}-K_{n}-c_{\mathrm{pc}}K_{n})\norm{\nabla_{t}\tilde{P}_{n}}^{2})-K_{n}\epsilon_{n}^{2}\\      
      &\ge \frac{2}{3}(\norm{\nabla_{s}\tilde{P}_{n}}^{2}+\norm{\nabla_{t}\tilde{P}_{n}}^{2})-K_{n}\epsilon_{n}^{2}\\
      &\ge \frac{1}{3}(\norm{\nabla_{s}\tilde{P}_{n}}^{2}+\norm{\nabla_{t}\tilde{P}_{n}}^{2})+\frac{1}{3c_{\mathrm{pc}}}\norm{\tilde{P}_{n}}^{2}-K_{n}\epsilon_{n}^{2}.
    \end{aligned}
  \end{equation*}
  Here we assume that $n$ is sufficiently large to make $(1+c_{\mathrm{pc}})K_{n}<1/3$. Recalling the definition of $\gamma_{n}=\norm{\tilde{P}_{n}}^{2}$, conclude that
  \begin{equation*}
    \ddot{\gamma}_{n}-\delta^{2}\gamma_{n}\ge \frac{1}{3}(\norm{\nabla_{s}\tilde{P}_{n}}^{2}+\norm{\nabla_{t}\tilde{P}_{n}}^{2})-K_{n}\epsilon_{n}^{2},
  \end{equation*}
  where $\delta^{2}=1/(3c_{\mathrm{pc}})$. This completes the proof of \ref{statementA:a}. The fact that \ref{statementA:a} implies \ref{statementA:b} follows from Lemma \ref{lemma:statementB} below.
\end{proof}

\begin{lemma}\label{lemma:statementB}
  Let $\gamma,\alpha:[-R-0.75,R+0.75]\to \R^{+}$ be smooth functions and $\delta,K>0$ so that
  \begin{equation*}
    \ddot{\gamma}-\delta^{2}\gamma\ge \alpha-K\epsilon^{2}.
  \end{equation*}
  Then for all $s\in [-R,R]$ we have
  \begin{equation*}
    \gamma(s)+\int_{s-0.5}^{s+0.5}\alpha(s)\,\d s\le C_{1}e^{-\delta(R+s)}+C_{2}e^{-\delta(R-s)}+C_{3}K\epsilon^{2},
  \end{equation*}
  where $C_{1},C_{2},C_{3}$ can be taken to be
  \begin{equation}\label{eq:formula-for-C}    C_{1}=9\gamma(-R-0.75),\hspace{2mm}C_{2}=9\frac{\abs{\dot{\gamma}(R+0.75)+\delta\gamma(R+0.75)}}{2\delta},\text{ and }C_{3}=\frac{9}{\delta^{2}}+\frac{5}{4}.
  \end{equation}
\end{lemma}

\begin{remark}
  To see how this lemma shows that \ref{statementA:a} implies \ref{statementA:b} in Lemma \ref{lemma:statementA}, can take
  \begin{equation*}
    \theta_{n}=\max\set{9\gamma_{n}(-R_{n}-0.75),9\frac{\abs{\dot{\gamma}_{n}(R_{n}+0.75) + \delta\gamma_{n}(R_{n}+0.75)}}{2\delta},(\frac{9}{\delta^{2}}+\frac{5}{4})K_{n}}.
  \end{equation*}
  Then the assumption that $\norm{\d u_{n}}_{C^{0}}\to 0$ implies that the expressions involving $\gamma$ and $\dot\gamma$ tend to zero. Since $K_{n}$ tends to zero from \ref{statementA:a}, we conclude that $\theta_{n}$ tends to zero, as desired.
\end{remark}

\begin{proof}[of Lemma \ref{lemma:statementB}]
  This is inspired by the proof of \cite[Lemma 3.1]{robbinsalamon}. To simplify the notation, introduce $R' = R + 0.75$. Define $\beta:[-R',R']\to \R^{+}$ to be:
  \begin{equation*}
    \beta(s)= \gamma (s)- \frac{K}{\delta^2}\epsilon^2 - Ae^{-\delta(R'-s)},
  \end{equation*}
  where $A = \max\set{(2\delta)^{-1}(\dot{\gamma}(R')+ \delta \gamma(R')),0}$. The reason for this choice of $A$ will be made apparent momentarily.

  Compute \[\ddot{\beta}(s) = \ddot{\gamma}(s)-\delta^2Ae^{-\delta(R'-s)},\] and hence \[ \ddot{\beta}(s) - \delta^2\beta(s) = \ddot{\gamma}(s) - \delta^2\gamma(s) + K\epsilon^2 \geq 0. \]
  
  This implies \[
    \dfrac{d}{ds}\big(e^{-\delta s}(\dot{\beta}+ \delta \beta)\big)\geq 0\]
  
  In particular, if $\dot{\beta}(R') + \delta \beta (R')\leq 0$, then $\dot{\beta}(s) + \delta \beta (s)\leq 0$ for all $s\in \left[-R',R'\right]$. 
  
  By our choice of $A$, we have
  $$\dot{\beta}(R') + \delta\beta(R')=\dot\gamma(R')+\delta \gamma(R')-\frac{K}{\delta}\epsilon^{2}-2\delta A = -\frac{K}{\delta}\epsilon^2 \leq 0,$$ so $\dot{\beta}(s) + \delta \beta (s)\leq 0$ for all $s\in \left[-R',R'\right]$.

  It follows that $e^{\delta s}\beta(s)$ is decreasing on $\left[-R',R'\right]$. Hence \[\beta(s) \leq e^{-\delta(s+R')}\beta(-R')\le e^{-\delta(s+R')}\gamma(-R').\] This implies that:
  \begin{equation}\label{equ:beta}
    \gamma(s)\leq B_1 e^{-\delta(R'+s)} + B_2 e^{-\delta(R'-s)} + B_3 K\epsilon^2,
  \end{equation} 
  where $B_1 = \gamma(-R')$, $B_2 = A \le (2\delta)^{-1}\abs{\dot{\gamma}(R')+ \delta \gamma(R')},$ and $B_3 = \delta^{-2}.$
  
  We estimate the integral $\int_{s-0.5}^{s+0.5}\alpha(s)ds$. As $\gamma \ge 0$, we have $\alpha \leq \ddot{\gamma} + K\epsilon^2$, so \[\int_{s-0.5}^{s+0.5}\alpha(s)ds \leq \dot{\gamma}(s+0.5)-\dot{\gamma}(s-0.5) + K\epsilon^2.\]
  To bound $\dot{\gamma}(s+0.5)$, notice the bound on $\gamma$ by the previous analysis, and since $\ddot{\gamma} \geq \delta^2\gamma + \alpha - K\epsilon^2\geq -K\epsilon^2$, $\dot{\gamma}$ is almost an increasing function, up to order $\epsilon^2$.
  More precisely, for $s' >s$, we have: \[\dot{\gamma}(s') - \dot{\gamma}(s) = \int_{s}^{s'} \ddot{\gamma}(x)dx \geq -K(s'-s)\epsilon^2.\]
  Integrating both sides against $s'$:
  \[ \int_{s'=s}^{s+0.25}\dot{\gamma}(s')ds' - \int_{s'= s}^{s+0.25}\dot{\gamma}(s)ds' \geq \int_{s'=s}^{s+0.25}-K(s'-s)\epsilon^2ds' = \dfrac{-K}{32}\epsilon^2,\]
  which yields:
  \[\gamma(s+0.25) - \gamma(s) - \frac{1}{4}\dot{\gamma}(s) \geq\dfrac{-K}{32}\epsilon^2.\]
  Rearranging:
  \begin{align*}
    \dot{\gamma}(s)
    &\leq4\gamma(s+0.25) - 4\gamma(s) + \dfrac{K}{8}\epsilon_{n}^2 \\
    & \leq  4\gamma(s+0.25)+\dfrac{K}{8}\epsilon^2. 	
  \end{align*}

  Switching the roles of $s$ and $s'$ and integrating $s'$ from $s-0.25$ to $s$, we obtain the inequality in the other direction:
  \[\dot{\gamma}(s) \geq - 4\gamma(s-0.25) - \dfrac{K}{8}\epsilon^2.\]
  
  Therefore:
  \begin{align*}
    \int_{s-0.5}^{s+0.5}\alpha(s)ds
    &\leq \dot{\gamma}(s+0.5)-\dot{\gamma}(s-0.5)+K\epsilon^2 \\
    &\leq 4\gamma(s+0.75)+4\gamma(s-0.75)+\dfrac{5K}{4}\epsilon^2.
  \end{align*}
  Appealing to ($\ref{equ:beta}$), conclude:
  \begin{equation*}
    \gamma(s)+\int_{s-0.5}^{s+0.5}\alpha(s)ds\leq 9B_{1}e^{-\delta(R+s)}+9B_{2}e^{-\delta(R-s)}+(9B_{3}+\frac{5}{4})K\epsilon^{2}.
  \end{equation*}
  The expressions for $B_{i}$ in \eqref{equ:beta} imply that we can set $C_{i}$ as in \eqref{eq:formula-for-C}. This completes the proof.
\end{proof}

\subsection{Bootstrapping $W^{1,2}$ estimates to $C^{1}$ estimates}
\label{sec:bootstrapping-to-C1}

The $W^{1,2}$ estimates on $\tilde{P}_{n}$ will be upgraded to obtain $C^{1}$ estimates. The following lemma is a key ingredient when proving the convergence of holomorphic strips with adiabatic boundary conditions to Morse flow flow lines.
\begin{lemma}\label{lemma:c1estimates}
  Assume the setup and hypotheses of Lemma \ref{lemma:statementA}. Recall that this involved a choice of $\epsilon_{n}\to 0$ and convergent sequences of $1$-forms $\mathfrak{a}_{n}$ and $\mathfrak{b}_{n}$, a sequence $J_{n}\to J$ so that $J_{n}|_{L}=J|_{L}$ and supposed a sequence of holomorphic maps $$u_{n}:[-R_{n}-1,R_{n}+1]\times [0,1]\to K$$ with boundary on $L_{\epsilon_{n}\mathfrak{a}_{n}}$ and $L_{\epsilon_{n}\mathfrak{b}_{n}}$ whose first derivatives converge to $0$.

  The conclusions of Lemma \ref{lemma:statementA} can be upgraded: there exist $\kappa_{n}=\kappa_{n}(u_{n},K,J,\mathfrak{a}_{n},\mathfrak{b}_{n})$ so that $\kappa_{n}\to 0$ and
  \begin{equation*}
    \abs{\tilde{P}_{n}(s,t)}+\abs{\nabla_{s}\tilde{P}_{n}(s,t)}+\abs{\nabla_{t}\tilde{P}_{n}(s,t)}\le \kappa_{n}(e^{-d(R_{n}+s)}+e^{-d(R_{n}-s)}+\epsilon_{n}),
  \end{equation*}
  for all $(s,t)\in [-R_{n},R_{n}]\times [0,1]$ and $n$ sufficiently large. Here $2d=\delta$ is the constant from Lemma \ref{lemma:statementA}.
\end{lemma}
\begin{proof}  
  This proof mainly uses the elliptic estimates for the Levi-Civita Laplacian proved in Lemma \ref{lemma:elliptic estimate}:
  \begin{equation}\label{eq:elliptic-estimate}
    \begin{aligned}
      \norm{\tilde{P}_{n}}_{W^{k+1,2}(\Omega(r))}\le C_{k}(\norm{\Delta^{\mathrm{lc}}\tilde{P}_{n}}_{W^{k-1,2}(\Omega(r+\delta))}+\norm{\tilde{P}_{n}}_{W^{k,2}(\Omega(r+\delta))}).
    \end{aligned}
  \end{equation}
  Together with the exponential estimates on the $W^{1,2}$-norm of $\tilde{P}_n$ developed in Lemma \ref{lemma:statementA}, use \eqref{eq:elliptic-estimates} to estimate the $W^{3,2}$-norm of $\tilde{P}_n$. Then applying the Sobolev embedding theorem (see Lemma \ref{lemma:prereq1}) gives the desired $C^1$ estimates on $\tilde{P}_n$.

  The most important calculations needed to apply \eqref{eq:elliptic-estimate} were performed in the proof of Lemma \ref{lemma:statementA}. In particular, equation \eqref{eq:dylan2} can be rearranged to obtain
  \begin{equation}\label{eq:daren-lapl}
    \Delta^{\mathrm{lc}}\tilde{P}_{n}=\epsilon_{n}R^{10}_{n}+\tilde{P}_{n}\cdot R^{11}_{n}+\nabla_{s}\tilde{P}_{n}\cdot R^{3}_{n}+\nabla_{t}\tilde{P}_{n}\cdot R^{8}_{n},
  \end{equation}
  where the remainder terms and their derivatives converge to zero as $n\to\infty$, uniformly for $(s,t)\in [-R_{n},R_{n}]\times [0,1]$.
  
  Now we give the details of the proof. Consider a sequence of $J$-holomorphic curves $u_n$ satisfying the conditions of Lemma \ref{lemma:statementA}. For $n$ large enough, we can assume $u_n$ lies entirely in the tubular neighbourhood $N$, and write $u_n = (Q_n,P_n)$ in the Levi-Civita coordinates. Define the perturbed maps $\tilde{P}_n$ as in equation (\ref{eq:perturbedP}). Note that the perturbation is along the fiber direction of $T^*L$, so $\tilde{P}_n$ is also a section of $Q^*_nT^*L$, and we will apply Lemma \ref{lemma:elliptic estimate} to $Q_n$ and $\tilde{P}_n$. 
  
  Choose any point $(s_0,t_0) \in [-R_n,R_n]\times [0,1]$. Define the rectangle
  \begin{equation*}
    \Omega(r)=[-s_0-r,s_0+r] \times [0,1].
  \end{equation*}
  around $(s_0,t_0)$.

  Take $r = 0.25$, $\delta = 0.25$, $k=1$ and $N=1$ in Lemma \ref{lemma:elliptic estimate}, defining constants $c_{1}=c(r,\delta)$ and $C_{1}=C(r,\delta,k,N)$. Then take $r=0.125$, $\delta=0.125$, $k=2$ and $N=2$, defining $c_{2}$ and $C_{2}$. Take $c=\mathrm{min}(c_{1},c_{2})$.

  By our assumption that $\limsup_{n\to\infty} \norm{\d u_{n}}_{C^{0}}=0$ and the a priori estimates from Corollary \ref{corollary:apriori-estimates}, we have 
  \begin{equation*}
    \abs{\bd_{s}Q_n}+\abs{\bd_{t}Q_n}\le c\text{ and }\sum_{\ell=0}^{2}\abs{\nabla^{\ell}\bd_{s}Q_n}+\sum_{\ell=0}^{2}\abs{\nabla^{\ell}\bd_{t}Q_n}\le 1
  \end{equation*} 
  on the strip $[-R_{n}-1,R_{n}+1]\times [0,1]$ for $n$ sufficiently large. The a priori estimates in Lemma \ref{lemma:elliptic estimate} yield:
  \begin{equation}\label{eq:elliptic-estimates}
    \begin{aligned}
      \norm{\tilde{P}_{n}}_{W^{2,2}(\Omega(0.25))}&\le C_{1}(\norm{\Delta^{\mathrm{lc}}\tilde{P}_{n}}_{W^{0,2}(\Omega(0.5))}+\norm{\tilde{P}_{n}}_{W^{1,2}(\Omega(0.5))})\\
      \norm{\tilde{P}_{n}}_{W^{3,2}(\Omega(0.125))}&\le C_{2}(\norm{\Delta^{\mathrm{lc}}\tilde{P}_{n}}_{W^{1,2}(\Omega(0.25))}+\norm{\tilde{P}_{n}}_{W^{2,2}(\Omega(0.25))}),
    \end{aligned}
  \end{equation}
  for $n$ large enough. The right hand side of \eqref{eq:elliptic-estimates} can be estimated by the previous $W^{1,2}$ bound of $\tilde{P}_n$ as follows.

  For the term $\norm{\Delta^{\mathrm{lc}}\tilde{P}_n}_{W^{0,2}(\Omega(0.5))}$, equation \eqref{eq:daren-lapl} implies
  $$\norm{\Delta^{\mathrm{lc}}\tilde{P}_n}_{W^{0,2}(\Omega(0.5))} \leq K'_n (\epsilon_{n} + \norm{\tilde{P}_n}_{W^{1,2}(\Omega(0.5))}), $$ where $$K'_n = \sup_{(s,t)\in \Omega{(0.5)}}\text{max}\{ \abs{R^{10}_n(s,t)}, \abs{R^{11}_n(s,t)},\abs{R^{3}_n(s,t)}, \abs{R^{8}_n(s,t)}\}.$$

  It follows from the proof of Lemma \ref{lemma:statementA} that $\lim_{n\to\infty}K'_n=0.$
  For the second term $\norm{\tilde{P}_n}_{W^{1,2}(\Omega(0.5))}$, apply part (b) of Lemma \ref{lemma:statementA} to conclude
  \begin{equation*}
    \begin{aligned}
      \norm{\tilde{P}_n}_{W^{1,2}(\Omega(0.5))}
      &= \big( \int_{s'=s-0.5}^{s+0.5}\int_{t=0}^{1}\abs{\tilde{P}_{n}}^{2}+\abs{\nabla_{s}\tilde{P}_{n}}^{2}+\abs{\nabla_{t}\tilde{P}_{n}}^{2}\d t\d s' \big)^{1/2}\\
      & \leq \big(\theta_{n}(e^{-\delta(R_{n}+s)}+e^{-\delta(R_{n}-s)}+\epsilon_{n}^{2})\big)^{1/2}\\
      &\leq \theta_{n}^{1/2}(e^{-\delta(R_{n}+s)/2}+e^{-\delta(R_{n}-s)/2}+\epsilon_{n}),
    \end{aligned}
  \end{equation*}
  where $\theta_{n}\to 0.$

  Using equation \eqref{eq:elliptic-estimates} conclude
  \begin{equation*}
    \begin{aligned}
      \norm{\tilde{P}_n}_{W^{2,2}(\Omega(0.25))}
      &\le C_{1}(\norm{\Delta^{\mathrm{lc}}\tilde{P}_n}_{W^{0,2}(\Omega(0.5))}+\norm{\tilde{P}_n}_{W^{1,2}(\Omega(0.5))}) \\ 
      &\leq C_1\big(2K'_n\epsilon_{n} +(2K'_n+1)\norm{\tilde{P}_n}_{W^{1,2}(\Omega(0.5))}\big)\\
      & \leq \kappa_{n}'(e^{-d(R_{n}+s)}+e^{-d(R_{n}-s)}+\epsilon_{n}),
    \end{aligned}
  \end{equation*}
  where $d = \delta/2$, $\kappa_{n}'=2C_1K'_n+(2K'_n+1)\theta_{n}^{1/2}.$ Note that $\kappa'_{n}\to 0$.

  Now we perform the bootstrapping argument. As mentioned above, the important calculation in the bootstrapping procedure is equation \eqref{eq:daren-lapl} which expresses $\Delta^{\mathrm{lc}}\tilde{P}_n$ in terms of $\epsilon_{n}$, $\tilde{P}_n$, $\nabla_{s}\tilde{P}_n$, and $\nabla_{t}\tilde{P}_n$ with coefficients $R^i_n(s,t)$.

  Hence, 
  $$\norm{\Delta^{\mathrm{lc}}\tilde{P}_n}_{W^{1,2}(\Omega(0.25))} \leq 2 K_n (\epsilon_{n} + \norm{\tilde{P}_n}_{W^{2,2}(\Omega(0.25))}), $$
  where $K_n$ is some combination of $ R^{10}_n(s,t), R^{11}_n(s,t),R^{3}_n(s,t) , R^{8}_n(s,t)$ and their first derivatives. By equation \eqref{eq:remainder-goes-to-zero}, $R^i_n(s,t)$ as well as their derivatives converge to $0$ uniformly in $s,t$ as $n\to\infty$, and so $K_n$ does as well. Applying \eqref{eq:elliptic-estimates},
  \begin{align*}
    \norm{\tilde{P}_n}_{W^{3,2}(\Omega(0.125))}
    &\le C_{2}(\norm{\Delta^{\mathrm{lc}}\tilde{P}_n}_{W^{1,2}(\Omega(0.25))}+\norm{\tilde{P}_n}_{W^{2,2}(\Omega(0.25))})\\
    & \leq 2C_2K_n\epsilon_{n} +2C_2K_n\norm{\tilde{P}_n}_{W^{2,2}(\Omega(0.25))} \\
    &\leq 2C_2K_n\epsilon_{n} + 2C_2K_n\kappa_{n}'(e^{-d(R_{n}+s)}+e^{-d(R_{n}-s)}+\epsilon_{n})\\
    &\leq \kappa''_{n}(e^{-d(R_{n}+s)}+e^{-d(R_{n}-s)}+\epsilon_{n}),
  \end{align*}
  where $\kappa''_{n} = 2C_2K_n(1+\kappa'_{n})$, which again converges to $0$ as $n \to \infty$. 

  Finally, by the Sobolev inequality, 
  $$\norm{\tilde{P}_n}_{C^{1}(\Omega(0.125))} \leq C_{\mathrm{Sob}} \norm{\tilde{P}_n}_{W^{3,2}(\Omega(0.125))}\le \kappa_{n}(e^{-d(R_{n}+s)}+e^{-d(R_{n}-s)}+\epsilon_{n})$$ for some constant $C_{\mathrm{Sob}}$, and $\kappa_{n} = C_{\mathrm{Sob}} \kappa''_{n}$. This is the desired $C_1$ bound on $\tilde{P}_n(s_0,t_0)$. Noting that $(s_0,t_0)$ was chosen arbitrarily in the rectangle $[-R_n,R_n]\times [0,1]$, the proof is complete.
\end{proof}

\section{Compact partial domains}
\label{sec:compact-partial}

In this section we develop the theory of \emph{compact partial domains} as the primary tool to analyze the degeneration of Riemann surfaces. The first three sections are concerned with the convergence of domains, while in \S\ref{sec:with_maps}, we consider domains $\Sigma_{n}$ equipped with smooth maps $u_{n}:\Sigma_{n}\to W$. In \S\ref{sec:conclusion_compactness}, we tie all the theory together and complete the proof of Theorem \ref{theorem:main_theorem}, modulo the technical arguments postponed to \S\ref{sec:low-energy-regions}; in particular, the analysis involving the convergence to broken flow lines is postponed to \S\ref{sec:low-energy-strips}, where we make use of the estimates proved in \S\ref{sec:exponential-estimates}. The reader who wishes to skip to the part requiring the exponential estimates should refer to Lemma \ref{lemma:finite-number-case}.

See \cite{gro85, hummel, BEHWZ, mcduffsalamon, chiu_chu_katz, MT, abbas-book} for related discussion on compactness in the space of Riemann surfaces and holomorphic curves.

\begin{defn}\label{defn:cpd}
  A \emph{compact partial domain} is a connected Riemann surface $\Sigma$ with corners together with a decomposition of boundary into two pieces $\bd\Sigma\cup \bd C$, so that:
  \begin{enumerate}
  \item $\bd\Sigma\cap \bd C$ forms the finite set of corners,
  \item each component of $\bd C$ can be covered by a neighbourhood $U$ conformally identified with $[0,r]\times S$, where $S=[0,1]$ or $S=\R/\Z$,
  \item $\bd C\cap U$ is mapped on to $\set{r}\times S$,
  \item if $S=[0,1]$, $\bd\Sigma\cap U$ is mapped onto $[0,r]\times \set{0,1}$,
  \item if $S=\R/\Z$, $\bd\Sigma\cap U=\emptyset$.
  \end{enumerate}
  For example, if $\Sigma_{\infty}$ is a punctured Riemann surface, and we remove open ends $C$ around each puncture, then $\Sigma:=\Sigma_{\infty}\setminus C$ is a compact partial domain. Clearly every compact partial domain arises in this fashion, via an obvious completion operation.
\end{defn}

\subsection{Notions of convergence for compact partial domains}
\label{sec:convergence_cpd}
In this section we define a notion of \emph{strong convergence} for a sequence compact partial domains $(\Sigma_{n},\mathfrak{e}_{n})$, where $\mathfrak{e}_{n}$ are ends; see \S\ref{sec:markings-and-ends}. First we explain a more naive notion of weak convergence.

\subsubsection{Weakly convergent sequences of domains}
\label{sec:weakly}
\begin{defn}
  Let $\Sigma_{\infty}$ be a punctured Riemann surface, with punctures $\Gamma$. A weakly convergent sequence of domains is the data of a sequence $(\Sigma_{n},\psi_{n})$ where $\Sigma_{n}$ is a compact partial domain and $\psi_{n}:\Sigma_{n}\to \Sigma_{\infty}$ is an embedding so that:
  \begin{enumerate}[label=(\alph*)]
  \item $\psi_{n}^{-1}(\bd\Sigma_{\infty})=\bd\Sigma_{n}$.
  \item the compact sets $\psi_{n}(\Sigma_{n})$ exhaust $\Sigma_{\infty}$,
  \item $\psi_{n,*}(j_{n})$ converges to $j_{\infty}$ on compact subsets of $\Sigma_{\infty}$. 
  \end{enumerate}
\end{defn}
The limit $\Sigma_{\infty}$ is not uniquely determined by $\Sigma_{n}$, however it will be unique if we impose the following condition:
\begin{defn}
  Let $(\Sigma_{n},\psi_{n})$ and $(\Sigma_{n},\psi_{n}')$ be two weakly convergent sequences (with the same $\Sigma_{n}$). We say that the sequences are \emph{relatively proper} if $$\psi_{n}(z_{n})\to\infty\iff \psi_{n}'(z_{n})\to \infty$$ for all sequences $z_{n}\in \Sigma_{n}$.
\end{defn}
\begin{defn}
  A punctured surface $\Sigma$ is \emph{semi-stable} if there are no non-constant holomorphic maps $S^{2}\to \Sigma$ or $(D,\bd D)\to (\Sigma,\bd\Sigma)$.
\end{defn}
\begin{lemma}\label{lemma:rel-prop}
  Suppose $\Sigma$ and $\Sigma'$ are the limits of two relatively proper sequences with the same underlying $\Sigma_{n}$, and $\Sigma,\Sigma'$ are semi-stable. Consider $f_{n}:=\psi_{n}^{\prime}\psi_{n}^{-1}$ as a sequence which is eventually defined on any compact set. Then, after passing to a subsequence, $f_{n}$ converges to a biholomorphism $f:\Sigma\to \Sigma'$.
\end{lemma}
\begin{proof}
  The relatively proper assumption implies that this sequence is uniformly proper, in the sense that if $f_{n}(z_{n})\in \Sigma'$ diverges to $\infty$, then $z_{n}\in \Sigma$ must be also be diverging to $\infty$.

  In particular, if $K$ is a compact set in $\Sigma$, then $f_{n}(K)$ remains in some compact set $K'$ (independent of $n$). Moreover, for any choice of global metrics, the derivative of $f_{n}$ will be uniformly bounded on $K$, otherwise we would conclude by a bubbling argument the existence of a holomorphic sphere or disk in $\Sigma'$ (contradicting semi-stability). Here we use that $f_{n}$ is approximately holomorphic $K\to K'$, by assumption on $\psi_{n},\psi_{n}'$.

  After passing to a subsequence, $f_{n}$ converges in $C^{0}_{\mathrm{loc}}$ by Arzel\`a-Ascoli to a limit $f:\Sigma\to \Sigma'$. The usual elliptic regularity arguments imply that the convergence is actually in $C^{\infty}_{\mathrm{loc}}$, and hence $f$ is holomorphic. The same argument proves that $f_{n}^{-1}$ converges to a holomorphic map which must be $f^{-1}$. Thus $f$ is a biholomorphism, as desired.
\end{proof}

\subsubsection{Markings and ends}
\label{sec:markings-and-ends}
The data $\mathfrak{e}$ of a collection of holomorphic embeddings $[0,r_{k}]\times S\to \Sigma$ satisfying Definition \ref{defn:cpd} is called a \emph{collection of ends}. Here $r_{k}$ is allowed to vary amongst the components of $\bd C$.

A \emph{marked} compact partial domain is the additional data of a metric on $\bd C$. Ends $\mathfrak{e}$ which parametrize $\bd C$ with constant speed (in time $1$) are said to be compatible with the marking.
\begin{prop}\label{prop:com-mark}
  If $\mathfrak{e},\mathfrak{e}'$ are ends compatible with a marking, then
  \begin{equation*}
    \mathfrak{e}(r_{k}-s,t)=\mathfrak{e}'(r_{k}'-s,t+t_{0})
  \end{equation*}
  for all $s\le \mathrm{min}(r_{k},r_{k}')$, and for some $t_{0}$. In words, the ends agree up to a rotation.
\end{prop}
\begin{proof}
  This follows by analytic continuation. The inclusions of $\mathfrak{e},\mathfrak{e}'$ into $\Sigma$, after precomposing with a rotation, if necessary, have infinitely many intersections (as holomorphic maps). Then by analytic continuation they agree on their entire domain.
\end{proof}

\subsubsection{Strong convergence of compact partial domains with ends}
\label{sec:upgrade}
In order to define a good notion of convergence for a sequence $(\Sigma_{n},\psi_{n})$, we require an additional property which forces the \emph{relatively proper} condition for any other convergent sequence $(\Sigma_{n},\psi_{n}')$. The additional property we use depends on the notion of ends from \S\ref{sec:markings-and-ends}.

\begin{defn}
  A sequence $(\Sigma_{n},\psi_{n},\mathfrak{e}_{n})$ \emph{strongly converges} to $\Sigma$ if $(\Sigma_{n},\psi_{n})$ is weakly convergent, and $\mathfrak{e}_{n}$ is a collection of ends with conformal modulus tending to $\infty$, so that $\psi_{n}(z_{n})$ converges to $\infty$ if and only if $z_{n}$ eventually enters $\mathfrak{e}_{n}$ and converges to $\infty$.
\end{defn}
\begin{defn}
  Two sequences of ends $\mathfrak{e}_{n},\mathfrak{e}_{n}'$ are \emph{compatible} provided $z_{n}$ converges to $\infty$ as measured by $\mathfrak{e}_{n}$ if and only if it does as measured by $\mathfrak{e}_{n}'$, for all sequences $z_{n}\in \Sigma_{n}$.
\end{defn}
\begin{prop}
  If $(\Sigma_{n},\psi_{n},\mathfrak{e}_{n})$ and $(\Sigma_{n},\psi_{n}',\mathfrak{e}_{n}')$ strongly converge, and $\mathfrak{e}_{n},\mathfrak{e}_{n}'$ are compatible, then $\psi_{n},\psi_{n}'$ are relatively proper (and hence $\psi_{n}'\psi_{n}^{-1}$ converges to a biholomorphism $\Sigma\to \Sigma'$).
\end{prop}
\begin{proof}
  Suppose not. Then we can find $\psi_{n}'(z_{n})\to\infty$ so that $\psi_{n}(z_{n})$ converges to $p$. However, by the strong convergence assumption, $z_{n}$ eventually enters $\mathfrak{e}_{n}$ and converges to $\infty$, which then implies that $\psi_{n}(z_{n})$ converges to $\infty$, contradicting the earlier convergence to $p$.
\end{proof}

\subsubsection{Strong convergence and markings}
\label{sec:and_markings}
If $\mathfrak{e}_{n},\mathfrak{e}_{n}'$ are two ends, then the induced change of parametrization of a component of $\bd C_{n}$ is the diffeomorphism $\varphi_{n}(t)$ so that $\mathfrak{e}_{n}'(r_{n}',t)=\mathfrak{e}_{n}(r_{n},\varphi_{n}(t))$. Here $r_{n},r_{n}'$ are the moduli of $\mathfrak{e}_{n},\mathfrak{e}_{n}'$ at the boundary component under consideration. 
\begin{prop}
  If $\mathfrak{e}_{n}$ and $\mathfrak{e}_{n}'$ are compatible ends, then $\varphi_{n}$ converges to a rotation after passing to a subsequence. 
\end{prop}
\begin{proof}
  Double $\mathfrak{e}_{n}$ and $\mathfrak{e}_{n}'$ and consider the change of trivialization as a partially defined map:
  \begin{equation*}
    f_{n}:[-r_{n},r_{n}]\times S\to [-r'_{n},r'_{n}]\times S,
  \end{equation*}
  where $\bd C_{n}$ is identified with $0$, and $f_{n}$ maps $\set{0}\times S$ onto $\set{0}\times S$ via $\varphi_{n}$.

  The compatibility condition guarantees that for finite $s$, $f_{n}$ is eventually defined on $[-s,s]\times S$ (and holomorphic). This is because otherwise we would have a sequence of points which went to infinity for $\mathfrak{e}_{n}$ but not for $\mathfrak{e}'_{n}$.

  Thus there is $s_{n}\to\infty$ so that $f_{n}$ is defined on $[-s_{n},s_{n}]\times S$. The derivative of $f_{n}$ must be bounded on compact subsets, using the translation invariant metric on the source and target. This is because there is no non-constant map $\C\to \R\times \R/\Z$ with finite Hofer energy. Recall that the Hofer energy is a translation invariant energy assigned to holomorphic maps valued in $\R\times \R/\Z$, which assigns finite energy to any embedding but infinite energy to any non-constant map $\C\to \R/\Z$.

  In particular, $f_{n}$ is proper, in the sense that $f_{n}(z_{n})\to\infty$ implies $z_{n}\to\infty$. It follows from the same arguments in Lemma \ref{lemma:rel-prop} that $f_{n}$ converges to a proper biholomorphism $\R\times \R/\Z\to \R\times \R/\Z$ (after passing to a subsequence).

  The only proper biholomorphisms are rotations and translations. Clearly $f_{\infty}$ must preserve $\set{0}\times \R/\Z$ by its construction, and hence $\varphi_{\infty}=f_{\infty}|_{\set{0}\times\R/\Z}$ is a rotation, as desired.
\end{proof}

\subsubsection{Strong convergence implies convergence of the ends}
\label{sec:convergence_of_ends}
A strongly convergent sequence has the auxiliary data of a sequence of ends. As we now explain, we can pass to a subsequence so that this sequence of ends converges.
\begin{prop}
  Suppose that $(\Sigma_{n},\psi_{n},\mathfrak{e}_{n})$ converges strongly to $\Sigma$. Then, after passing to a subsequence,
  \begin{equation*}
    v_{n}:=\psi_{n}|\mathfrak{e}_{n}([1,r_{n}]\times S)
  \end{equation*}
  converges on compact subsets of $[1,\infty)\times S$ to a proper holomorphic embedding.
\end{prop}
\begin{proof}
  The strong convergence assumption implies $v_{n}|_{[0,r]}$ remains bounded independently of $n$, for any fixed $r$. By assumption, $\Sigma$ has at least one puncture, and hence is semi-stable. Bubbling analysis then implies that $v_{n}$ has bounded derivative. Elliptic regularity theory for approximately holomorphic maps (with $C^{1}$ bounds) implies $v_{n}$ converges in $C^{\infty}_{\mathrm{loc}}$ to a limit $v_{\infty}$ (after passing to a subsequence). We shrink from $[0,\infty)$ to $[1,\infty)$ in order to apply the elliptic estimates. See \ref{sec:ell_reg_appx} for more details on elliptic regularity we use.

  We claim that $v_{\infty}$ is non-constant. Suppose not, and let $v_{\infty}=p$. Then we could find $z_{n}\to\infty$ so that $v_{n}(z_{n})$ converged to $p$. But this contradicts the definition of strong convergence.

  Indeed, more generally, if $v_{\infty}$ has $p$ as an asymptotic limit point, then we can find $z_{n}\to\infty$ so $v_{n}(z_{n})$ converges to $p$. The argument is as follows: let $\rho_{k}\to\infty$ and, for $N\in \mathbb{N}$, let $k(N)$ be the largest integer so that for $n\ge N$ we have $$\mathrm{dist}(v_{\infty}(z),v_{n}(z))<1/k$$ for all $z\in [1,\rho_{k}]\times S$.
  
  If we define $r_{n}:=\rho_{k(n)}$, then $v_{n}$ converges uniformly to $v_{\infty}$ on $[1,r_{n}]\times S$.

  In particular, if $p$ is an asymptotic limit point of $v_{\infty}$, then we can find $z_{n}\to\infty$ so $v_{n}(z_{n})$ converges to $p$, contradicting the definition of strong convergence.

  Thus $v_{\infty}$ has no asymptotic limit points, which is equivalent to $v_{\infty}$ being proper.

  Finally, we prove that $v_{\infty}$ is an embedding. If not, then we can find $z^{0}\ne z^{1}$ so that $v_{n}(z^{0})$ converges to $v_{n}(z^{1})$.

  If $\abs{\d v_{\infty}(z^{0})}>\delta$, then $\abs{\d v_{n}(z^{0})}> \delta$ for $n$ sufficiently large, and hence $v_{n}(z^{1})$ cannot enter a ball of fixed radius around $v_{n}(z^{0})$ (using the fact that $v_{n}$ is known to be injective and $z_{0}\ne z_{1}$). Here we use an arbitrary metric on $\Sigma$.

  Thus $\d v_{\infty}(z_{0})=0$. However, by local representation results for holomorphic maps, $v_{\infty}$ must appear as $v_{\infty}(z_{0})+(z-z_{0})^{k}$ in local coordinates around $z_{0},v_{\infty}(z_{0})$. Let $D$ be a small enough disk around $z_{0}$, so that $v_{\infty}(\bd D)$ is disjoint from $v_{\infty}(z_{0})$. Then for $n$ sufficiently large, $v_{n}(\bd D)$ is disjoint from $v_{\infty}(z_{0})$. Let $\Delta$ be a disk around $v_{\infty}(z_{0})$ which contains all of $v_{\infty}(\bd D)$.

  Then $v_{n}|_{\bd D}$ converges in $\Delta\setminus \set{v_{\infty}(z_{0})}$ to $v_{\infty}|_{\bd D}$. However, the winding number is continuous under such convergence, and hence $v_{n}|_{\bd D}$ has winding number $k>1$. It follows easily that $v_{n}|_{D}$ is non-injective, which is a contradiction. Thus $v_{\infty}$ is an embedding, as desired.
\end{proof}

\subsubsection{Strong convergence implies convergence of coordinate charts}
\label{sec:coordinate_disks}
The next result ensures the existence of convergent families of coordinate charts. We state the results for boundary holomorphic charts (defined on $\Omega(1):=D(1)\cap \bar{\mathbb{H}}$), and leave the interior case to the reader. Let $s_{n}\in \bd \Sigma_{n}$ be a sequence of points which remains bounded, as measured by ends $\mathfrak{e}_{n}$, and suppose that $\Sigma_{n},\mathfrak{e}_{n}$ converge strongly to $\Sigma$ for some $\psi_{n}$. Pick metrics $g_{n}$ so that $\psi_{n,*}g_{n}$ converges on compact sets to a complete metric $g$.

\begin{lemma}
  Let $h_{n}:\Omega(1)\to (\Sigma_{n},\bd \Sigma_{n})$ be a sequence of holomorphic maps so that $h_{n}(0)=s_{n}$, and $\abs{\d h_{n}}\in [\delta,\delta^{-1}]$ for some $\delta$, as measured by $g_{n}$. Then, after passing to a subsequence, $\psi_{n}\circ h_{n}$ converges in $C^{\infty}_{\mathrm{loc}}$ to a holomorphic map $h_{\infty}:\Omega(1)\to (\Sigma,\bd \Sigma)$ with $\abs{\d h_{\infty}}\in [\delta,\delta^{-1}]$ (here $C^{\infty}_{\mathrm{loc}}$ means on compact subsets disjoint from $\bd D(1)$).
\end{lemma}
\begin{proof}
  It suffices to prove the existence of subsequences so that $\psi_{n}\circ h_{n}$ converges in $C^{\infty}$ on $\Omega(r)$, for every $r<1$.

  Observe that $h_{n}(\Omega(1))$ remains bounded. This is because the $g_{n}$ length of a path joining $h_{n}(0)$ to $h_{n}(x)$ can be bounded by $\delta^{-1}$. Thus, if $h_{n}(x_{n})$ converges to $\infty$, there are paths in $\Sigma$ whose left endpoint converges but right endpoint diverges to $\infty$ \emph{with bounded $g$-length}. This contradicts the definition of completeness.

  Since $h_{n}(\Omega(1))$ is bounded, we can apply the elliptic regularity result for approximately holomorphic maps given in the next section \S\ref{sec:ell_reg_appx}. This completes the proof.
\end{proof}

\subsubsection{Elliptic regularity for approximately holomorphic maps}
\label{sec:ell_reg_appx}
Here we state the elliptic regularity statement used in the previous sections.
\begin{lemma}
  Suppose that $u_{n}:\Omega(1)\to (W,L,J)$ is \emph{bounded} and satisfies:
  \begin{equation*}
    \bd_{s}u_{n}+J \bd_{t}u_{n}=a_{n}\cdot \d u_{n}+b_{n}
  \end{equation*}
  where $a_{n},b_{n}$ are smooth (tensor valued) functions on the disk which converge to zero (along with all their covariant derivatives) as $n\to\infty$.
  \begin{enumerate}[label=(\alph*)]
  \item If $u_{n}$ has bounded derivative on $\Omega(2/3)$, then all derivatives of $u_{n}$ are bounded on $\Omega(1/2)$. In particular, a subsequence of $u_{n}$ converges in $C^{\infty}$ on $\Omega(1/2)$.
  \item If $u_{n}$ has an unbounded derivative on $\Omega(2/3)$, then there exists a non-constant bounded holomorphic map $\C\to (W,J)$ or $\mathbb{H}\to (W,L,J)$ with bounded derivative.
  \end{enumerate}
  The same result holds with $\Omega$ replaced by $D$.
\end{lemma}
\begin{proof}
  We explain how (a) implies (b). If the derivative was unbounded, then we could apply Hofer's lemma to conclude a sequence of rescaled approximately holomorphic maps defined on $\Omega(r_{n})$ or $D(r_{n})$ with $r_{n}\to\infty$, \emph{with bounded first derivative}, and uniformly non-zero derivative at the origin. Then we can apply (a) to conclude all the higher derivatives of this rescaling are bounded on compact subsets. Hence Arzel\`a-Ascoli implies the rescalings converge to a holomorphic plane or half-plane, with non-zero derivative at the origin.

  The proof of (a) follows from standard bootstrapping, see, e.g., Lemma \ref{lemma:baby-ellreg}.
\end{proof}
\begin{remark}
  Here $(W,J,L)$ is an arbitrary almost complex manifold with totally real submanifold $L$. We use any metric on $W$, all of which give equivalent results since  $u_{n}$ remains uniformly bounded as a function of $n$. The submanifold $L$ is required to be properly embedded, i.e., have closed image. 
\end{remark}

\subsection{Stability of domains}
\label{sec:stability}
We say that a domain $\Sigma$ is \emph{stable} provided its double $D(\Sigma)$ has a negative Euler characteristic. 

\subsubsection{Ends for unstable domains}
\label{sec:canonical_ends_unstable}
An unstable compact partial domain domain is either $\cl{D}(r),\Omega(r)$, $[a,b]\times [0,1]$ or $[a,b]\times \R/\Z$, i.e., the compact approximations of $\C,\mathbb{H}$, $\R\times [0,1]$ or $\R\times \R/\Z$.

The ends of the form $[a,b]\times S$ have canonical markings, while the domains $\cl{D}(1)$, $\Omega(1)=\cl{D}(1)\cap \mathbb{H}$ have canonical markings up to the action of their automorphism group (which is a certain subgroup of the group of M\"obius transformations). We note that $\Omega(1)$ has $\bd C$ equal to $\bd D(1)\cap \Omega(1)$.

Henceforth, we require that unstable domains are given ends which are compatible with these canonical markings.

\subsubsection{Hyperbolic metrics for stable surfaces}
\label{sec:hyperbolic_metrics}
Let $\Sigma_{n}$ be a stable compact partial domain, with boundaries $\bd \Sigma_{n},\bd C_{n}$.
\begin{prop}
  There is a unique hyperbolic metric in the conformal class prescribed by $j_{n}$ so that $\bd\Sigma_{n}$ and $\bd C_{n}$ are geodesic.
\end{prop}
\begin{proof}
  The double of $D(\Sigma_{n})$ along $\bd\Sigma_{n}$ then has boundary circles obtained by doubling $\bd C_{n}$. We can double across these circles a second time and obtain a compact Riemann surface $D(D(\Sigma_{n}))$. If $\Sigma_{n}$ is stable, the uniformization theorem implies that $\mathbb{H}$ is the universal cover of $D(D(\Sigma_{n}))$, and this induces a hyperbolic metric in the conformal class prescribed by the complex structure. This metric renders the doubling loci $\bd \Sigma_{n},\bd C_{n}$ geodesic, as they are preserved under the doubling involution, which must preserve the hyperbolic metric. This proves existence.

  For uniqueness, it suffices to observe that any such metric on $\Sigma_{n}$ can be doubled to obtain a hyperbolic metric on $D(D(\Sigma_{n}))$, by a standard gluing operation from hyperbolic geometry. In particular, the uniqueness reduces to the case when $\Sigma_{n}$ is closed.

  If $\mu_{1}=e^{f}\mu_{0}$ and $\mu_{0}$ is hyperbolic, then one can show that $\mu_{1}$ is hyperbolic if and only if:
  \begin{equation*}
    \Delta f=\frac{2(e^{f}-1)}{y^{2}},
  \end{equation*}
  where $x+iy$ is a coordinate chart where $\mu_{0}$ appears as $y^{-2}(\d x^{2}+\d y^{2})$. In particular, $f$ cannot have a local positive maximum or negative minimum, and hence $f=0$. This proves uniqueness. This above equation for $f$ follows from a tedious computation of the curvature:
  \begin{equation*}
    \kappa=\frac{g(R(X,Y)Y,X)}{g(X,X)g(Y,Y)}\text{ (for any orthogonal frame $X,Y$)},
  \end{equation*}
  and setting the result equal $-1$. 
\end{proof}
\subsubsection{Hyperbolic ends for stable domains}
\label{sec:canonical_ends_stable}
Let $\Sigma_{n}$ be a stable compact partial domain equipped with its canonical hyperbolic structure.
\begin{defn}
  A sequence of ends $\mathfrak{e}_{n}$ is \emph{strongly compatible} with the hyperbolic structure provided that:
  \begin{equation*}
    \mathfrak{i}(z_{n})\to 0\iff z_{n}\text{ eventually enters $\mathfrak{e}_{n}$ and converges to $\infty$},
  \end{equation*}
  where $\mathfrak{i}$ is the injectivity radius. This determines an (at most) unique compatibility class of ends.
\end{defn}
\begin{defn}
  A sequence of ends $\mathfrak{e}_{n}$ is \emph{weakly compatible} with the hyperbolic structure provided that the outermost end of $\mathfrak{e}_{n}$ has injectivity radius converging to zero, while the innermost end of $\mathfrak{e}_{n}$ has injectivity radius bounded below.
\end{defn}
The following theorem restates the well-studied phenomenon of degenerations of hyperbolic surfaces:
\begin{theorem}
  There exist ends $\mathfrak{e}_{n}$ weakly compatible with the hyperbolic structure sequence $\Sigma_{n}$ if and only if the hyperbolic length of $\bd C_{n}$ converges to zero. The parametrization of $\bd C_{n}$ induced by $\mathfrak{e}_{n}$ converges to the constant speed one. The resulting sequence $(\Sigma_{n},\mathfrak{e}_{n})$ converges strongly if and only if the injectivity radius is bounded below by a positive number on $\Sigma_{n}\setminus \mathfrak{e}_{n}$.
\end{theorem}
\begin{proof}
  The argument is rather technical, implicitly relying on results involving pairs of pants decompositions and hyperbolic hexagons. We defer the reader to \cite{abbas-book} or \cite[\S14.4.1]{donaldson}. 
  
  For the claim about the parametrization of $\bd C_{n}$, we observe that $\mathfrak{e}_{n}$ can be doubled and will map larger and larger regions centered around $\bd C_{n}$ properly onto the hyperbolic annulus centered on $\bd C_{n}$ determined by the marking. This map will converge to a proper holomorphic map $\R\times \R/\Z\to \R\times \R/\Z$ which maps $0\times \R/\Z$ onto $0\times \R/\Z$ with winding number $1$. It follows that the map is a rotation.
  
  The results in \cite{abbas-book} explain how to construct the weakly compatible ends, by taking the \emph{isometric} ends defined by equations $\mathfrak{i}<c$ for a sufficiently small universal constant $c$ (whose precise value is irrelevant in our argument).

  In general, the set $\set{\mathfrak{i}<c}$ contains necks $\mathfrak{n}$, ends $\mathfrak{e}$, and \emph{collapsing boundaries} $\mathfrak{b}$. The distinction between $\mathfrak{b}$ and $\mathfrak{e}$ is that the outermost boundary of $\mathfrak{b}_{n}$ lies in $\bd \Sigma_{n}$ rather than $\bd C_{n}$. See \ref{sec:digression_boundary_components}.

  The ends $\mathfrak{e}$ formed by this process are always weakly compatible. It is not hard to see that there is a unique compatibility class of ends which is weakly compatible with the hyperbolic metric.

  If $\set{\mathfrak{i}_{n}<c}$ always contains necks or collapsing boundaries, for all $c$, then $(\Sigma_{n},\mathfrak{e}_{n})$ cannot possible converge. We argue by contradiction. First, suppose that we could find a collapsing boundary~$\mathfrak{b}_{n}$ for $c_{n}\to 0$. Considering $\mathfrak{b}_{n}\simeq [-r_{n},0]\times \R/\Z$, the modulus $r_{n}$ must tend to infinity if $c_{n}\to 0$. Then the sequence of maps $\psi_{n}|_{\mathfrak{b}_{n}}$ can be thought of as bounded sequence of maps into $(\Sigma,\bd\Sigma)$ converging to a punctured disk $(-\infty,0]\times \R/\Z$, which has a removable singularity at $-\infty$. Since the collapsing boundary component must be a circle component, $\psi_{n}$ maps this component diffeomorphically onto a boundary component of $\bd\Sigma$. In particular, the mapping degree of $\psi_{n}|_{\bd \mathfrak{b}_{n}}$ is always $1$, and hence cannot converge to a constant map. However, the existence of a non-constant holomorphic disk contradicts semi-stability. 

  The case of necks is similar, although the argument ruling out the existence of a constant limit is different. Let $\gamma_{n}$ be the central geodesic arc or loop at the center of a neck $\mathfrak{n}$. If the limiting neck $\R\times S\to \Sigma$ were constant, then one can use the Jordan curve theorem to conclude that $\psi_{n}(\gamma_{n})$ must bound a disk (or half-disk) in the limiting surface. Then $\gamma_{n}$ must divide $\Sigma_{n}$ into two pieces, and one of these pieces must be embedded into the disk, say $\Sigma^{+}_{n}$. Since the inside of the disk remains bounded in the limiting surface, $\Sigma^{+}_{n}$ must contain no $\bd C_{n}$ components. It follows that $(\Sigma^{+}_{n},\bd\Sigma^{+}_{n})$ is mapped diffeomorphically onto the disk or half-disk bounded by $\psi_{n}(\gamma_{n})$. However, this implies that the geodesic at the center of the neck $\gamma_{n}$ is contractible in $\Sigma_{n}$, contradicting the (well-known) hyperbolic geometry results that the geodesics arcs and loops at the centers of necks are non-contractible.
\end{proof}
Thus we obtain a definition of convergence for sequences $\Sigma_{n}$ of stable domains:
\begin{defn}
  A sequence of stable compact partial domains $\Sigma_{n}$ \emph{converges in the hyperbolic sense} to $\Sigma$ if there exists $\psi_{n},\mathfrak{e}_{n}$ so that $(\Sigma,\psi_{n},\mathfrak{e}_{n})$ strongly converges and $\mathfrak{e}_{n}$ is weakly compatible with the canonical hyperbolic structure. 
\end{defn}
The limit $\Sigma$ does not depend on the choice of $\psi_{n},\mathfrak{e}_{n}$, up to biholomorphism; indeed, for any choice of $\psi_{n},\psi_{n}'$ the map $\psi_{n}'\psi_{n}^{-1}$ converges to a biholomorphism $\Sigma\to \Sigma'$. 

\subsection{Cutting along necks and compactness}
\label{sec:cutting_and_pasting}
Let $(\Sigma_{n},\mathfrak{e}_{n})$ be a sequence of domains, and suppose that $\mathfrak{n}_{n}\subset \Sigma_{n}$ is an embedded strip $[a_{n},b_{n}]\times S$ (with $[a_{n},b_{n}]\times \bd S\subset \bd \Sigma_{n}$), and so $\mathfrak{n}_{n}$ is disjoint from $\mathfrak{e}_{n}$.

Suppose that $a_{n}<0<b_{n}$, and consider the operation of cutting along $\mathfrak{n}_{n}$, i.e., removing the set $\set{0}\times \bd S$, and then compactifying the result open surface to obtain two new boundary components $\bd C_{n}^{-}$ and $\bd C_{n}^{+}$, with ends $[a_{n},0]$ and $[0,b_{n}]$ adjoined to the collection $\mathfrak{e}_{n}$.

This process shall be referred to as \emph{cutting along a neck}.

\subsubsection{Compactness for hyperbolic domains}
\label{sec:compactness_for_domains}

The following theorem is well-known from hyperbolic geometry.
\begin{theorem}
  Let $\Sigma_{n}$ be a sequence of compact hyperbolic partial domains so that the hyperbolic length of $\bd C_{n}$ converges to zero and the Euler characteristic of $\Sigma_{n}$ is bounded below. Then, after passing to a subsequence, there exist disjoint necks $\mathfrak{n}_{n}^{1},\dots,\mathfrak{n}_{n}^{k}$, so that if we symmetrically cut $\Sigma_{n}$ along those necks, the resulting surface with ends $\mathfrak{e}_{n}$ converges strongly to a limit $\Sigma'$. Here $\mathfrak{e}_{n}$ consists of the original (isometric) ends around the geodesics $\bd C_{n}$ and the ends arising from the disjoint necks.

  The punctures of $\Sigma'$ are divided into \emph{nodal pairs}, which arise from the necks $\mathfrak{n}$, and \emph{original punctures}, arising from the original ends around $\bd C_{n}$. As suggested by the name, there is a duality involution on the set of nodal pairs.
\end{theorem}
\begin{figure}[H]
  \centering
  \begin{tikzpicture}[scale=.8]
    \begin{scope}[scale=0.7]
      \node at (.8,.4) {(a)};
      \node at (5.4,.9) {(b)};
      \node at (9.0,3)[above] {(c)};
      \node at (10.5,1) {$\Sigma_{n}$};
      \draw (-0.3,0) to[out=90,in=-90] (-.5,2) to[out=90,in=180] (1,3) to[out=0,in=180] (5,.5) to [out=0,in=-90] (8,2) arc (180:0:1) to[out=270,in=90] (8.5,0);
      \draw (0.3,0) to[out=90,in=-90] (0.5,2) to[out=90,in=90] (2.9,0);
      \draw (0,2) circle (0.5 and 0.1);
      \draw (9,2) circle (0.7 and 0.7);
      \draw (4.4,0) circle (0.1 and .6);
      \draw (6.4,0) circle (0.1 and .6);
      \draw (8.15,2) circle (0.15 and 0.05);
      \draw (10-.15,2) circle (0.15 and 0.05);
      \draw[dashed] (-1,0)--(10.5,0);
      \begin{scope}[yscale=-1]
        \node at (10.5,1) {$\sigma(\Sigma_{n})$};
        \draw (-0.3,0) to[out=90,in=-90] (-.5,2) to[out=90,in=180] (1,3) to[out=0,in=180] (5,.5) to [out=0,in=-90] (8,2) arc (180:0:1) to[out=270,in=90] (8.5,0);
        \draw (0.3,0) to[out=90,in=-90] (0.5,2) to[out=90,in=90] (2.9,0);
        \draw (0,2) circle (0.5 and 0.1);
        \draw (9,2) circle (0.7 and 0.7);
        \draw (8.15,2) circle (0.15 and 0.05);
        \draw (10-.15,2) circle (0.15 and 0.05);
        \node at (9.0,3)[below] {$\sigma(\mathrm{c})$};
      \end{scope}
    \end{scope}   
    \begin{scope}[shift={(10.5,0)},every node/.style={draw,circle,shading=ball,inner sep=1.5pt}]

      \draw [dashed] (1,0)--(2,0);
      \draw (1,0) arc (0:180:1) (0,0) circle (1 and 0.2) (3,0) circle (1 and 0.2) (4,0) arc (0:180:1);
      \begin{scope}[shift={(150:2.5)},rotate=-120]
        \draw (0.5,0) arc (0:180:.5);
        \draw (0,0) circle (0.5 and 0.1);
        \draw[dashed] (0,.5)--(0,1.5);
        \node at (0,.5) {};
        \node at (0,1.5) {};
      \end{scope}

      \begin{scope}[shift={(3,0)}]
        \draw[dashed] (30:1) to[out=30,in={60},looseness=3] (60:1);
        \node at (-1,0) {};
        \node at (30:1) {};
        \node at (60:1) {};
      \end{scope}
      \node at (1,0) {};
    \end{scope}
  \end{tikzpicture}
  \caption{$\Z/2$-equivariant long cylinders in the doubled domain $D(\Sigma_n)=\Sigma_{n}\cup \sigma(\Sigma_{n})$ (left) converging to nodes, shown as abstract pairings between marked points on the limit domain (right). As explained below, collapsing boundary components, as in (a), need to be isolated via an additional cut.}
  \label{fig:long-cylinders}
\end{figure}
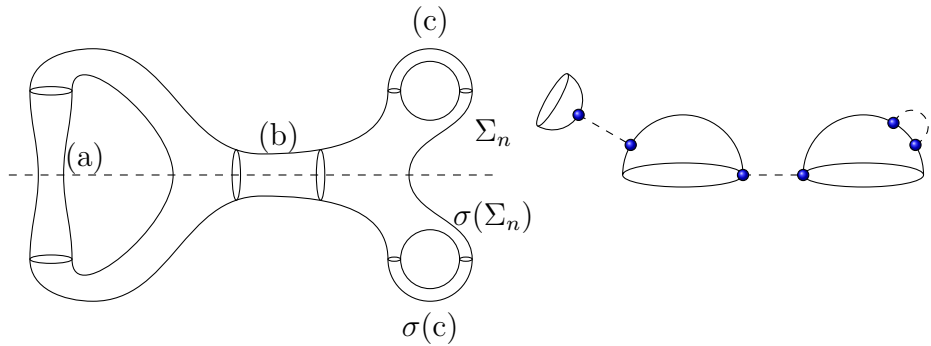
\begin{remark}
  As stated, the limit surface $\Sigma'$ is \emph{not unique}, as one could add in arbitrarily many copies of $\R\times \R/\Z$. 
\end{remark}

\subsubsection{Digression on collapsing boundary components}
\label{sec:digression_boundary_components}
If a sequence of hyperbolic surfaces $\Sigma_{n}$ with $\abs{\bd C_{n}}\to 0$ has a collapsing $\bd\Sigma_{n}$ boundary component $\mathfrak{b}_{n}$, then it will fail to converge. We must therefore make cuts of the form shown in Figure \ref{fig:making_cuts_collapsing_boundary}.
\begin{figure}[H]
  \centering
  \begin{tikzpicture}[scale=.7]
    \draw (0,0) circle (0.3 and 1) (10,0) circle (0.3 and 1) (5,0) circle (0.3 and 1) (0,1)--+(10,0) (0,-1)--+(10,0);
    \draw[draw=none] (0.3,0)--+(5,0)node[right]{$\bd C_{n}$}--+(10,0)node[right]{$\bd \Sigma_{n}$};
    \draw[dashed] (7.5,0) circle (0.3 and 1);
    \begin{scope}[shift={(14,0)}]
      \draw (0,0) circle (0.3 and 1) (0,1) arc (90:-90:1.2 and 1);
      \draw (2.4,0) circle (0.3 and 1) (2.4,1) arc (90:270:1.2 and 1);
      \node at (1.2,0)[draw,circle, inner sep=1pt,fill]{};
      \node at (2.7,0)[right] {$\bd \Sigma$};
    \end{scope}
  \end{tikzpicture}
  \caption{Making cuts to isolate a collapsing boundary. This introduces new $\bd C_{n}$ components. On the right we show a representation of the nodal limit.}
  \label{fig:making_cuts_collapsing_boundary}
\end{figure}
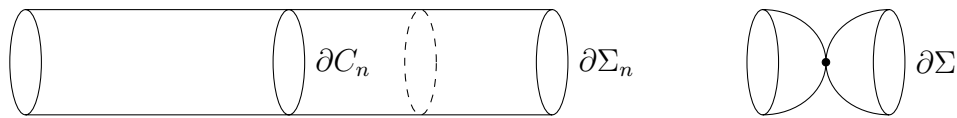
The requirement that $\psi_{n}^{-1}(\bd \Sigma)=\bd \Sigma_{n}$ forces the limit to develop these singly-punctured disks at each collapsing boundary component. This is in contrast to the usual compactification of the space of hyperbolic surfaces, which forbids any unstable domains from appearing in the limit.

\subsubsection{Marked points}
\label{sec:marked_points}
Marked points are the data of a finite subset $\Theta_{n}\subset \Sigma_{n}$. We do not consider these as punctures. A sequence $(\Sigma_{n},\Theta_{n},\mathfrak{e}_{n})$ converges \emph{with marked points} to $(\Sigma,\Theta)$ if there exist $\psi_{n}$ so that $\psi_{n}(\Theta_{n})=\Theta$, the restriction $\psi_{n}:\Theta_{n}\to \Theta$ is a bijection, and $(\Sigma_{n},\psi_{n},\mathfrak{e}_{n})$ strongly converges to $\Sigma$. The main application is considering holomorphic maps with allowed singularities at marked points. In the following statement, we use metrics $g_{n}$ so that $\psi_{n,*}g_{n}$ converges to a complete metric on $\Sigma$.

\begin{lemma}\label{lemma:marked_pt}
  Let $\Sigma_{n},\mathfrak{e}_{n}$ be a sequence of compact partial domains which converges strongly to $\Sigma$, and suppose that $\Theta_{n}$ is a sequence of marked points which satisfies the following assumptions:
  \begin{enumerate}
  \item $\Theta_{n}$ has bounded cardinality,
  \item the marked points remain bounded (as measured by $\mathfrak{e}_{n}$),
  \item there is a minimum distance between interior marked points and $\bd\Sigma_{n}$,
  \item there is a minimum distance between distinct marked points,
  \end{enumerate}
  Then we can pass to a subsequence and construct a set $\Theta$ so that $\Sigma_{n},\Theta_{n},\mathfrak{e}_{n}$ converges to $\Sigma,\Theta$ with marked points.
\end{lemma}
\begin{proof}
  By passing to a subsequence, one ensures that the cardinalities of $\Theta_{n}\cap \bd \Sigma_{n}$ and $\Theta_{n}$ are constant, and moreover that their images $\psi_{n}(\Theta_{n})\cap \bd \Sigma$ and $\psi_{n}(\Theta_{n})$ converge to sets $\Theta\cap \bd \Sigma$ and $\Theta$.

  The ``nearest element'' map $\Theta_{n}\to \Theta$ is injective, and hence is a bijection. By perturbing $\psi_{n}$ by a $C^{1}$ small map which converges to zero, we may suppose that $\psi_{n}(\Theta_{n})=\Theta$, without affecting the strong convergence assumption. This uses the fact that interior marked points remain far from the boundary (otherwise perturbing $\psi_{n}$ could ruin the condition that $\psi_{n}^{-1}(\bd\Sigma)=\bd\Sigma_{n}$). This completes the proof.
\end{proof}

\subsubsection{Compactness for domains with marked points}
\label{sec:application_marked_points}

In this section we prove the following compactness theorem for domains with marked points $\Theta_{n}$:
\begin{lemma}\label{lemma:marked_pt_1}
  Let $\Sigma_{n},\mathfrak{e}_{n}$ be a sequence of compact partial domains which converges strongly to $\Sigma$, and suppose that $\Theta_{n}\subset \Sigma_{n}$ is a collection of marked points with bounded cardinality, and which remains bounded as measured by $\mathfrak{e}_{n}$.

  After passing to a subsequence, there exist a disjoint collection of necks $$\mathfrak{n}^{i}_{n}\simeq [-\rho_{n},\rho_{n}]\times S,$$ $i=1,\dots,k$, with $\rho_{n}\to\infty$, so that the new sequence $\Sigma_{n}'$ obtained by cutting along the centers of the necks satisfies:
  \begin{enumerate}
  \item the marked points $\Theta_{n}$ are disjoint from each neck $\mathfrak{n}$, and hence $\Theta_{n}$ induces a collection of marked points in $\Sigma_{n}'$,
  \item the induced sequence $(\Sigma_{n}',\mathfrak{e}_{n}',\Theta_{n})$ converges with marked points to $(\Sigma',\Theta)$, for some $\Sigma'$, and
  \item the punctures of $\Sigma'$ consist of the original punctures of $\Sigma$ together with a collection of nodal pairs.
  \end{enumerate}
\end{lemma}
\begin{proof}
  As in Lemma \ref{lemma:marked_pt}, it is sufficient to ensure that the marked points remain a minimum distance apart, and interior marked points remain a minimum distance from the boundary. To perform this analysis, introduce metrics $g_{n}$ which converge to a complete metric $g$ on $\Sigma$, via the map $\psi_{n}$.
  
  First consider interior marked points which are converging to the boundary. Suppose that $\psi_{n}(z_{n})$, where $z_{n}$ are interior marked points, converges to $s_{\infty}$ on the boundary of the limit $\Sigma$.

  As in \S\ref{sec:coordinate_disks}, pick coordinate disks $h_{n}:\Omega(1)\to (\Sigma_{n},\bd\Sigma_{n})$ which converge to a coordinate disk $h_{\infty}$ centered on $s_{\infty}$, so that $z_{n}=h_{n}(it_{n})$, $t_{n}\to 0$.

  Introduce the neck $\mathfrak{m}_{n}=h_{n}(\Omega(1)\setminus \Omega(t_{n}))$; and note that $\psi_{n}$ converges uniformly on $\mathfrak{m}_{n}$ to a punctured disk (namely, the punctured image of $h_{\infty}$). Moreover, $(h_{n}(\Omega(1)),z_{n})$, with the map $\psi_{n}^{1}=\frac{1}{t_{n}}h_{n}^{-1}$ converges with marked points to $(\mathbb{H},i)$. 

  We explain how to make the cuts. Consider $\mathfrak{m}_{n}$ as identified with $[0,r_{n}]\times [0,1]$ where $r_{n}\to\infty$, via the map $(s,t)\mapsto t_{n}e^{\pi (s+it)}$.

  Call a radius $k$-\emph{special} if the subneck $[0,r]\times [0,1]$ contains $k$ marked points infinitely often. Henceforth let us abbreviate the notation and denote rectangle $[a,b]\times [0,1]$ by their base $[a,b]$.

  Let $k$ be the maximal integer for which there exists a $k$-special subneck, and let $r$ be such a radius. By passing to a subsequence, we may suppose that $[0,r]$ \emph{always} contains $k$ marked points.

  Clearly, there exist $\rho_{n}\to\infty$ so that $[0,r+2\rho_{n}]$ always contains $k$ marked points, and \emph{none of which} are in $[r,r+2\rho_{n}]$.

  Similarly, we can find $r',\rho_{n}'$ so that $[r_{n}-r'-2\rho_{n}',r_{n}]$ never has any marked points in $[r_{n}-r'-2\rho_{n}',r_{n}-r']$. We may suppose that $2(\rho_{n}+\rho_{n}')<r_{n}$. By letting $r=\mathrm{max}(r,r')$ and shrinking $\rho_{n},\rho_{n}'$ if necessary, we may suppose that $r=r'$ and $\rho_{n}'=\rho_{n}$.
  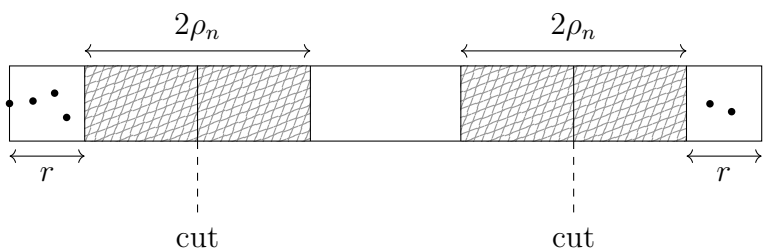
\begin{figure}[H]
    \centering
    \begin{tikzpicture}
      \fill[pattern={Lines[angle=20]}, pattern color=black!40!white] (1,0)rectangle(4,1);
      \fill[pattern={Lines[angle=70]}, pattern color=black!40!white] (1,0)rectangle(4,1);
      \fill[pattern={Lines[angle=20]}, pattern color=black!40!white] (9,0)rectangle+(-3,1);
      \fill[pattern={Lines[angle=70]}, pattern color=black!40!white] (9,0)rectangle+(-3,1);
      \draw (0,0) rectangle (10,1);    
      \draw (1,0) --+ (0,1) (9,0) --+ (0,1) (2.5,0) --+ (0,1) (4,0) --+ (0,1) (6,0)--+(0,1) (7.5,0) --+ (0,1);
      \draw[shift={(0,-0.2)},<->] (0,0)--node(B)[below]{$r$}+(1,0);
      \draw[shift={(0,0.2)},<->] (1,1)--node[above]{$2\rho_{n}$}+(3,0);
      \draw[shift={(0,0.2)},<->] (6,1)--node[above]{$2\rho_{n}$}+(3,0);
      \draw[shift={(0,-0.2)},<->] (9,0)--node(A)[below]{$r$}+(1,0);
      \draw[dashed] (2.5,0)--+(0,-1) node[below] {cut};
      \draw[dashed] (7.5,0)--+(0,-1) node[below] {cut};
      \path (A)--+(0,1) coordinate (Y);
      \path (B)--+(0,.9) coordinate (X);
      \path[every node/.style={inner sep=1pt,fill,circle}] (X)--+(60:0.2)node{}--+(160:0.2)node{}--+(330:0.3)node{};
      \path[every node/.style={inner sep=1pt,fill,circle}] (0,0.5)node{}--(Y)--+(-60:0.2)node{}--+(-160:0.2)node{};
    \end{tikzpicture}
    \caption{Making cuts. The shaded regions never have any marked points, by construction.}
    \label{fig:cutting}
  \end{figure}

  We cut the surface as shown above, i.e., 
  \begin{equation*}
    \mathfrak{m}_{n}=[0,r+\rho_{n}]\cup [r+\rho_{n},r_{n}-r-\rho_{n}]\cup [r_{n}-r-\rho_{n},r_{n}],
  \end{equation*}
  and let the corresponding necks be equal to:
  \begin{equation*}
    \begin{aligned}
      \mathfrak{n}_{n}^{1}&=[r,r+\rho_{n}]\cup [r+\rho_{n},r+2\rho_{n}]\simeq [-\rho_{n},\rho_{n}]\\\mathfrak{n}_{n}^{2}&=[r_{n}-r-2\rho_{n},r_{n}-r-\rho_{n}]\cup [r_{n}-r-\rho_{n},r_{n}-r]\simeq [-\rho_{n},\rho_{n}].
    \end{aligned}
  \end{equation*}
  It is clear that when we symmetrically cut $\Sigma_{n}$ along these two necks, we obtain three sequences of compact partial domains:
  \begin{enumerate}
  \item The region $\Sigma_{n}^{0}=\Sigma_{n}\setminus h_{n}(\Omega(e^{-\pi(r+\rho_{n})}))$, with the right half of the neck $\mathfrak{n}^{2}$ added to the existing ends $\mathfrak{e}_{n}$. It is clear that this converges to $\Sigma\setminus \set{s_{\infty}}$ using the same maps $\psi_{n}$.
  \item The region $\Sigma_{n}^{1}=h_{n}(\Omega(t_{n}e^{\pi(r+\rho_{n})}))$ which converges to $\mathbb{H}$ using the map $\psi_{n}^{1}$, with the left half of the neck $\mathfrak{n}^{1}$ as the sole end.
  \item The middle region $\Sigma_{n}^{2}\simeq [r+\rho_{n},r_{n}-r-\rho_{n}]$. By our earlier assumption that $2(\rho_{n}+\rho_{n})<r_{n}$, this middle region has conformal modulus tending to infinity. We give this the ends $[r+\rho_{n},r+2\rho_{n}]$, and $[r_{n}-r-2\rho_{n},r_{n}-r-\rho_{n}]$ (i.e., the right of $\mathfrak{n}^{1}$ and the left of $\mathfrak{n}^{2}$). This middle region with these ends will not converge strongly, in general.
  \end{enumerate}  
  Observe that the ends have been chosen so that no marked points enter $\mathfrak{n}^{1}_{n}$ or $\mathfrak{n}^{2}_{n}$. In particular, the hypotheses of this theorem still apply to $\Sigma_{n}^{0}$ and $\Sigma_{n}^{1}$ (i.e., the marked points remain bounded as measured by the ends). Moreover, in the new sequence, the marked points are slightly better behaved: if there were $N$ interior marked points which converged to points on the boundary previously, now there are at most $N-1$, since we have fixed the behaviour for $z_{n}$.
  
  Let us explain how to fix the middle region $\Sigma_{n}^{2}$ so as to make it converge. Observe that, for some $R_{n}$,
  \begin{equation*}
    \Sigma_{n}^{2}=[-R_{n}-\rho_{n},R_{n}+\rho_{n}],
  \end{equation*}
  and the ends are $[R_{n},R_{n}+\rho_{n}]$ and $[-R_{n}-\rho_{n},-R_{n}]$, where $\rho_{n}$ tends to $\infty$. Such a sequence converges if and only if $R_{n}$ is bounded (in which case we can use $\psi_{n}=\mathrm{id}$ to make it strongly converge to $\R\times [0,1]$). Thus, assume $R_{n}$ is unbounded. Such a sequence is called a \emph{long neck}.

  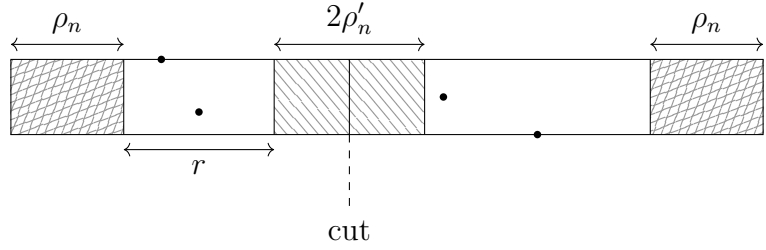
\begin{figure}[H]
    \centering
    \begin{tikzpicture}
      \fill[pattern={Lines[angle=20]}, pattern color=black!40!white] (0,0)rectangle(1.5,1);
      \fill[pattern={Lines[angle=70]}, pattern color=black!40!white] (0,0)rectangle(1.5,1);
      \fill[pattern={Lines[angle=130]}, pattern color=black!40!white] (3.5,0)rectangle(5.5,1);
      \fill[pattern={Lines[angle=20]}, pattern color=black!40!white] (10,0)rectangle+(-1.5,1);
      \fill[pattern={Lines[angle=70]}, pattern color=black!40!white] (10,0)rectangle+(-1.5,1);
      \draw (0,0) rectangle (10,1);
      \draw (1.5,0) --+ (0,1) (8.5,0) --+ (0,1) (3.5,0)--+(0,1) (5.5,0)--+(0,1) (4.5,0)--+(0,1);
      \draw[shift={(0,0.2)},<->] (0,1)--node[above]{$\rho_{n}$}+(1.5,0);
      \draw[shift={(0,0.2)},<->] (10,1)--node[above]{$\rho_{n}$}+(-1.5,0);
      \draw[shift={(0,0.2)},<->] (3.5,1)--node[above]{$2\rho_{n}'$}+(2,0);
      \draw[shift={(0,-0.2)},<->] (1.5,0) --node[below]{$r$}+ (2,0);
      \draw[dashed] (4.5,0)--+(0,-1) node[below] {cut};
      \path[every node/.style={inner sep=1pt,fill,circle}] (7,0)node{}--(5.75,0.5) node {};
      \path[every node/.style={inner sep=1pt,fill,circle}] (2,1)node{}--(2.5,0.3) node {};
    \end{tikzpicture}
    \caption{Making another cut to a long cylinder. The region on the left of the cut converges strongly (as $r$ is bounded).}
    \label{fig:cutting-2}
  \end{figure} 
  By our assumption, all the marked points are in $[-R_{n},R_{n}]$. As we argued above, let $r$ be a $k$-special radius for maximal $k$, and by passing to a subsequence assume that there are no marked points in $[-R_{n}+r,-R_{n}+r+2\rho_{n}']$ where $2\rho_{n}'$ increases to infinity slowly compared to $R_{n}$. We should choose $\rho_{n}'$ somewhat optimally in order for the process to terminate, and so we require that there \emph{is} a marked point in:
  \begin{equation*}
    [-R_{n}+r,-R_{n}+r+2\rho_{n}'+1],
  \end{equation*}
  or, if there no marked points, pick $\rho_{n}'$ so that $-R_{n}+r+2\rho_{n}+1=R_{n}-\rho_{n}$ (i.e., go right up to the right end). Then we make the cut at $-R_{n}+r+\rho_{n}'$ so that 
  \begin{equation*}
    \Sigma_{n}^{3}=[-R_{n}-\rho_{n},-R_{n}+r+\rho_{n}']\text{ and }\Sigma_{n}^{4}=[-R_{n}+r+\rho_{n}',R_{n}+\rho_{n}],
  \end{equation*}
  with the ends as shown in Figure \ref{fig:cutting-2}. 

  By construction, $\Sigma_{n}^{3}$ converges strongly. By iterating the argument, we can make finitely many cuts so that each piece converges strongly to $\R\times [0,1]$. We remark that our optimal choice of $\rho_{n}'$ implies that the next time we make a cut, the maximal $k$ will be at least $1$; this ensures the iterative cutting process will terminate in finitely many steps.

  Summarizing, we can decompose our original sequence $\Sigma_{n}$ into convergent pieces (by cutting along necks), so that the hypotheses of the theorem still apply to $\Sigma_{n}$, and so that a sequence of interior points $z_{n}$ which originally converged to the boundary now no longer converges to the boundary.
  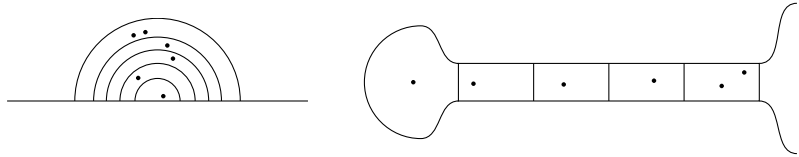
\begin{figure}[H]
    \centering
    \begin{tikzpicture}
      \draw (-2,0) -- (2,0);
      \begin{scope}
        \clip (-2,0) rectangle (2,1.1);
        \draw (0,0) circle (1.1) circle (0.85) circle (0.5) circle (0.3) circle (0.68);
        \path[every node/.style={fill,circle,inner sep=0.6pt}] (40:0.1) node{} (130:0.4) node{} (70:0.6) node{} (80:0.75) node{} -- (110:0.93) node{} -- (100:0.93) node{};
      \end{scope}
      \begin{scope}[shift={(4,0)}]
        \draw (0,0.5) to[out=180,in=0] (-0.5,1) arc (90:270:0.75) (0,0) to[out=180,in=0] (-0.5,-0.5) (0,0) rectangle (4,0.5);
        \draw (4,0.5) to[out=0,in=180] (4.5,1.3) (4,0) to[out=0,in=180] (4.5,-0.7);
        \foreach \k in {1,2,3} {
          \draw (\k,0)--+(0,0.5);
        }
        \node[fill,circle,inner sep=0.6pt] at (3.5,0.2) {};
        \node[fill,circle,inner sep=0.6pt] at (-0.6,0.25) {};
        \foreach \k in {1,2,3,4} {
          \node[fill,circle,inner sep=0.6pt] at (\k+0.2*\k-1,0.3-0.1*\k+0.03*\k*\k) {};
        }
      \end{scope}
    \end{tikzpicture}
    \caption{Fixing an interior puncture converging to the boundary}
    \label{fig:fixing_int_bd}
  \end{figure}
  Returning to the original sequence $\Sigma_{n}$, let $k$ be the maximal number so that there are injective maps $z_{n}:\set{1,\dots,k}\to \Theta_{n}$ so that each element in $z_{n}(k)$ has a limit point on the boundary. If $k>0$, then we can apply the above process to obtain a new sequence $\Sigma_{n}'$ (cut into many pieces, but considered as a single compact partial domain with ends), with new number $k'<k$. In this fashion, it suffices to prove the theorem when $k=0$.

  The problem of two marked points converging together is handled in a similar fashion: let $m$ be the maximal number so that:
  \begin{enumerate}
  \item there exist injective maps $p_{n}:\set{1,\dots,m}\to \Theta_{n}\times \Theta_{n}$ which are disjoint from the diagonal, and
  \item $\mathrm{dist}(p_{n}(j))\to 0$ for each $j$.
  \end{enumerate}
  By performing a similar iterative argument, it suffices to prove the theorem when $m=0$. However, in this case, we can directly apply Lemma \ref{lemma:marked_pt}, and this completes the proof.
\end{proof}

\subsection{Compact partial domains equipped with punctured maps}
\label{sec:with_maps}
Let $(\Sigma_{n},\mathfrak{e}_{n},\Theta_{n},u_{n})$ be a sequence of compact partial domains with ends $\mathfrak{e}_{n}$, marked points $\Theta_{n}$, and maps $u_{n}:\Sigma_{n}\setminus \Theta_{n}\to W$, where $W$ is a metric space.
\begin{defn}
  We say that $(\Sigma_{n},\mathfrak{e}_{n},\Theta_{n},u_{n})$ \emph{converges uniformly} to $(\Sigma,u)$ provided that $(\Sigma_{n},\mathfrak{e}_{n},\Theta_{n},\psi_{n})$ converges with marked points to $(\Sigma,\Theta)$ and $$\lim_{n\to\infty}\sup_{z\in \Sigma_{n}\setminus \Theta_{n}}\mathrm{dist}(u\circ \psi_{n}(z),u_{n}(z))=0.$$
  This is the notion of convergence used in Definition \ref{defn:convergence_defn}.
\end{defn}

\subsubsection{Tame sequences of maps}
\label{sec:holomorphic_necks}
\label{sec:tame_sequences}
Let $W$ be a smooth manifold, and let $(\Sigma_{n},\mathfrak{e}_{n},\Theta_{n})$ converge with marked points to $(\Sigma,\Theta)$. A sequence $u_{n}:\Sigma_{n}\setminus \Theta_{n}\to W$ is called \emph{tame} if any sequence of \emph{end restrictions} $u_{n}|\mathfrak{e}_{n}$ has a subsequence which converges uniformly (in $C^{1}$) to a map defined on $[0,\infty)\times S$. Since $\psi_{n}|\mathfrak{e}_{n}$ converges to infinite ends of $\Sigma$, it follows that $u_{n}\circ \psi_{n}^{-1}$ converges uniformly on a neighbourhood of each puncture in $\Sigma$ (bear in mind that these neighborhoods are punctured). 

A sequence of maps $[a_{n},b_{n}]\times S\to W$ is said to be \emph{very tame} if it converges to a constant map. Note that a tame sequence has a preferred direction for $s$ (and cannot be reversed in general), while very tame sequences are preserved under the reparametrization $(s,t)\mapsto (-s,1-t)$.

Often it will occur that a tame end can be concatenated with a very tame sequence, in the sense that the domains can be glued (as shown) and the maps extend smoothly over the interface.
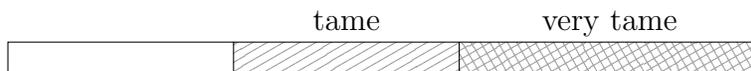
\begin{figure}[H]
  \centering
  \begin{tikzpicture}[yscale=.4]
    \fill[pattern={Lines[angle=30]},pattern color=black!40!white] (3,0) rectangle (10,1);
    \fill[pattern={Lines[angle=-60]},pattern color=black!40!white] (6,0) rectangle (10,1);
    \path (3,2.4)--node[below]{tame}+(3,0);
    \path (6,2.4)--node[below]{very tame}+(4,0);
    \draw (0,0) rectangle (10,1) (3,0)--+(0,1) (6,0)--+(0,1);
  \end{tikzpicture}
  \caption{Adding a very tame sequence to existing tame ends.}
  \label{fig:tame-very-tame}
\end{figure}
The resulting glued map will also have tame ends.  

\subsubsection{Cutting holomorphic necks}
\label{sec:cutting_and_pasting}
We describe a cutting lemma which results in \emph{tame} ends, at the expense of an additional long neck in the middle, as shown in Figure \ref{fig:cutting} 
\begin{defn}
  A sequence of \emph{holomorphic necks} is the data of: $$u_{n}:[a_{n},b_{n}]\times S\to (W,L_{n},J_{n})$$ where $L_{n}=L^{0}_{n}\cup L^{1}_{n},J_{n}$ are convergent sequences of totally real submanifolds and almost complex structures, $u_{n}$ are $J_{n}$-holomorphic, and $b_{n}-a_{n}\to\infty$.
\end{defn}
The next lemma will be used to cut holomorphic necks so as to ensure the tame condition holds.
\begin{lemma}\label{lemma:neck-cutting-lemma}\label{lemma:little-trick}
  Let $u_{n}$ be a bounded sequence of holomorphic necks with uniformly bounded energy. Then there exist $r$, $\rho_{n}\to\infty$, so that the restrictions:
  \begin{equation*}
    u_{n}|_{[a_{n}+r,a_{n}+r+2\rho_{n}]}\text{ and }u_{n}|_{[b_{n}-r-2\rho_{n},b_{n}-r]}
  \end{equation*}
  converge \emph{uniformly} to holomorphic limits defined on $[0,\infty)$ and $(-\infty,0]$ (after the appropriate retranslation, i.e., $a_{n}+r$ and $b_{n}-r$ should be identified with $0$). The restrictions:
  \begin{equation*} 
    u_{n}|_{[a_{n}+r+\rho_{n},a_{n}+r+2\rho_{n}]}\text{ and }u_{n}|_{[b_{n}-r-2\rho_{n},b_{n}-r-\rho_{n}]}    
  \end{equation*}
  are very tame, i.e., converge uniformly to points, (these are the heavily shaded regions in Figure \ref{fig:cutting-holo}).
\end{lemma}
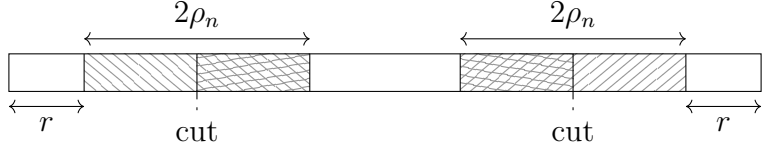
\begin{figure}[H]
  \centering
  \begin{tikzpicture}[yscale=.5]
    \fill[pattern={Lines[angle=10]}, pattern color=black!40!white] (4,0)rectangle+(-1.5,1);
    \fill[pattern={Lines[angle=-40]}, pattern color=black!40!white] (4,0)rectangle+(-3,1);
    \fill[pattern={Lines[angle=-10]}, pattern color=black!40!white] (6,0)rectangle+(1.5,1);
    \fill[pattern={Lines[angle=40]}, pattern color=black!40!white] (6,0)rectangle+(3,1);
    \draw (0,0) rectangle (10,1);    
    \draw (1,0) --+ (0,1) (9,0) --+ (0,1) (2.5,0) --+ (0,1) (4,0) --+ (0,1) (6,0)--+(0,1) (7.5,0) --+ (0,1);
    \draw[shift={(0,-0.4)},<->] (0,0)--node[below]{$r$}+(1,0);
    \draw[shift={(0,0.4)},<->] (1,1)--node[above]{$2\rho_{n}$}+(3,0);
    \draw[shift={(0,0.4)},<->] (6,1)--node[above]{$2\rho_{n}$}+(3,0);
    \draw[shift={(0,-0.4)},<->] (9,0)--node[below]{$r$}+(1,0);
    \draw[dashed] (2.5,0)--+(0,-.5) node[below] {cut};
    \draw[dashed] (7.5,0)--+(0,-.5) node[below] {cut};
  \end{tikzpicture}
  \caption{Using the lemma to make cuts. The shaded regions converge uniformly to holomorphic limits. The heavily shaded regions are very tame. The neck in the middle may be \emph{long}, i.e., might not converge strongly, although it will have very tame ends.}
  \label{fig:cutting-holo}
\end{figure}
\begin{proof}
  We explain the argument at the left end. Without loss of generality, suppose that $a_{n}=0$. The argument is the similar to the proof of Lemma \ref{lemma:marked_pt_1}.  

  Suppose that $u_{n}$ remains in a bounded set $K$. Let $\hbar$ be a small quantity of energy with the property that:
  \begin{equation*}
    E_{n}([s-1,s+1]\times S)\le \hbar\implies \sup_{t}\abs{\d u_{n}(s,t)}\le C,
  \end{equation*}
  where $C$ is independent of $u_{n}$ and $s$, and $E_{n}$ is the energy of $u_{n}$. The existence of constants $C$ and $\hbar$ follows easily from the mean-value property. 

  We say that a radius $r$ is $k$\emph{-special} if $E_{n}([0,r-1]\times S)\ge k\hbar$ holds infinitely often. Clearly there is a maximal $k$ which admits a $k$-special radius. Let us fix this maximal $k$ and a $k$-special radius $r$. Let $v_{n}(s,t)=u_{n}(s+r,t)$, and consider $v_{n}$ as defined on subsets of $[0,\infty)\times S$.
  \begin{figure}[H]
    \centering
    \begin{tikzpicture}
      \draw (0,0) rectangle (10,1);    
      \draw (2,0) --+ (0,1);
      \draw (3,0) --+ (0,1);
      \draw[shift={(0,-0.2)},<->] (0,0)--node[below]{$r-1$}+(2,0);
      \draw[shift={(0,-0.2)},<->] (3,0)--node[below]{domain of $v_{n}$}+(7,0);
    \end{tikzpicture}
  \end{figure}
  Maximality of $k$ and the choice of $\hbar$ implies $v_{n}$ has bounded derivative on compact subsets. We apply Arzel\`a-Ascoli to $v_{n}$ and conclude it converges on compact sets to a holomorphic limit $v_{\infty}:[0,\infty)\times S\to W$.

  The next step is a little trick to upgrade convergence on compact sets to uniform convergence. Let $\rho_{k}$, $k=1,2,\dots$ be an arbitrary sequence which converges to $+\infty$. For each $N\in \mathbb{N}$, consider the maximal integer $k=k(N)$ satisfying:
  \begin{equation*}
    n\ge N\implies \sup \set{\mathrm{dist}_{W}(v_{n}(s,t),v_{\infty}(s,t)):(s,t)\in [0,2\rho_{k}]\times S}\le k^{-1}.
  \end{equation*}
  By the aforementioned convergence on compact sets, $k(N)\to \infty$ as $N\to\infty$. We then define $\rho_{n}:=\rho_{k(n)}$. Then $v_{n}$ converges uniformly to $v_{\infty}$ on this region. Since $v_{\infty}$ is bounded and has finite energy, $v_{\infty}(s,t)$ converges uniformly in $t$ to a removable singularity $x_{+}$ as $s\to\infty$. It follows that the region $s\in [r+\rho_{n},r+2\rho_{n}]$ is very tame, and converges uniformly to this point $x_{+}$.
  \begin{figure}[H]
    \centering
    \begin{tikzpicture}[yscale=.5]
      \fill[pattern={Lines[angle=10]}, pattern color=black!40!white] (7,0)rectangle +(-4,1);
      \fill[pattern={Lines[angle=-40]}, pattern color=black!40!white] (7,0)rectangle +(-2,1);    
      \draw (0,0) rectangle (10,1) (3,0) --+ (0,1) (5,0)--+(0,1) (7,0)--+(0,1);
      \draw[shift={(0,-0.4)},<->] (0,0)--node[below]{$r$}+(3,0);
      \draw[shift={(0,0.4)},<->] (3,1)--node[above]{$2\rho_{n}$}+(4,0);
      \draw[dashed] (5,0)--+(0,-.5)node[below]{cut};
    \end{tikzpicture}
    \caption{The total shaded region is tame and converges to $v_{\infty}$, the heavily shaded region is very tame, and converges to $v_{\infty}(+\infty)$.}
  \end{figure}
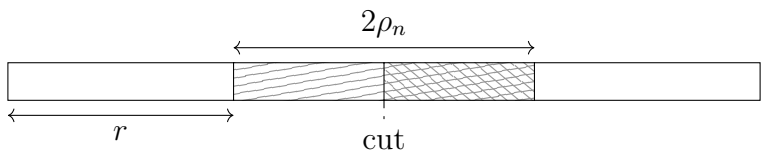
  This completes the proof at the left end. The right end is handled in the same way. 
\end{proof}

\subsubsection{Cutting lemma for holomorphic ends}
\label{sec:for_ends}
A similar cutting construction can be performed near the marked points, where we recall $u_{n}$ is allowed to have a singularity. Let us suppose that the underyling domain of $u_{n}$ converges with marked points. By the results of \S\ref{sec:coordinate_disks}, we can find convergent sequences of coordinate disks around the marked points, which we reparametrize via an appropriate exponential map to obtain a map defined on $[0,\infty)\times S$. 
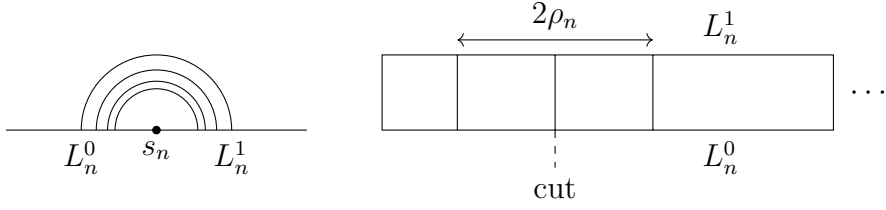
\begin{figure}[H]
  \centering
  \begin{tikzpicture}
    \draw (1,0) arc (0:180:1);
    \draw (0.8,0) arc (0:180:0.8);
    \draw (0.65,0) arc (0:180:0.65);
    \draw (0.55,0) arc (0:180:0.55);
    \draw (-2,0)--node(A)[draw,fill,inner sep=1pt,circle]{}node[pos=0.25,below]{$L^{0}_{n}$}node[pos=0.75,below]{$L^{1}_{n}$}(2,0);
    \node at (A) [below] {$s_{n}$};
    \begin{scope}[shift={(3,0)}]
      \draw (0,0) rectangle (6,1) (1,0)--+(0,1) (2.3,0)--+(0,1) (3.6,0)--+(0,1);
      \node at (6.5,0.5) {$\cdots$};
      \draw[<->] (1,1.2)--node[above]{$2\rho_{n}$}+(2.6,0);
      \draw[dashed] (2.3,0)--+(0,-0.5)node[below]{cut};
      \node at (4.5,0)[below] {$L^{0}_{n}$};
      \node at (4.5,1)[above] {$L^{1}_{n}$};
    \end{scope}
  \end{tikzpicture}
  \caption{We can think of a punctured neighbourhood of a marked point as a half-infinite end. Cutting the end corresponds to removing a subdisk.}
  \label{fig:concentric-necks}
\end{figure}
We have the following cutting lemma:
\begin{lemma}\label{lemma:end-cutting-lemma}
  Let $u_{n}:[0,\infty)\times S\to W$ be a bounded sequence of $J_{n}$-holomorphic maps with finite energy, with boundary on $L^{0}_{n}\cup L^{1}_{n}$, and suppose that $J_{n},L^{i}_{n}$ are convergent. Then there exist $r>0$ and $\rho_{n}\to\infty$ so that
  $u_{n}$ converges uniformly to a holomorphic limit on the region $[r,r+2\rho_{n}]\times S$.
\end{lemma}
\begin{proof}
  The argument is the exact same as the proof of Lemma \ref{lemma:neck-cutting-lemma}.
\end{proof}

\subsection{Boundary conditions for holomorphic necks}
\label{sec:boundary_conditions}
The preceding cutting results do not depend on the \emph{Lagrangian} aspects of our problem, except for the a priori estimates coming from the mean-value property (which use the totally real hypothesis). In the subsequent results, a sequence of holomorphic necks $u_{n}:[a_{n},b_{n}]\times [0,1]\to W$ should satisfy one of the following boundary conditions. Suppose that $u_{n}([a_{n},b_{n}]\times \set{0})\subset L^{0}_{n}$ and $u_{n}([a_{n},b_{n}]\times \set{1})\subset L^{1}_{n}$, and
\begin{enumerate}
\item $L^{0}_{n}=L^{1}_{n}$ (same Lagrangian boundary conditions),
\item $L^{0}_{n}\to L^{0}$ and $L^{1}_{n}\to L^{1}$, and $L^{0}\cap L^{1}$ is an isolated set, or
\item $L^{0}_{n},L^{1}_{n}$ converge to the same Lagrangian $L$ \emph{adiabatically}, as explained in \S\ref{sec:intro}.
\end{enumerate}

\subsection{High and low energy decomposition of long necks and ends}
\label{sec:low-energy-de}
Let $$u_{n}:\Sigma_{n}=[-R_{n}-\rho_{n},R_{n}+\rho_{n}]\times S\to W$$ be a sequence of holomorphic curves as above (i.e., bounded, finite energy, etc), and suppose that the ends of length $\rho_{n}$ converge uniformly to points $x_{-}$ and $x_{+}$.

Fix a quantity of energy $\hbar>0$. A sequence $u_{n}$ is said to be \emph{high energy} if it has energy at least $\hbar$, infinitely often. On the other hand, if the energy of $u_{n}$ is eventually less than $\hbar$, and the ends are very tame, we say that the sequence is \emph{low energy}.

The quantity $\hbar$ will need to be chosen small enough that the following \emph{energy quantization} result applies.
\subsubsection{Quantization of energy for necks}
\label{sec:quantization-neck}
Let $u_{n}$ be a sequence of necks, as above, which converges uniformly to $x_{\pm}$ on its ends.\footnote{We also suppose that the energy of $u_{n}$ on the necks also converges to zero, which certainly holds if, e.g., the second derivatives of $u_{n}$ is bounded on the $\bd C_{n}$ boundary of the neck. In all cases we consider, this will hold true, since $u_{n}$ extends across the $\bd C_{n}$ boundary in a controlled fashion. This additional assumption will be implicit in our arguments.}
\begin{theorem}\label{theorem:quantization}
  There exists a constant $\hbar>0$, depending only on (i) the compact set containing $u_{n}$, and (ii) the limits of the data $L^{0}_{n},L^{1}_{n},J_{n}$, so that:
  \begin{equation*}
    \limsup_{n}E_{n}<\hbar\implies \lim_{n}\sup_{s,t}\abs{\d u_{n}(s,t)}=0\implies \lim_{n}E_{n}= 0,
  \end{equation*}
  where $E_{n}$ is the $\omega$-energy of $u_{n}$.
\end{theorem}
\begin{proof}
  The proof is different in the case when $S=[0,1]$ and $S=\R/\Z$. The case when $S=[0,1]$ further splits into three cases depending on the three allowable boundary conditions. The result is a combination of Propositions \ref{prop:low-e-cyl}, \ref{prop:low-e-strips}, \ref{prop:low-e-strips-2}, and \ref{prop:low-e-strips-3}.
\end{proof}

\subsubsection{An energy decomposition lemma}
\begin{lemma}\label{lemma:energy-D}
  For sufficiently small $\hbar$, there exists a sequence of tame necks $\mathfrak{n}_{n}^{i}\subset [-R_{n},-\rho_{n},R_{n}+\rho_{n}]$, so that when we symmetrically cut $\Sigma_{n}$ along $\mathfrak{n}^{i}_{n}$, the resulting surface splits into an alternating concatenation of \emph{low energy} and \emph{high energy} pieces. Moreover, the high energy pieces of the domain converge strongly. The rightmost and leftmost regions will be low energy and have very tame ends.
  \begin{figure}[H]
    \centering
    \begin{tikzpicture}[scale=.4]
      \fill[pattern={Lines[angle=10]}, pattern color=black!40!white] (3,0)rectangle+(4,1);
      \fill[pattern={Lines[angle=-40]}, pattern color=black!40!white] (3,0)rectangle+(2,1);    
      \fill[pattern={Lines[angle=10]}, pattern color=black!40!white] (-5,0)rectangle+(1.5,1);
      \draw[pattern={Lines[angle=-40]}, pattern color=black!40!white] (-5,0)rectangle+(1.5,1);    
      \fill[pattern={Lines[angle=10]}, pattern color=black!40!white] (25,0)rectangle+(-1.5,1);
      \draw[pattern={Lines[angle=-40]}, pattern color=black!40!white] (25,0)rectangle+(-1.5,1);
      \fill[pattern={Lines[angle=10]}, pattern color=black!40!white] (17,0)rectangle+(-2,1);
      \fill[pattern={Lines[angle=-40]}, pattern color=black!40!white] (17,0)rectangle+(-4,1);    
      \draw (-5,0) rectangle (25,1) (3,0) --+ (0,1) (5,0)--+(0,1) (7,0)--+(0,1) (13,0) --+ (0,1) (15,0)--+(0,1) (17,0)--+(0,1);
      \path (-5,0)--+(30,0)node[pos=0.16,below]{low} node[pos=0.5,below]{high} node[pos=0.84,below]{low};
      \draw (5,0)--+(0,-.5)node[below]{cut};
      \draw (15,0)--+(0,-.5)node[below]{cut};
    \end{tikzpicture}
    \caption{Cutting the domain into high and low energy regions. The heavily shaded ends are very tame. The underlying domain of the high energy regions converge strongly with the ends induced by the cutting.}
  \end{figure}
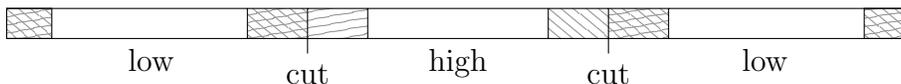
\end{lemma}
\begin{proof}
  We pick $\hbar$ small enough that the quantization result in Theorem \ref{theorem:quantization} applies. Then, if $\abs{\d u_{n}}$ converges to zero uniformly, the energy $E_{n}$ must converge to zero, and hence the whole domain is low energy, i.e., no cuts are necessary.

  Thus there exists a sequence $z_{n}$ so that $\abs{\d u_{n}(z_{n})}$ remains bounded below, after passing to a subsequence. Since $u_{n}$ converges to constants on the ends, $z_{n}=s_{n}+it_{n}$ remains far from the ends, i.e., we can find $r_{n}\to \infty$ so that $[s_{n}-r_{n},s_{n}+r_{n}]\times S$ is disjoint from the ends $[-R_{n}-\rho_{n},-R_{n}]\times S$ and $[R_{n},R_{n}+\rho_{n}]\times S$.

  Make two cuts in the region $[s_{n}-r_{n},s_{n}+r_{n}]\times S$, as shown below.
  \begin{figure}[H]
    \centering
    \begin{tikzpicture}[scale=.4]
      \fill[pattern={Lines[angle=10]}, pattern color=black!40!white] (3,0)rectangle+(2,1);
      \fill[pattern={Lines[angle=-40]}, pattern color=black!40!white] (3,0)rectangle+(4,1);    
      \fill[pattern={Lines[angle=10]}, pattern color=black!40!white] (-5,0)rectangle+(1.5,1);
      \draw[pattern={Lines[angle=-40]}, pattern color=black!40!white] (-5,0)rectangle+(1.5,1);    
      \fill[pattern={Lines[angle=10]}, pattern color=black!40!white] (25,0)rectangle+(-1.5,1);
      \draw[pattern={Lines[angle=-40]}, pattern color=black!40!white] (25,0)rectangle+(-1.5,1);
      \fill[pattern={Lines[angle=10]}, pattern color=black!40!white] (15,0)rectangle+(-4,1);
      \fill[pattern={Lines[angle=-40]}, pattern color=black!40!white] (15,0)rectangle+(-2,1);    
      \draw (-5,0) rectangle (25,1) (1,0)--+(0,1)node[above]{$s_{n}-r_{n}$} (17,0)--+(0,1.0)node[above]{$s_{n}+r_{n}$} (3,0) --+ (0,1) (5,0)--+(0,1) (7,0)--+(0,1) (11,0) --+ (0,1) (15,0)--+(0,1) (13,0)--+(0,1);
      \draw (5,0)--+(0,-.5)node[below]{cut};
      \draw (13,0)--+(0,-.5)node[below]{cut};
      \draw[shift={(0,-.4)},<->] (7,0)--node[below]{$2r$}(11,0);
      \draw[shift={(0,.4)},<->] (3,1)--node[above]{$2\rho_{n}'$}+(4,0);
      \draw[shift={(0,.4)},<->] (11,1)--node[above]{$2\rho_{n}'$}+(4,0);
    \end{tikzpicture}
  \end{figure}
  Similarly to the proof of Lemma \ref{lemma:neck-cutting-lemma}, we say $r$ is $k$-special if the energy of $[s_{n}-r,s_{n}+r]\times S$ is greater than $k\hbar$ infinitely often. Fix $r>0$ to be a $k$-special radius for the maximal $k$. As above, there is $\rho_{n}'$ so that $v_{n}$ converges uniformly on $[s_{n}\pm r,s_{n}\pm (r+\rho_{n}')]$ (all after passing to a subsequence).

  Since $r$ is finite, independently of $n$, the middle region converges strongly. Moreover, the left and right regions satisfy the hypotheses of this lemma, and hence the argument can be iterated.

  It remains only to prove that the middle region is high energy, i.e., $k>0$. This is straightforward; if not, then the quantization theorem would imply the derivative on the middle region converged to zero, which it does not. 

  Iterate the argument until there is not enough energy for the iteration to continue. By construction, the resulting decomposition alternates between low energy and high energy regions.
\end{proof}

\subsubsection{Energy decomposition for ends}
\label{sec:e-d-for-ends}
There is an analogous result for ends. Consider a sequence $u_{n}:[0,\infty)\times S\to W$, initially with a single very tame end identified with $[0,\rho_{n}]\times S$ which converges uniformly to a point $x_{-}$. The domain can be decomposed into an alternating sequence of low-energy and high-energy regions, ending with a low energy end. The domains of the high energy regions converge strongly (i.e., the part uncovered by the ends remains of bounded modulus).
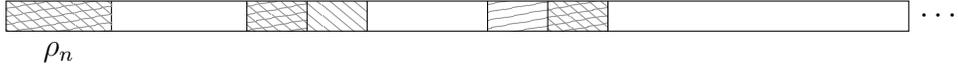
\begin{figure}[H]
  \centering
  \begin{tikzpicture}[scale=.4]
    \fill[pattern={Lines[angle=10]}, pattern color=black!40!white] (15,0)rectangle+(-4,1);
    \fill[pattern={Lines[angle=-40]}, pattern color=black!40!white] (15,0)rectangle+(-2,1);    
    \fill[pattern={Lines[angle=10]}, pattern color=black!40!white] (3,0)rectangle+(2,1);
    \fill[pattern={Lines[angle=-40]}, pattern color=black!40!white] (3,0)rectangle+(4,1);

    \foreach \x in {-1.5,3,5,7,11,13,15} {
      \draw (\x,0)--+(0,1);
    }
    \fill[pattern={Lines[angle=10]}, pattern color=black!40!white] (-5,0)rectangle+(3.5,1);
    \fill[pattern={Lines[angle=-40]}, pattern color=black!40!white] (-5,0)rectangle+(3.5,1);
    \draw (-5,0) rectangle (25,1);
    \draw (-5,0)+(1.75,0)node[below]{$\rho_{n}$};
    \node at (25,0.5)[right]{$\cdots$};
  \end{tikzpicture}
  \caption{A sequence of ends decomposed into a concatenation of a low-high-low sequence. As above, the construction is an iterative process.}
\end{figure}

\subsubsection{Convergence of low energy regions}
\label{sec:low-energy-convergence}
Let $\hbar$ be small enough that the quantization results of \S\ref{sec:quantization-neck} hold. Suppose that $u_{n}$ is a bounded low energy sequence of necks or ends, with ends of length $\rho_{n}$, i.e., $u_{n}$ has limiting energy less than $\hbar$, and $u_{n}$ converges uniformly to points $x_{-}$ and $x_{+}$ at its ends (there is only $x_{-}$ in when $u_{n}$ is defined on $[0,\infty)\times S$).

\begin{theorem}
  We have three cases for the convergence of low energy necks and ends, dependending on the boundary conditions $L^{0}_{n},L^{1}_{n}$.
  \begin{enumerate}
  \item If $S=\R/\Z$, then $x_{-}=x_{+}=x$ and $u_{n}$ converges uniformly to $x$ on the entire domain. 
  \item If $S=[0,1]$ and $L^{0}_{n}=L^{1}_{n}$ for all $n$, then $x_{-}=x_{+}=x$ and $u_{n}$ converges uniformly to $x$ on the entire domain.
  \item If $S=[0,1]$ and $L^{0}_{n}$ converges to $L^{0}$ and $L^{1}_{n}$ converges to $L^{1}$, and $L^{0}\cap L^{1}$ is a finite set, then $x_{-}=x_{+}=x$ is an intersection point in $L^{0}\cap L^{1}$ and $u_{n}$ converges uniformly to $x$ on the entire domain.
  \item If $S=[0,1]$ and $L^{0}_{n}=L_{\epsilon_{n}\mathfrak{a}_{n}}$ and $L^{1}_{n}=L_{\epsilon_{n}\mathfrak{b}_{n}}$, i.e., the boundary conditions are of adiabatic type, then $u_{n}$ converges to a \emph{broken flow line} for the limiting Morse function $f_{\infty}$ satisfying $\lim_{n}\mathfrak{b}_{n}-\mathfrak{a}_{n}=\lim_{n}\d f_{n}=\d f_{\infty}$.
  \end{enumerate}
\end{theorem}
\begin{proof}
  The proofs are given in Proposition \ref{prop:low-e-cyl}, \ref{prop:low-e-strips}, \ref{prop:low-e-strips-2}, \ref{prop:low-e-strips-3}.  
\end{proof}
In particular, we observe that any low-energy neck is \emph{a priori} very tame, and hence can be glued\footnote{See Figure \ref{fig:tame-very-tame} for the gluing of tame and very tame ends.} onto any adjacent tame end \emph{except in the adiabatic boundary conditions case}. Thus these adiabatic low-energy necks should be interpreted as stable objects in the limit of a sequence, i.e., they cannot simply be absorbed into the ends of another component.

\subsubsection{Digression on convergence to flow lines}
\label{sec:digression-flow-line}
We introduce a few of the concepts used in \S\ref{sec:low-energy-strips} concerning the convergence to broken flow lines.

In the adiabatic case, recall that $L^{0}_{n}=L_{\epsilon_{n}\mathfrak{a}_{n}}$ and $L^{1}_{n}=L_{\epsilon_{n}\mathfrak{b}_{n}}$, where $\mathfrak{a}_{n},\mathfrak{b}_{n}$ converge, $\mathfrak{b}_{n}-\mathfrak{a}_{n}=\d f_{n}$ converges to $\d f_{\infty}$ for a Morse function $f_{\infty}$, and $\epsilon_{n}>0$ converges to $0$.

Since Morse functions are non-zero somewhere, $\epsilon_{n}$ is uniquely determined from $L^{0}_{n},L^{1}_{n}$ by the requirement that $\mathfrak{a}_{n},\mathfrak{b}_{n}$ converge, up to the equivalence relation which requires that $\log(\epsilon_{n}'/\epsilon_{n})$ remains bounded. In particular, $c^{-1}\epsilon_{n}'<\epsilon_{n}<c\epsilon_{n}'$ for some $c>0$ holds for any two sequences.
\begin{defn}
  A sequence $u_{n}:[a_{n},b_{n}]\times S\to W$ converges \emph{uniformly to a flow line} $v_{\infty}$, with rescaling parameter $\epsilon_{n}\to 0$, if there is a sequence $s_{n}$ so that:
  \begin{equation*}
    \lim_{n\to \infty}\sup_{s,t}\mathrm{dist}_{W}(u_{n}(s,t),v_{\infty}(\epsilon_{n}(s-s_{n})))=0.
  \end{equation*}
  We distinguish a few cases:
  \begin{enumerate}
  \item $\epsilon_{n}(b_{n}-a_{n})$ converges to a finite number $\ell$. In this case $v_{\infty}$ is defined on an interval of length $\ell$.
  \item $\epsilon_{n}(b_{n}-a_{n})$ diverges to $\infty$, but one of $\epsilon_{n}(a_{n}-s_{n})$ or $\epsilon_{n}(b_{n}-s_{n})$ converges to a finite number. Then $v_{\infty}$ is defined on a half-infinite interval. We can replace $s_{n}$ by either $a_{n}$ or $b_{n}$ in this case, so that the half-infinite interval starts at $0$.
  \item The final case is when $v_{\infty}$ is defined on $\R$.
  \end{enumerate}
\end{defn}
\begin{theorem}\label{theorem:main_morse}
  Let $u_{n}:[-R_{n}-\rho_{n},R_{n}+\rho_{n}]\times [0,1]$ be a low energy strip with adiabatic boundary conditions, as above. If $\epsilon_{n}(R_{n}+\rho_{n})$ is bounded above, then, after passing to a subsequence, $u_{n}$ converges uniformly to a finite length flow line, with rescaling parameter $\epsilon_{n}$.

  Otherwise, there exist finitely many disjoint necks $\mathfrak{n}^{i}_{n}$ contained in $[-R_{n},R_{n}]$ so that if we symmetrically cut along these necks, we obtain an alternating sequence of sub-strips
  \begin{equation*}
    [-R_{n}-\rho_{n},R_{n}+\rho_{n}]\times [0,1]=\Sigma^{-}_{n}\cup \sigma^{0}_{n}\cup \Sigma^{1}_{n}\cup \dots\cup \Sigma^{k}_{n}\cup \sigma^{k}_{n}\cup \Sigma^{+}_{n},
  \end{equation*}
  where the restriction of $u_{n}$ to:
  \begin{enumerate}
  \item $\sigma^{k}_{n}$ converges uniformly to a constant map valued at a critical point $y^{i}$ for $f_{\infty}$. In particular, $\sigma^{k}_{n}$ is very tame.
  \item $\Sigma^{i}_{n}=[a^{i}_{n},b^{i}_{n}]\times [0,1]$ converges to an infinite flow line, with rescaling parameter $\epsilon_{n}$, joining $y^{i-1}$ and $y^{i}$, when we use $\psi^{i}_{n}(s,t)=\epsilon_{n}(s-\frac{1}{2}(a^{i}_{n}+b^{i}_{n}))$,
  \item $\Sigma^{-}_{n}=[-R_{n}-\rho_{n},b^{-}_{n}]\times [0,1]$ converges to a (positively) half-infinite flow line joining a point $x^{-}$ to $y^{0}$, while $\Sigma^{+}_{n}=[a_{n}^{+},R_{n}+\rho_{n}]\times [0,1]$ converges to a (negatively) half-infinite flow line joining $y^{k}$ to a point $x^{+}$, both with rescaling parameter $\epsilon_{n}$.
  \end{enumerate}
  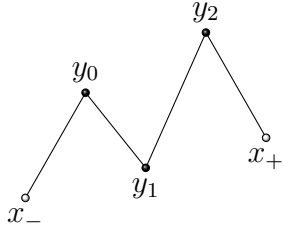
\begin{figure}[H]
    \centering
    \begin{tikzpicture}
      \begin{scope}
        \clip (-.5,-.4) rectangle (3.6,2.7);
        \path (0,0)-- coordinate[pos=0](x-) coordinate[pos=0.4](y2)coordinate [pos=0.8](y4)(4,1);
        \path (0,1)--coordinate[pos=0.2](y1)coordinate[pos=0.6](y3)coordinate[pos=1](x+)(4,3);
        \draw (x-)node[below]{$x_{-}$}--(y1)node[above]{$y_{0}$}--(y2)node[below]{$y_{1}$}--(y3)node[above]{$y_{2}$}--(y4)node[below]{$x_{+}$};
        \path[every node/.style={draw,circle,inner sep=1pt,shading=ball,ball color=black}] (x-)node[ball color=white]{}--(y1)node{}--(y2)node{}--(y3)node{}--(y4)node[ball color=white]{};
      \end{scope}
    \end{tikzpicture}
    \caption{A broken gradient flow line. The points $x_{\pm}$ need not be critical points.}
    \label{fig:broken-flow}
  \end{figure}
\end{theorem}
The proof is given in \S\ref{sec:proof_main_morse}.

\subsection{Formation of bubble trees}
\label{sec:tame-ends-2}
Let $(u_{n},\Sigma_{n},\Theta_{n},\mathfrak{e}_{n})$ converge with marked points to $(\Sigma,\Theta)$, and have tame ends. Suppose that $u_{n}$ satisfies the bounded, finite energy, and boundary condition assumptions from \S\ref{sec:boundary_conditions}.

It is still possible that $u_{n}$ fails to converge uniformly (along any subsequence), due to a \emph{bubbling phenomenon}. The derivative of $u_{n}$ may blow up, and this will lead to the formation of a bubble, i.e., a small region on $u_{n}$ being expanded to have a large diameter in $W$. This phenomenon will obviously prevent uniform convergence to a map defined on $\Sigma$.

Even in cases where one expects to have derivative bounds, e.g., in exact symplectic manifolds, one has bubbling at the marked points. This is because the map $u_{n}$ is not required to be defined at the marked points.
\begin{figure}[H]
  \centering
  \begin{tikzpicture}
    \draw (-2,0)--(2,0);
    \draw (0,-1) circle (1);
    \draw (0,-2.5) circle (.5);
    \draw (0,-2.5) + (-20:1) circle (.5);
    \path[every node/.style={draw,circle,fill,inner sep=1pt}] (0,0)node{}--(0,-2)--(0,-3)node{};
  \end{tikzpicture}
  \caption{Bubbling at a boundary marked point.}
  \label{fig:bubbling-switch}
\end{figure}
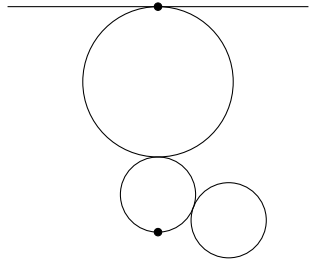

\subsubsection{Bubbling}
\label{sec:addition-gamma}

Continuing with the set-up of the previous section, let $\psi_{n}:(\Sigma_{n},\Theta_{n})\to (\Sigma,\Theta)$ realize the convergence of the underlying domain with marked points.

\begin{lemma}
  After adding a finite set $\Gamma$ to $\Theta$, we can ensure that $u_{n}\circ \psi_{n}^{-1}$ has bounded derivative on the complement of any neighbourhood of $\Gamma\cup \Theta$ (allowing subsequences). 
\end{lemma}
\begin{proof}
  This follows from Hofer's bubbling argument in \S\ref{sec:hofer-bubbling}, and the fact that the ends are tame (and hence converge in $C^{1}$, by assumption).
\end{proof}
Note that this induces marked points $\Gamma_{n}:=\psi_{n}^{-1}(\Gamma)$ so that $(\Sigma_{n},\mathfrak{e}_{n},\Theta_{n}\cup \Gamma_{n},\psi_{n})$ converges to $(\Sigma,\Theta\cup \Gamma)$. Fix convergent sequence of holomorphic disks or half-disks centered on a point $\zeta_{n}\in \Theta_{n}\cup \Gamma_{n}$. By a small modification of the map $\psi_{n}$, we may assume that:
\begin{enumerate}
\item $\zeta_{n}=z_{n}\in \Gamma_{n}^{\mathrm{int}}$ locally maximizes the derivative $\abs{\d u_{n}}$, and
\item $\zeta_{n}=s_{n}\in \bd \Gamma_{n}$ can be chosen so that $z_{n}=s_{n}+it_{n}$ locally maximizes the derivative for $t_{n}\to 0$.
\end{enumerate}

\begin{defn}
  For every puncture $\zeta$ in $\Theta\cup \Gamma$, let $$E(\zeta)=\inf_{U}\limsup_{n\to\infty}E_{n}(\psi_{n}^{-1}(U)),$$ where $U$ ranges over open neighbourhoods of $\zeta$. Let $\hbar>0$ be a small quantity of energy so that the quantization results of \S\ref{sec:quantization-neck} hold. The \emph{bubbling energy level} of the sequence is the maximal $k$ for which: $$\sum_{\zeta\in \Gamma^{\mathrm{int}}} E(\zeta)+2\sum_{\zeta\in \bd \Gamma}E(\zeta)+3\sum_{\zeta\in \Theta}E(\zeta)\ge k\hbar.$$ The reason why we include the points with weights is so that the iterative argument strictly decreases the bubbling energy level at each step (and hence must terminate).
\end{defn}
\begin{lemma}
  Let $(\Sigma_{n},\mathfrak{e}_{n},\Theta_{n}\cup \Gamma_{n},\psi_{n},u_{n})$ be as above. Then, after passing to a subsequence and cutting along finitely many concentric necks around the points in $\Theta_{n}\cup \Gamma_{n}$, the resulting sequence has bubbling energy level strictly less than initial sequence.

  In particular, after making finitely many cuts, (resulting in a new sequence with new ends), the bubbling energy level is zero.
\end{lemma}
\begin{proof}
  The argument is similar (on a formal level) to the convergence results for marked points in \S\ref{sec:application_marked_points}. We suppress mentioning when we pass to subsequences.
  
  Suppose the bubbling energy level is non-zero. Then there must be at least one puncture $\zeta$ in $\Theta\cup \Gamma$ whose bubbling energy level is non-zero. We consider three cases:
  \begin{enumerate}
  \item $\zeta=\psi_{n}(z_{n})$ is an interior puncture in $\Gamma$, and $z_{n}$ maximizes the derivative on $\psi_{n}^{-1}(D)$, where $D$ is a sufficiently small disk around $\zeta$ (of fixed radius),
  \item $\zeta=\psi_{n}(s_{n})$ is a boundary puncture in $\Gamma$, and $z_{n}=s_{n}+it_{n}$ maximizes the derivative, for $t_{n}\to 0$,
  \item $\zeta\in \Theta$. 
  \end{enumerate}
  Implicitly we refer to convergent sequences of coordinate disks/half-disks, especially when referring to the real and imaginary parts $s_{n},t_{n}$ in (ii).

  First we explain the iteration in case (i). Let us suppose that the contribution to the bubbling energy level due to $\zeta$ is $k>0$.

  Fix some notation: let $c_{n}$ be a convergent sequence of coordinate disks centered on the points in $z_{n}$. Think of $u_{n}$ as a family of functions on $D(1)$. By assumption, we can pick $c_{n}$ so that $u_{n}$ has maximal derivative at $0$, say of size $R_{n}$. Let $v_{n}=u_{n}(R_{n}^{-1}z),$ so that $v_{n}$ is defined on a sequence of disks which exhaust $\C$. Then $v_{n}$ converges uniformly on compact subsets to a \emph{non-constant} holomorphic plane $v_{\infty}$, which has a removable singularity at $\infty$. See \S\ref{sec:hofer-bubbling} for more details.
  
  By the same cutting arguments given above, we can pick $\rho_{n}\to \infty$ so that $v_{n}$ converges \emph{uniformly} on $D(e^{4\pi\rho_{n}})$. In other words, using the holomorphic parameterization $(s+it)\mapsto e^{2\pi (s+it)}$, then the neck $s\in [0,2\rho_{n}]$ is tame and $s\in [\rho_{n},2\rho_{n}]$ is very tame. Without loss of generality, shrink $\rho_{n}$ so that $R_{n}^{-1}e^{4\pi\rho_{n}}$ converges to zero.

  After shrinking $\rho_{n}$ further, suppose that $$N_{n}:=D(e^{-\rho_{n}})\setminus D(R_{n}^{-1}e^{2\pi \rho_{n}})$$
  has very tame ends $D(e^{-\rho_{n}})\setminus D(e^{-2\rho_{n}})$ and $D(R_{n}^{-1}e^{4\pi \rho_{n}})\setminus D(R_{n}^{-1}e^{2\pi \rho_{n}})$. The set-up is similar to that of Figure \ref{fig:fixing_int_bd}.
  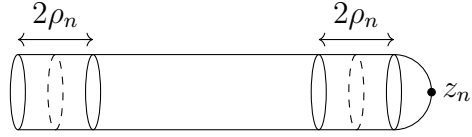
\begin{figure}[H]
    \centering
    \begin{tikzpicture}[xscale=-1]
      \draw (5,0)--+(-5,0) coordinate (X) arc (90:270:.5) node(A)[pos=0.5,draw,circle,inner sep=1pt,fill]{}coordinate (Y)--+(5,0);
      \path (X)--coordinate(Z)(Y);
      \node at (A) [right]{$z_{n}$};
      \draw (Z) circle (0.1 and 0.5) (Z)+(5,0) circle (0.1 and 0.5) (Z)+(1,0) circle (0.1 and 0.5) (Z)+(4,0) circle (0.1 and 0.5);
      \draw[dashed] (Z)+(4.5,0) circle (0.1 and 0.5) (Z)+(.5,0) circle (0.1 and 0.5);
      \draw[<->] (X)+(0,0.2)--node[above]{$2\rho_{n}$}+(1,0.2);
      \draw[<->] (X)+(4,0.2)--node[above]{$2\rho_{n}$}+(5,0.2);
    \end{tikzpicture}
    \caption{Making cuts to isolate where the derivative blows up. We make the cuts at the dashed circles.}
    \label{fig:making-cuts-2}
  \end{figure}
  Referring to Figure \ref{fig:making-cuts-2}, let $\Sigma^{0}_{n},\Sigma^{1}_{n},\Sigma^{2}_{n}$ (left to right) denote the sequences of compact partial domains, with new ends of length $\rho_{n}$, resulting from the cutting process. It is clear that $\Sigma^{0}_{n}$ has tame ends and converges to $\Sigma\setminus \zeta$ (with the other marked points unaffected). By construction, $\Sigma^{2}_{n}$ and $u_{n}$ converges \emph{uniformly} to a non-constant holomorphic plane, with energy at least $\hbar$, when we use the $\psi_{n}$ which rescales its domain by $R_{n}$. The uniform convergence implies the $\Sigma^{2}_{n}$ sequence no longer contributes to the bubbling energy level.

  By definition, the limiting energy of $D(e^{-\rho_{n}})$ converges to some number in the interval $[k\hbar,k\hbar+\hbar)$. It follows that the limiting energy of the long neck $\Sigma_{n}^{1}$ converges to a number \emph{strictly less than} $k\hbar$.

  We apply the high-low energy decomposition from \S\ref{sec:low-energy-de} to $\Sigma_{n}^{1}$, and thereby conclude that the domain can be decomposed into an alternating sequence of low and high energy regions. None of the high energy regions have marked points. However, applying Hofer's bubbling argument again, we conclude a new set $\Gamma_{n}'$ contained in the high energy regions, so that the derivative of $u_{n}$ is bounded\footnote{We require using metrics which converge to complete metrics on the limiting domains of the high energy regions (i.e., the translation invariant metric).} on the complement of neighbourhood of $\Gamma_{n}'$. In particular, it makes sense to talk about bubbling energy level after cutting.

  Clearly the limiting total energy of all the high energy regions is less than $k\hbar$, and in particular, the contribution to the bubbling energy level is strictly less than $k$. It follows that we have decreased the bubbling energy level by at least $1$.

  The argument in the other cases is similar, but more notationally involved. Essentially, one does the same rescaling argument. In case (ii), one can perform the rescaling and cutting argument near $\zeta\in \bd\Gamma$, and conclude that the resulting sequence of half-disks either converges uniformly, or has unbounded derivative near an \emph{interior point}. Because we have doubly weighted the energy contributed by points in $\bd\Gamma$, we conclude that the bubbling energy level will decrease by at least $1$. The set-up is similar to Figure \ref{fig:fixing_int_bd}. We leave the details to the reader.

  In case (iii), we apply the result of \S\ref{sec:e-d-for-ends}. As in the other two cases, consider a punctured neighbourhood of $\zeta_{n}=\theta_{n}$ as an end $[0,\infty)\times S$. We can decompose this end into an alternating sequence:
  \begin{equation*}
    \Sigma^{0}_{n}\cup \sigma^{0}_{n}\cup \dots \cup \Sigma^{N}_{n}\cup \sigma^{N}_{n},
  \end{equation*}
  where $\Sigma^{0}_{n}=[0,\rho_{n})\times S$ is a new end of $\Sigma_{n}$, $\sigma^{i}_{n}$ is a low-energy region, while $\Sigma^{j}_{n}$ is high energy (whose underlying domains converge and $u_{n}$ has tame ends). In particular, the final region $\sigma^{N}_{n}$ is low energy. Since this final region contains the marked point $\theta_{n}$, we conclude that the bubbling energy level decreases (since we have \emph{triply} weighted the bubbling energy of points in $\Theta_{n}$). In other words, all of the new bubbling energy formed this process is supported by the new points $\Gamma'$ added to the high energy regions. Since the $\Gamma'$ points have weighting either $1$ or $2$, we conclude the bubbling energy level decreases. This completes the proof.  
\end{proof}

\subsubsection{Sequences with zero bubbling energy}
\label{sec:sequences-with-zero-bubbling-energy}
It remains to analyze sequences with zero bubbling energy. At punctures where the incident Lagrangian boundary conditions are adiabatic, it is possible that broken flow lines appear in the limit. Each puncture $\theta\in \bd\Theta$ has a sequence of pairs of incident Lagrangians $L^{0}_{n}$ and $L^{1}_{n}$. After passing to a subsequence, suppose that the pair $L^{0}_{n},L^{1}_{n}$ remains in one of the three allowable boundary conditions, see \ref{sec:boundary_conditions}. If the pair is adiabatic, then $\theta$ is called an adiabatic marked point.

\begin{lemma}
  Let $(u_{n},\Sigma_{n},\mathfrak{e}_{n},\Theta_{n})$ be a sequence of compact partial domains with tame ends $\mathfrak{e}_{n}$, maps $u_{n}$, and marked points $\Theta_{n}$, and suppose that the derivative of $u_{n}$ is uniformly bounded (i.e., there is no bubbling energy). After making cuts near each adiabatic boundary marked point, the resulting sequence converges uniformly to a holomorphic map with broken flow lines attached at the adiabatic marked points. 
\end{lemma}
\begin{proof}
  For every puncture $\theta$, pick convergent sequences of disks or half-disks around $\theta_{n}$, and identify these neighbourhoods with infinite strips $[0,\infty)\times S$, where $S=[0,1]$ or $S=\R/\Z$.

  Using the cutting lemmas, decompose this end into two pieces $$\Sigma_{0}\cup \Sigma_{1}=[0,r+\rho_{n}]\times S\cup [r+\rho_{n},\infty)\times S,$$ where $\Sigma_{0}$ has $[r,r+\rho_{n}]\times S$ as a tame end and $\Sigma_{1}$ has $[r+\rho_{n},r+2\rho_{n}]\times S$ as a very tame end. Since the bubbling energy level is known to be zero, $\Sigma_{1}$ is low-energy. There are now two cases to consider.

  If $S=\R/\Z$ or the incident boundary conditions are non-adiabatic, then $\Sigma_{1}$ is a priori very tame, and hence can be concatenated onto the tame end of $\Sigma_{0}$, while keeping the ends tame. Thus, in the non-adiabatic cases, making the cuts at $\set{r+\rho_{n}}\times S$ was a superfluous operation.

  When the incident boundary conditions \emph{are} adiabatic, then the $\Sigma_{1}$ region converges to a broken Morse flow line which starts at the removable singularity $x=\lim_{n\to\infty}(u_{n}(r+\rho_{n},t))$, as explained in \S\ref{sec:digression-flow-line}. For each $n$ fixed, the limit $\lim_{s\to\infty}(u_{n}(s,t))$ exists and is a critical point of the Morse function $f_{n}$, because the intersection points between $L_{0}^{n}\cap L^{1}_{n}$ are in bijection with critical points of $f_{n}$. Since $f_{n}$ converges to the function $f_{\infty}$, any sequence of critical points of $f_{n}$ has a subsequence which converges to a critical point of $f_{\infty}$. After taking appropriate subsequences, the $\Sigma_{1}$ regions converge to broken flow lines which terminate at these limiting critical points.

  Removing the $\Sigma_{1}$ regions for each adiabatic puncture, we are left with a sequence of holomorphic curves which converges uniformly to a limiting holomorphic map (taking subsequences as required). The new ends introduced at the adiabatic punctures have tame ends which limit to removable singularities. As explained above, each $\Sigma_{1}$ region converges to a broken flow line which starts at the corresponding removable singularity and terminates at a criticial point. This completes the proof.
\end{proof}

\begin{figure}[H]
  \centering
  \begin{tikzpicture}
    \coordinate (A) at (2,0);
    \coordinate (B) at (120:2);
    \coordinate (C) at (260:2);
    \draw (A) to[out=60,in=30] (B) (B)to[out=210,in=150] (C) (C)to[out=-30,in=240] (A);
    \path[every node/.style={draw,circle,inner sep=1pt,fill}] (A)node{}--(B)node{}--(C)node{};
    \draw[decoration={random steps,segment length=1.5mm},decorate,every node/.style={draw,circle,inner sep=1pt,fill}] (A)--+(-20:1)node{} (B)--+(150:1)node{} (C)--+(-110:1)node{};
  \end{tikzpicture}
  \caption{A sequence of maps with zero bubbling energy forms half-infinite broken flow lines at its marked points. The starting points of the broken flow lines need not by critical points of the limiting Morse functions, but the terminal points of the flow lines are guaranteed to be critical points.}
  \label{fig:zero-bubbling-energy}
\end{figure}
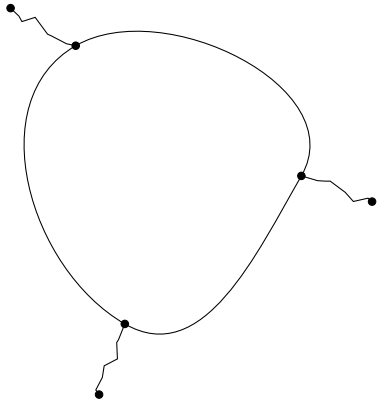

\subsection{Conclusion of the compactness statement}
\label{sec:conclusion_compactness}
We conclude by tying together and summarizing the results of this section. This will complete the proof of Theorem \ref{theorem:main_theorem}.

Suppose the data of $(\Sigma_{n},u_{n},\Theta_{n})$ where $\Sigma_{n}$ is a compact Riemann surface with boundary, and $u_{n}$ is defined on $\Sigma_{n}\setminus \Theta_{n}$. Assume that $u_{n}$ has uniformly bounded energy, $\abs{\Theta_{n}}$ and $\mathrm{X}(\Sigma_{n})$ are bounded above, and $u_{n}$ takes values in a compact set $K$. Consider a degenerating sequence of Lagrangians $L^{i}_{n}\to L^{\pi(i)}$, an convergent sequence of $\omega$-tame complex structures $J_{n}\to J$ (so that $J_{n}|L^{j}=J|L^{j}$ for each limit Lagrangian $L^{j}$). 

Throughout the argument, we suppress mentioning the passage to subsequences.

Recall from \S\ref{sec:intro} that the desired limit object is a generalized holomorphic curve, i.e., a graph $\Gamma$ whose vertices are labelled by holomorphic curves, and whose edges are labelled by either points (nodes) or gradient flow lines. See Definition \ref{defn:generalized_curve} for the details.

\subsubsection{Convergence of the underyling domain}
\label{sec:conclusion-1}
Consider two cases, either (i) the sequence of domains $\Sigma_{n}$ can be analyzed using the unique hyperbolic metrics rending the boundary $\bd\Sigma_{n}$ geodesic, or (ii) the sequence of domains $\Sigma_{n}$ satisfies $\mathrm{X}(\Sigma_{n})\ge 0$. In case (i), apply the results of \S\ref{sec:compactness_for_domains}, while the cases in (ii) are handled in ad-hoc fashion (either $\Sigma_{n}$ is a sequence of annuli, tori, disks, or spheres). In either case, we conclude the existence of a disjoint sequence of necks $\mathfrak{n}_{1}\cup \dots\cup \mathfrak{n}_{k}$ so that if we cut along the center of these necks, the resulting sequence of compact partial domains with ends converges strongly to a limiting domain.

\begin{remark}
  We have yet to discuss the case when $\Sigma_{n}$ is a sequence of tori or annuli. In the former case, we consider $\Sigma_{n}$ as a point in the topological space of lattices modulo the action of $\mathrm{SL}(2,\Z)$, while in the case of annuli we consider $\Sigma_{n}$ as a point in $(0,\infty)$ via the ``modulus'' function. In either case, it suffices to make a single cut (or no cuts at all). The details are left to the reader.
\end{remark}

\subsubsection{Convergence of the marked points}
\label{sec:conclusion_marked}
Consider a sequence of domains $\Sigma_{n}$ with the necks $\mathfrak{n}_{i}$ and marked points. First arrange that the marked points $\Theta_{n}$ remain disjoint from the necks, as follows. Given a particular neck $\mathfrak{n}_{i}^{n}$, replace $\mathfrak{n}^{i}_{n}$ by a collection of sub-necks $\mathfrak{m}^{1}_{n}\cup \dots\cup \mathfrak{m}^{\ell(i)}_{n}$ so that the complement $\mathfrak{n}^{i}_{n}\setminus(\mathfrak{m}_{n}^{1}\cup \dots\cup \mathfrak{m}^{\ell(i)}_{n})$ has uniformly bounded modulus and contains all the marked points in $\mathfrak{n}^{i}_{n}$. Each neck $\mathfrak{m}^{j}_{n}$ has length tending to infinity as $n\to\infty$. This can be arranged with only finitely many necks, following an argument similar to the one in \S\ref{sec:application_marked_points}. Passing to a subsequence if necessary, the number of necks remains constant, and the number of marked points in each complementary region remains fixed as well.

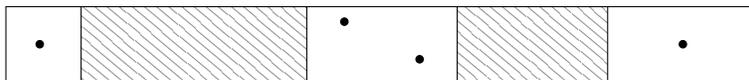
\begin{figure}[H]
  \centering
  \begin{tikzpicture}
    \fill[pattern={Lines[angle=-40]},pattern color=black!40!white] (1,0) rectangle (4,1);
    \fill[pattern={Lines[angle=-40]},pattern color=black!40!white] (6,0) rectangle (8,1);
    \draw (0,0) rectangle (10,1);
    \foreach \x in {1,4,6,8} {
      \draw (\x,0)--+(0,1);
    }
    \path[every node/.style={draw,circle,inner sep=1pt,fill}] (0,0.5)--+(0.45,0)node{}--+(4.5,0.3)node{}--+(5.5,-0.2)node{}--+(9,0)node{};
  \end{tikzpicture}
  \caption{Replacing the neck $\mathfrak{n}^{i}_{n}$ by sub-necks better adapted to the marked points. The new necks are shown as shaded regions.}
  \label{fig:replacing-1}
\end{figure}

Using the results of \S\ref{sec:application_marked_points}, add additional necks to the collection of all the $\mathfrak{m}^{j}_{n}$, so that the compact partial domain resulting from cutting converges strongly \emph{with marked points} to a limiting domain $(\Sigma,\Theta)$. Henceforth, relabel the collection of all necks by $\mathfrak{n}^{1}_{n},\dots,\mathfrak{n}^{k}_{n}$.

\subsubsection{Digression on the formation of the limit graph}
\label{sec:formation_of_graph_1}

Let $(\Sigma_{n}',\Theta_{n},\mathfrak{e}_{n})$ be the sequence resulting by cutting the necks $\mathfrak{n}_{n}^{i}$ along their centers. This domain converges strongly with marked points to a limit domain $(\Sigma,\Theta)$. Implicitly, we have fixed a subsequence by this stage in the argument. 

Form a preliminary graph $\Gamma_{\mathrm{dom}}$ describing the limit of the domain, as follows. Set:
\begin{enumerate}
\item $V(\Gamma_{\mathrm{dom}})=\pi_{0}(\Sigma)$, i.e., the connected components of the limit, 
\item $E_{\mathrm{int}}(\Gamma_{\mathrm{dom}})$ the set of punctures of $\Sigma$ modulo the \emph{nodal involution}, and
\item $E_{\mathrm{ext}}(\Gamma_{\mathrm{dom}})=\Theta$.
\end{enumerate}
Since the original sequence of domains $\Sigma_{n}$ is compact, there is a bijection between $\mathfrak{n}\sqcup \mathfrak{n}$ and the ends $\mathfrak{e}_{n}$ of $\Sigma_{n}'$. Since the ends $\mathfrak{e}_{n}$ converge to the punctures of $\Sigma$, we conclude a natural involution on the punctures of $\Sigma$; i.e., two punctures are identified if they correspond to two ends which are glued to make one of the necks in $\mathfrak{n}_{n}$.

An edge $e$ in $\Gamma_{\mathrm{dom}}$ is connected to a vertex $v$ if (i) $e$ corresponds to a non-compact end of $v$, or (ii) $e$ is a marked point on $v$. Note that there is a bijection between $\mathfrak{n}_{n}$ and the set of interior edges of $\Gamma_{\mathrm{dom}}$.

\subsubsection{Convergence of the maps}
\label{sec:convergence_of_maps}

Consider the maps $u_{n}$. Apply the results of \S\ref{sec:low-energy-de} to decompose each neck $\mathfrak{n}^{i}_{n}$ into a sequence of low and high energy regions. The low-energy regions either converge to broken flow lines joining two removable singularities of the limiting map, or converge uniformly to a point.

After cutting, there remains a collection of compact partial domains with tame ends. As explained in \ref{sec:tame-ends-2}, by cutting along additional necks, the derivative of $u_{n}$ is uniformly bounded. Together with the tameness of the ends, it is straightforward to apply Arzel\`a-Ascoli to conclude that a subsequence converges uniformly to a holomorphic limit defined on the limiting domain. As explained previously, there may be half-infinite broken flow lines attached at the adiabatic marked points.

\subsubsection{Constructing the limiting generalized holomorphic curve}
\label{sec:graph_construction}
Let us be a bit more explicit about the argument in \S\ref{sec:convergence_of_maps}, with special care to the construction of the limiting graph. First let us focus on an edge in $\Gamma_{\mathrm{dom}}$, which corresponded to a sequence of necks. To handle the degeneration of the maps, cut the neck into an alternating sequence of low and high energy regions. By construction, we guarantee at least one low energy region.

Denote by $\Gamma'$ the graph resulting from Figure \ref{fig:string_edge_vertices}. Each vertex of $\Gamma'$ represents a convergent sequence of domains with tame ends. The maps may not have a convergent subsequence, due to the bubbling phenomenon. Cut along concentric necks around points where the derivative is blowing up in order to ensure convergence. The effect of this bubbling argument is to change $\Gamma'$ by adding on finitely many trees to each vertex; each added tree corresponds to a ``bubble tree'' arising from an unbounded derivative. 

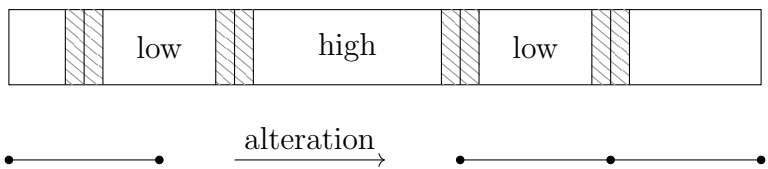
\begin{figure}[H]
  \centering
  \begin{tikzpicture}
    \draw (0,0) rectangle (10,1);
    \foreach \x in {1,3,6,8} {
      \fill[pattern={Lines[angle=-40]},pattern color=black!40!white] (\x-0.25,0) rectangle +(.5,1);
      \draw (\x,0)--+(0,1) (\x-0.25,0)--+(0,1) (\x+0.25,0)--+(0,1);
    }
    \path (0,0.5)--+(2,0)node{low}--+(4.5,0)node{high}--+(7,0)node{low};

    \begin{scope}[shift={(0,-1)}]
      \draw[every node/.style={draw,circle,inner sep=1pt,fill}] (0,0)node{}--(2,0)node{};
    \end{scope}
    \draw[->] (3,-1)--node[above]{alteration}+(2,0);
    \begin{scope}[shift={(6,-1)}]
      \draw[every node/.style={draw,circle,inner sep=1pt,fill}] (0,0)node{}--+(2,0)node{}--+(4,0)node{};
    \end{scope}
  \end{tikzpicture}
  \caption{An edge in $\Gamma_\mathrm{dom}$ becomes a sequence of edges and vertices. The new edges are in bijection with the low energy regions, and the added vertices are in bijection with the high energy regions. A similar alteration process is performed at the exterior edges of $\Gamma_{\mathrm{dom}}$; i.e., we consider a neighborhood of each marked point as a half infinite end, which we then decompose into alternating low and high energy regions.}
  \label{fig:string_edge_vertices}
\end{figure}

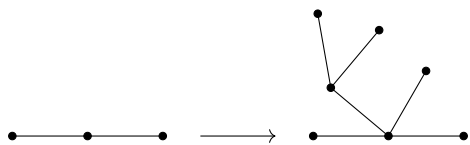
\begin{figure}[H]
  \centering
  \begin{tikzpicture}
    \draw[every node/.style={draw,circle,inner sep=1pt,fill}] (0,0)node{}--(1,0)node{}--(2,0)node{};
    \draw[->] (2.5,0)--(3.5,0);
    \begin{scope}[shift={(4,0)}]
      \draw[every node/.style={draw,circle,inner sep=1pt,fill}] (0,0)node{}--(1,0)node{}--(2,0)node{} (1,0)--+(60:1)node{} (1,0)--+(140:1)coordinate(X)node{} (X)--+(50:1)node{} (X)--+(100:1)node{};
    \end{scope}
  \end{tikzpicture}
  \caption{We alter $\Gamma'$ by adding finitely many trees to each vertex. The added trees correspond to bubble trees.}
  \label{fig:bubble_vertices}
\end{figure}
Denote by $\Gamma$ the graph formed from $\Gamma'$ by incorporating the bubble trees. Then $\Gamma$ can be labelled so as to become a generalized holomorphic curve (Definition \ref{defn:generalized_curve}); simply recall that vertices and edges of $\Gamma$ correspond to regions on the limiting domain $\Sigma_{n}$, and the limiting behaviour of the restriction of $u_{n}$ to each piece determines the label. It is tautological to check that $u_{n}$ indeed converges to the generalized holomorphic curve described by $\Gamma$, according to Definition \ref{defn:convergence_defn}. The verification of the various compatibility is a routine affair, and is left to the reader. This completes the proof of Theorem \ref{theorem:main_theorem}, modulo the technical details needed to prove the results in \S\ref{sec:low-energy-convergence}, which we treat in the next section.

\section{Low energy regions}
\label{sec:low-energy-regions}
In this section we analyze \emph{low energy} strips (as defined in \S\ref{sec:low-energy-de}). The analysis is different in each case, i.e., cylindrical necks in \S\ref{sec:low-energy-cylinders}, strip necks with \emph{adiabatic} boundary conditions in \S\ref{sec:low-energy-strips}, and the other two boundary conditions in \S\ref{sec:other_low_energy_strips}. The case of adiabatic boundary conditions is the most analytically involved, and requires the elliptic estimates from \S\ref{sec:exponential-estimates}. 

\subsection{Low energy cylinders}
\label{sec:low-energy-cylinders}
In this section we prove that long cylinders with low energy converge uniformly to points. This is often used in arguments proving that ``bubbles connect.'' Throughout $K$ denotes a compact subset of the target $(W,\omega)$.

\begin{prop}\label{prop:low-e-cyl}
  There exists a constant $\hbar=\hbar(K,\omega,J)>0$ with the following property: if $J_{n}\to J$ is a $C^{\infty}$ convergent family of almost complex structures, $J$ is tame, and $u_{n}:[-R_{n},R_{n}]\times \R/\Z\to K$ is a sequence of $J_{n}$-holomorphic cylinders with bounded energy and
  \begin{equation}\label{eq:end-plate-assumption}
    \lim_{n\to \infty}\abs{\d u_{n}(s-R_{n},t)}+\abs{\d u_{n}(R_{n}-s,t)}=0\text{ uniformly for $(s,t)$ in compact sets},
  \end{equation}
  then:
  \begin{equation}\label{eq:energy-assumption}
    \liminf_{n\to\infty}E(u_{n})>0\implies \liminf_{n\to\infty}\max_{s,t}\abs{\d u_{n}(s,t)}>0\implies\liminf_{n\to\infty}E(u_{n})>\hbar.
  \end{equation}
  On the other hand:
  \begin{equation}\label{eq:uniformly-to-pt}
    \lim_{n\to\infty}\max_{s,t}\abs{\d u_{n}(s,t)}=0\implies\lim_{n\to\infty}\text{diam}(u_{n}([-R_{n},R_{n}]\times \R/\Z))=0.
  \end{equation}
\end{prop}
\begin{proof}
  Our argument is similar to the ones \cite[Chapter 4]{mcduffsalamon}. The hypothesis \eqref{eq:end-plate-assumption} says that $\d u_{n}(z_{n})$ converges to zero provided $z_{n}$ remains a finite distance from the boundary circles $\set{\pm R_{n}}\times \R/\Z$. 
  
  Start with the second implication of \eqref{eq:energy-assumption}. Then there exists a sequence $(s_{n},t_{n})$ so that $\lim_{n\to\infty}\abs{s_{n}\pm R_{n}}=\infty$ and $\liminf_{n\to\infty}\abs{\d u_{n}(s_{n},t_{n})}>0$. We claim that the maps $(s,t)\mapsto u_{n}(s_{n}+s,t)$ have energy bounded below by the energy of a sphere bubble (e.g.,\ a non-constant holomorphic sphere). Indeed, if $(s,t)\mapsto u_{n}(s_{n}+s,t)$ has derivative bounded on compact sets, then $u_{n}$ converges in $C^{\infty}_{\mathrm{loc}}$ to a holomorphic cylinder, which has removable singularities and so extends to a holomorphic sphere. This limit is necessarily non-constant since $\abs{\d u_{n}(s_{n},t_{n})}$ is bounded below. On the other hand, if the derivative is unbounded, then the bubbling arguments in the proof Lemma \ref{lemma:basic_bubble} constructs a non-constant holomorphic plane, which again extends to a $J$-holomorphic sphere. In either case, we have proved the desired energy bound.

  Consider the first implication of \eqref{eq:energy-assumption}. We argue the contrapositive. Then $u_{n}$ has a subsequence so that $\max_{s,t}\abs{\d u_{n}(s,t)}$ converges to zero. It suffices to show that the energy of $u_{n}$ also converges to zero, after potentially passing to another subsequence. In the course of our argument we will also prove \eqref{eq:uniformly-to-pt}.

  The Arzel\`a-Ascoli theorem produces a subsequence for which the central loop $u_{n}(0,t)$ converges uniformly to a point $p$. Then there exists $\epsilon>0$ and a primitive $\lambda$ of $\omega$ (i.e., $\omega=\d\lambda$) defined on the closed $\epsilon$ ball $\cl{B}(p,\epsilon)$ around $p$. We claim that eventually the entire image of $u_{n}$ lies inside of $B(p,\epsilon)$. The idea depends on a exponential estimate (see \cite[Chapter 4]{mcduffsalamon} for similar results). Here is the argument: by contradiction, for $n$ sufficiently large, we can find radii $r_{n}\le R_{n}$ so that
  \begin{equation*}
    u_{n}([-r_{n},r_{n}]\times \R/\Z)\subset \cl{B}(p,\epsilon)\text{ but }u_{n}([-r_{n},r_{n}]\times \R/\Z)\not\subset B(p,\epsilon).
  \end{equation*}
  Clearly, as the derivative of $u_{n}$ converges to zero uniformly and $u_{n}(0,t)$ converges uniformly to $p$, these radii $r_{n}$ must be diverging to infinity. Introduce the notation $\gamma_{n,r}(t):=u_{n}(r,t)$. 

  Consider the quantity $E_{n}(r)$, defined for $r\le r_{n}$: 
  \begin{equation*}
    E_{n}(r)=\int_{-r}^{r}\int_{0}^{1}u_{n}^{*}\omega=\int_{0}^{1}\gamma_{n,r}^{*}\lambda-\int_{0}^{1}\gamma_{n,-r}^{*}\lambda.
  \end{equation*}
  The idea is to show $E_{n}(r)$ satisfies a certain first-order differential inequality. We compute:
  \begin{equation*}
    E_{n}'(r)=\int_{0}^{1}\abs{\bd_{s}u_{n}(r,t)}_{\omega,n}^{2}\,\d t+\int_{0}^{1}\abs{\bd_{s}u_{n}(-r,t)}_{\omega,n}^{2}\,\d t,
  \end{equation*}
  where $\abs{v}_{\omega,n}^{2}=\omega(v,J_{n}v)$ (we do not assume that $\omega(v,J_{n}w)$ is symmetric). The similarity of $E_{n}(r)$ and $E_{n}'(r)$ allows us to set up a differential inequality. We require one trick to get a good estimate; observe that
  \begin{equation*}
    \int_{0}^{1}\gamma_{n,r}^{*}\lambda=\int_{0}^{1}\gamma_{n,r}^{*}(\lambda+\d f).
  \end{equation*}
  By picking $f$ judiciously, we can assume that $\lambda+\d f$ vanishes at $\gamma_{n,r}(0)$. Then we can conclude that, for any metric $g$ on the target $W$,
  \begin{equation*}
    \begin{aligned}
      \int_{0}^{1}\gamma_{n,r}^{*}\lambda&=\int_{0}^{1}\gamma_{n,r}^{*}(\lambda+\d f)\le C\int_{0}^{1}\abs{\d u_{n}(r,t)}_{g}\mathrm{dist}_{g}(u_{n}(r,t),u_{n}(r,0))\d t\\ &\le C\left[\int_{0}^{1}\abs{\d u_{n}(r,t)}_{g}\d t\right]^{2}\le C\int_{0}^{1}\abs{\d u_{n}(r,t)}_{g}^{2}\d t\le C\int_{0}^{1}\abs{\bd_{s}u_{n}(r,t)}^{2},
    \end{aligned}
  \end{equation*}
  using the definition of distance in the second inequality, the Cauchy-Schwarz inequality in the third, and the holomorphic curve equation in the third. Here $C$ is a constant corresponding to the $C^{1}$ size of $\lambda+\d f$ (we may need to enlarge $C$ in the final inequality). It is straightforward to construct $f$ so that the $C^{1}$ size of $\lambda+\d f$ is bounded by three times the $C^{1}$ size of $\lambda$, and hence $C$ can be taken to be independent of the precise choice of~$f$.

  Using this estimate, and the fact that $\abs{v}_{g}$ is uniformly commensurate with $\abs{v}_{\omega,n}$ (as $n\to\infty$), conclude that:
  \begin{equation*}
    E_{n}(r)\le CE_{n}'(r)\implies E_{n}'(r)-\delta E_{n}(r)\ge 0\implies e^{\delta (r_{n}-r)}E_{n}(r)\le E_{n}(r_{n}),
  \end{equation*}
  where $\delta=1/C$. 
  
  Consider the decomposition of $[-r_{n},r_{n}]\times \R/\Z=\Sigma_{1,n}\cup \Sigma_{2,n}$ where:
  \begin{equation*}
    \Sigma_{1,n}=([-r_{n},-r_{n}+1]\cup [r_{n}-1,r_{n}])\times \R/\Z\hspace{.3cm}\text{and}\hspace{.3cm}\Sigma_{2,n}=[-r_{n}+1,r_{n}-1]\times \R/\Z.
  \end{equation*}
  The mean-value property can be applied for $(s,t)\in \Sigma_{2,n}$ to conclude that the derivative satisfies an estimate of the form:
  \begin{equation*}
    \abs{\d u_{n}(s,t)}\le C'e^{\delta(r_{n}-s)}E_{n}(r_{n}).
  \end{equation*}
  In particular, using the exponential estimate, conclude the diameter of $u_{n}(\Sigma_{2,n})$ is bounded by a fixed constant, depending on $C'$ and $\delta$, times $E_{n}(r_{n})$. Exactness implies $E_{n}(r_{n})$ tends to zero since $\abs{\d u_{n}}\to 0$ on the boundary circles. On the other hand, since the two components of $\Sigma_{1,n}$ have bounded diameter, and $\abs{\d u_{n}}\to 0$, the diameter of $u_{n}(\Sigma_{1,n})$ also tends to zero. As a consequence, the diameter of $u_{n}([-r_{n},r_{n}]\times \R/\Z)$ converges to zero, contradicting $u_{n}(0,0)\to p$ and $u_{n}([-r_{n},r_{n}]\times \R/\Z)\cap \bd B(p,\epsilon)\ne \emptyset$.

  Thus the assumption that the curve eventually leaves $B(p,\epsilon)$ leads to a contradiction, and hence the entire image of $u_{n}$ must eventually lie in $B(p,\epsilon)$. Noting that $\epsilon>0$ could be taken arbitrarily small, we have completed the proof of \eqref{eq:uniformly-to-pt}.\footnote{The reader may complain that we have only shown that the diameter tends to zero along a subsequence. However, this issue is easily remedied by first taking a subsequence which realizes the limit supremum of the diameter.} 

  It remains only to prove that the energy of $u_{n}$ converges uniformly to zero. The same argument above yields:
  \begin{equation*}
    E(u_{n})\le C\int_{0}^{1}\abs{\d u_{n}(R_{n},t)}^{2}+\abs{\d u_{n}(-R_{n},t)}^{2}\d t.
  \end{equation*}
  By assumption, the $C^{1}$ size of $u_{n}$ converges to zero at both endpoints. Thus the right hand side in the above inequality converges to zero, and so the energy also converges to zero, as desired. This completes the proof of the contrapositive of the first implication in \eqref{eq:energy-assumption}, and the proof of the proposition.
\end{proof}

\subsection{Low energy strips}
\label{sec:low-energy-strips}
In this section we begin with an analogue of Proposition \ref{prop:low-e-cyl} for the case when $u_{n}$ is defined on a strip. In the case of adiabatic boundary conditions, the hypothesis that $\lim_{n\to\infty}\mathrm{max}\abs{\d u_{n}(s,t)}=0$ \emph{does not} imply that $u_{n}$ converges to a point. As explained in \S\ref{sec:digression-flow-line}, $u_{n}$ will converge to a \emph{broken flow line}. The main goal of this section is to prove this convergence result. The key input will be the exponential estimates from \S\ref{sec:exponential-estimates}. The other boundary conditions from \S\ref{sec:boundary_conditions} are studied these cases in \S\ref{sec:other_low_energy_strips}.


\begin{prop}\label{prop:low-e-strips}
  Let $J$ be $\omega$-tame, $L$ be $J$-totally real, and $K\subset W$ a compact set. Then there exists a constant $\hbar=\hbar(K,\omega,J,L)>0$ with the following property. If $J_{n}\to J$ and $u_{n}:[-R_{n},R_{n}]\times [0,1]\to K$ is a sequence of $J_{n}$-holomorphic strips with bounded energy satisfying the adiabatic boundary conditions:
  \begin{equation}\label{eq:boundary-conditions-pr}
    \begin{aligned}
      u_{n}(s,0)\in L_{\epsilon_{n}\mathfrak{a}_{n}}\text{ and }u_{n}(s,1)\in L_{\epsilon_{n}\mathfrak{b}_{n}}\\
    \end{aligned}
  \end{equation}
  where $\mathfrak{b}_{n}-\mathfrak{a}_{n}=\d f_{n}$, $\mathfrak{a}_{n}$ is closed, $\epsilon_{n}\mathfrak{a}_{n}\to 0$ and $\epsilon_{n}\mathfrak{b}_{n}\to 0$, and 
  \begin{equation}\label{eq:end-plate-assumption-2}
    \lim_{n\to \infty}\abs{\d u_{n}(s-R_{n},t)}+\abs{\d u_{n}(R_{n}-s,-t)}=0\text{ in $C^{0}_{\mathrm{loc}}([0,\infty)\times [0,1])$},\footnote{Even though each $u_n$ is not defined on the whole half infinite strip, any compact set will be contained in the domain of $u_{n}$ for $n$ large enough.}
  \end{equation}
  then the conclusion is
  \begin{equation}\label{eq:energy-assumption-2}
    \liminf_{n\to\infty}E(u_{n})>0\implies \liminf_{n\to\infty}\max_{s,t}\abs{\d u_{n}(s,t)}>0\implies\liminf_{n\to\infty}E(u_{n})>\hbar.
  \end{equation}
\end{prop}

\begin{proof}
  The proof of the second implication in \eqref{eq:energy-assumption-2} is the same as the proof of Proposition \ref{prop:low-e-cyl}.
  
  The first implication requires slightly different techniques from the ones used in Proposition \ref{prop:low-e-cyl}. As in Proposition \ref{prop:low-e-cyl}, argue the contrapositive: assuming the derivative goes to zero uniformly, show the energy must also go to zero. Then, for $n$ large enough $u_{n}$ maps $[-R_{n},R_{n}]\times [0,1]$ into the tubular neighbourhood $N\subset T^{*}L$.
  
  Define a sequence of primitives of $\omega$ on $T^{*}L$ by the formula
  \begin{equation}\label{eq:my-primitive}
    \lambda_{n}=-\lambda_{\mathrm{can}}+\epsilon_{n}\pr^{*}\mathfrak{a}_{n}.
  \end{equation}
  This is a primitive of $\omega=-\d\lambda_{\mathrm{can}}$ since $\mathfrak{a}_{n}$ is closed. Moreover, it is easy to see that:
  \begin{equation*}
    L_{\epsilon_{n}\mathfrak{a}_{n}}^{*}\lambda_{n}=0\text{ and }L_{\epsilon_{n}\mathfrak{b}_{n}}^{*}\lambda_{n}=-\epsilon_{n}\d f_{n}.
  \end{equation*}
  The energy of a holomorphic strip satisfying the boundary conditions \eqref{eq:boundary-conditions-pr} which remains entirely in the tubular neighbourhood $N$ can be computed by Stokes' theorem, see Figure \ref{fig:stokes}.

  \begin{figure}[H]
    \centering
    \begin{tikzpicture}
      \draw (0,0) rectangle (7,2);
      \node at (-0.2,1) [left] {$\displaystyle\int \lambda_{n}$};
      \node at (7.2,1) [right] {$\displaystyle\int \lambda_{n}$};
      \node at (3.5,2.2) [above] {$-\displaystyle\epsilon_{n}\int \d f_{n}$};
      \node at (3.5,-0.2) [below] {$0$};
      \begin{scope}[every path/.style={decoration={markings,mark=between positions 6mm and 1 step 6mm with {\arrow{>[scale=1.4]};}},decorate}]
        \draw (0,0)--(7,0);
        \draw (7,0)--(7,2);
        \draw (7,2)--(0,2);
        \draw (0,2)--(0,0);
      \end{scope}
    \end{tikzpicture}
    \caption{Computing the energy of a holomorphic strip with $\epsilon_n\mathfrak{a}_n$ and $\epsilon_{n}\mathfrak{b}_{n}$ boundary conditions.}
    \label{fig:stokes}
  \end{figure}
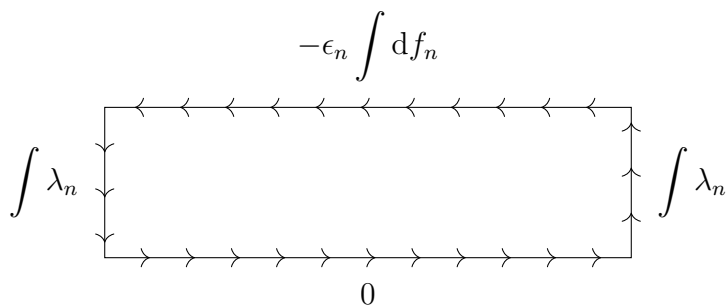
  
  By the assumption on $\d u_{n}$ in the statement of the lemma, we conclude that the two vertical integrals converge to zero. The integral along the bottom edge is obviously $0$, while the integral along the top edge is bounded by $\epsilon_{n}(\mathrm{max}(f_{n})-\mathrm{min}(f_{n}))$. The difference between the maximum and the minimum of a function is bounded by its derivative, and since $\epsilon_{n}(\mathfrak{b}_{n}-\mathfrak{a}_{n})=\epsilon_{n}\d f_{n}$ converges to $0$, the integral along the top edge is $o(1)$, and hence also converges to zero. Thus the energy of $u_{n}$ converges to zero, completing the proof of the contrapositive. This completes the proof of the lemma.
\end{proof}

The goal now is to show that a low energy strip with adiabatic boundary conditions converges to a broken flow line. The argument is divided into two parts, depending on whether $\ell_{n}=\epsilon_{n}r_n$ converges to a finite number or not. This limit is the length of the flow line. In the case when $\epsilon_{n}r_{n}$ converges to a finite number, no exactness assumptions on $\mathfrak{b}_{n}-\mathfrak{a}_{n}$ are required.

\begin{lemma}\label{lemma:finite-number-case}
  Let $J_{n}\to J$ be a convergent sequence of complex structures for which $J_{n}|_{L}=J|_{L}$, and $J|_{L}$ is $\omega$-compatible.
  
  Let $u_{n}:[-r_{n},r_{n}]\times [0,1]\to K$ be a sequence of $J_{n}$-holomorphic strips with
  \begin{equation*}
    \lim_{n\to\infty}\max_{s,t}\abs{\d u_{n}(s,t)}=0.
  \end{equation*}
  Suppose that $u_{n}(s,0)\in L_{\epsilon_{n}\mathfrak{a}_{n}}$ and $u_n(s,1)\in L_{\epsilon_{n}\mathfrak{b}_{n}}$. Assume that $\mathfrak{a}_{n},\mathfrak{b}_{n}$ converge (we do not assume any exactness), $\epsilon_{n}\to 0$, and $\ell_{n}=\epsilon_{n}r_{n}$ converges to a \emph{finite} limit $\ell_{\infty}$ (which could be zero).

  Then, after passing to a subsequence, the rescaled maps $v_{n}:[-\ell_{n},\ell_{n}]\times [0,\epsilon_{n}]\to K$ defined by:
  \begin{equation*}
    v_{n}(s,t)=u_{n}(\epsilon_{n}^{-1}s,\epsilon_{n}^{-1}t)    
  \end{equation*}
  converge uniformly to a flow line $v_{\infty}:[-\ell_{\infty},\ell_{\infty}]\to L$ for the vector field $g$-dual to the one-form $\mathfrak{c}_{\infty}=\lim_{n\to\infty}(\mathfrak{b}_{n}-\mathfrak{a}_{n})$. The convergence is given by:
  \begin{equation*}
    \lim_{n\to\infty}\sup_{s,t\in [-\ell_{n},\ell_{n}]\times [0,\epsilon_{n}]}\mathrm{dist}(v_{n}(s,t),v_{\infty}(s))=0,
  \end{equation*}
  where we use the fact that the domain of $v_{\infty}$ can be extended to all of $\R$ (it makes sense to evaluate $v_{\infty}$ at $s\not\in [-\ell_{\infty},\ell_{\infty}]$ in the above convergence statement).
\end{lemma}
\begin{proof}
  The idea to use the exponential decay estimates from Lemma \ref{lemma:c1estimates} to control the derivatives of the rescaled map $v_n$, and apply the Arzel\`a-Ascoli theorem. Analysis of the holomorphic curve equation for $u_{n}$ will imply limit of $v_n$ satisfies the flow line equation for the vector field dual to  $\mathfrak{c}_{\infty}$. 

  The assumptions on $u_{n},J_{n},\mathfrak{a}_{n},\mathfrak{b}_{n}$ enable us to apply Lemmas \ref{lemma:statementA} and \ref{lemma:c1estimates}. Decomposing $u_{n}$ into the $Q_{n}$, $P_{n}$ component functions we conclude equation \eqref{eq:daren3} from the proof of Lemma \ref{lemma:statementA}, reprinted here:
  \begin{equation*}\tag{\ref{eq:daren3}}
    \begin{aligned}
      g_*\qs &= \nabla_t \tilde{P}_n + \epsilon_{n}(\nabla \mathfrak{a}_{n}\circ Q_{n})\cdot \bd_{t}Q_{n}+\epsilon_{n}\mathfrak{c}_n\circ Q_n \\&\hspace{2cm}+ t\epsilon_{n}(\nabla\mathfrak{c}_n\circ Q_{n})\cdot\qt +A_{n}(u_{n})\cdot P_{n}\cdot \bd_{t}u_{n}\\
      g_*\qt &= -\nabla_s \tilde{P}_n - \epsilon_{n}(\nabla \mathfrak{a}_{n}\circ Q_{n})\cdot \qs\\&\hspace{2cm}- t\epsilon_{n}(\nabla\mathfrak{c}_n\circ Q_{n})\cdot\qs + B_{n}(u_{n})\cdot P_{n}\cdot \bd_{t}u_{n}.
    \end{aligned}
  \end{equation*}
  Recall that $\mathfrak{c}_{n}:=\mathfrak{b}_{n}-\mathfrak{a}_{n}$, and $\tilde{P}_{n}:=P_{n}-\epsilon_{n}t\,\mathfrak{c}_{n}\circ Q_{n}-\epsilon_{n}\mathfrak{a}_{n}\circ Q_{n}$,  and $A_{n}(u_n),B_{n}(u_n)$ are bounded error terms related to the difference between $J_{n}$ and $J$.
  
  Lemma \ref{lemma:c1estimates} and the fact that $\abs{\d u_{n}}\to 0$ imply that for $s\in [-r_{n}+1,r_{n}-1]$,
  \begin{equation}\label{eq:my-estimates}
    \begin{aligned}
      \abs{P_{n}-\epsilon_{n}t\mathfrak{c}_{n}\circ Q_{n}-\epsilon_{n}\mathfrak{a}_{n}\circ Q_{n}}&\le \kappa_{n}(e^{-\delta(r_{n}+s)}+e^{-\delta(r_{n}-s)}+\epsilon_{n})\\
      \abs{g_{*}\bd_{s}Q_{n}-\epsilon_{n}\mathfrak{c}_{n}\circ Q_{n}}&\le \kappa_{n}(e^{-\delta(r_{n}+s)}+e^{-\delta(r_{n}-s)}+\epsilon_{n})\\
      \abs{\bd_{t}Q_{n}}&\le \kappa_{n}(e^{-\delta(r_{n}+s)}+e^{-\delta(r_{n}-s)}+\epsilon_{n}),
    \end{aligned}
  \end{equation}
  for some sequence $\kappa_{n}\to 0$.
  
  Let $v_{n}(s,t)=u_{n}(\epsilon_{n}^{-1}s,\epsilon_{n}^{-1}t)$. Then:
  \begin{equation*}
    q_{n}(s,t)=Q_{n}(\epsilon_{n}^{-1}s,\epsilon_{n}^{-1}t)\hspace{1cm} p_{n}(s,t)=P_{n}(\epsilon_{n}^{-1}s,\epsilon_{n}^{-1}t)
  \end{equation*}
  are the horizontal and vertical components of $v_{n}$. As a consequence of the first line of \eqref{eq:my-estimates}, $\abs{p_{n}(s,t)}\to 0$, and hence:
  \begin{equation*}
    \lim_{n\to\infty}\sup_{s,t\in [-\ell_{n},\ell_{n}]\times [0,\epsilon_{n}]}\mathrm{dist}(q_{n}(s,t),v_{n}(s,t))=0.
  \end{equation*}
  Thus it suffices to prove that $q_{n}(s,t)$ converges uniformly to a flow line dual to $\mathfrak{c}_{\infty}$.

  Substituting $\d q_{n}(s,t)=\epsilon_{n}^{-1}\d Q_{n}(\epsilon_{n}^{-1}s,\epsilon_{n}^{-1}t)$ and $r_{n}=\ell_{n}/\epsilon_{n}$, conclude
  \begin{equation*}
    \abs{\bd_{s}q_{n}-g_{*}^{-1}\mathfrak{c}_{n}\circ q_{n}}+\abs{\bd_{t}q_{n}}\le 2\kappa_{n}\frac{(e^{-\delta/\epsilon_{n}(\ell_{n}+s)}+e^{-\delta/\epsilon_{n}(\ell_{n}-s)}+\epsilon_{n})}{\epsilon_{n}}.
  \end{equation*}
  In particular, using the exponential decay, conclude that for $s\in [-\ell_{n}+\rho_{n},\ell_{n}-\rho_{n}]$ where $\rho_{n}=(\epsilon_{n}/\delta)\log(1/\epsilon_{n})$ the estimate:
  \begin{equation}\label{eq:gradient-6-kappa}
    \abs{\bd_{s}q_{n}-g_{*}^{-1}\mathfrak{c}_{n}\circ q_{n}}+\abs{\bd_{t}q_{n}}\le 6\kappa_{n}.
  \end{equation}
  Note that $\rho_{n}\to 0$.

  Arzel\`a-Ascoli produces a subsequence so that $q_{n}(-\ell_{n}+\rho_{n},t)$ converges to a point $x_{-}$ uniformly in $t$. Let $v_{\infty}:[-\ell_{\infty},\ell_{\infty}]\to L$ denote the flow line for the vector field $g_{*}^{-1}\mathfrak{c}_{\infty}$ satisfying $v_{\infty}(-\ell_{\infty})=x_{-}$. Our goal is to prove that
  \begin{equation*}
    \lim_{n\to\infty}\sup_{s,t\in [-\ell_{n},\ell_{n}]\times [0,\epsilon_{n}]}\mathrm{dist}(q_{n}(s,t),v_{\infty}(s))=0,
  \end{equation*}
  This will complete the proof of the lemma.

  Let $\varphi_{n}:[-\ell_{n}+\rho_{n},\ell_{n}-\rho_{n}]\to L$ be the flow line for $g_{*}^{-1}\mathfrak{c}_{n}$ starting at $q_{n}(-\ell_{n}+\rho_{n},0)$. Standard estimates from the theory of ODEs and \eqref{eq:gradient-6-kappa} imply that\footnote{Morally, both the initial conditions and the differential equations differ by some error converging to zero uniformly, and the domains are bounded, so the solutions remain close. More explicitly, the difference $\Delta_{n,t}(s)$ between $q_{n}(\ell_{n}-\rho_{n}+s,t)$ and $\varphi_{n}(\ell_{n}-\rho_{n}+s)$ satisfies a differential inequality of the form $\Delta_{n,t}'-c_{1}\Delta_{n,t}\le c_{2}\kappa_{n}$ where $c_{1}$ depends on $\mathfrak{c}_{\infty}$ and $c_{2}\approx 6$. There is some metric distortion which happens when one works in coordinates. This differential inequality can be integrated to obtain $\Delta_{n,t}(s)\le e^{cs}(\kappa_{n}s+\Delta_{n,t}(0))$. One observes that $\Delta_{n,t}(0)$ can also be bounded in terms of $\kappa_{n}$ by integrating $\bd_{t}q_{n}$. The details are left to the reader.}
  \begin{equation}\label{eq:gronwall-estimate}
    \max_{s,t\in [-\ell_{n}+\rho_{n},\ell_{n}-\rho_{n}]\times [0,\epsilon_{n}]}\mathrm{dist}(q_{n}(s,t),\varphi_{n}(s))\le C\kappa_{n},
  \end{equation}
  where $C$ depends on $\ell_{\infty}$ and $\mathfrak{c}_{\infty}$.

  Since the starting point of $\varphi_{n}(s)$ converges to the starting point of $v_{\infty}$, and the domain of $\varphi_{n}$, namely $[-\ell_{n}+\rho_{n},\ell_{n}-\rho_{n}]$, converges to the domain of $v_{\infty}$, namely $[-\ell_{\infty},\ell_{\infty}]$, \eqref{eq:gronwall-estimate} yields:
  \begin{equation*}
    \lim_{n\to \infty}\max_{s,t\in [-\ell_{n}+\rho_{n},\ell_{n}-\rho_{n}]\times [0,\epsilon_{n}]}\mathrm{dist}(q_{n}(s,t),v_{\infty}(s))=0.
  \end{equation*}
  It remains only to analyze the behavior near the ends, i.e., the region $\left[-\ell_n,-\ell_n+\rho_{n}\right]$ and $\left[\ell_n-\rho_{n},\ell_n\right].$ It is sufficient to prove that:
  \begin{equation*}
    \lim_{n\to\infty}\mathrm{diam}(q_{n}([-\ell_{n},-\ell_{n}+\rho_{n}]\times [0,\epsilon_{n}]))+\mathrm{diam}(q_{n}([\ell_{n}-\rho_{n},\ell_{n}]\times [0,\epsilon_{n}]))=0.
  \end{equation*}
  For this, rescale back to the original map. Consider one of the two regions, as the situation is exactly the same in the other one. The region $[-\ell_{n},-\ell_{n}+\rho_{n}]\times [0,\epsilon_{n}]$ is expanded to the rectangle: $$\Omega_{n}=[-r_{n},-r_{n}+\delta^{-1}\log(1/\epsilon_{n})]\times [0,1].$$ It follows from \eqref{eq:my-estimates} that:
  \begin{equation*}
    \abs{\d Q_{n}(s,t)}\le \kappa_{n}(e^{-\delta(r_{n}+s)}+e^{-\delta(r_{n}-s)}+\epsilon_{n})+C\epsilon_{n},
  \end{equation*}
  where $C$ is a constant depending on $\mathfrak{c}_{\infty}$. To estimate the diameter it suffices to bound the integral $\abs{\d Q_{n}}$ over horizontal and vertical lines in $\Omega_{n}$. It is easy to see that:
  \begin{enumerate}
  \item integrals over vertical lines are bounded by $\kappa_{n}(2+\epsilon_{n})+C\epsilon_{n}$,
  \item integrals over horizontal lines are bounded by
    \begin{equation*}
      \int_{-r_{n}}^{-r_{n}+\delta^{-1}\log(1/\epsilon_{n})}\abs{\d Q_{n}(s,t)}\d s\le \kappa_{n}(\frac{2}{\delta})+(\kappa_{n}+C)\epsilon_{n}\delta^{-1}\log(\frac{1}{\epsilon_{n}}).
    \end{equation*}
    One key is that the integral of the exponential function $e^{-\delta(r_{n}+s)}$ over $[-r_{n},\infty)$ is bounded by $1/\delta$.
  \end{enumerate}
  Since $\epsilon_{n}\log(1/\epsilon_{n})\to 0$ and $\kappa_{n}\to 0$, we conclude that these integrals tend to zero, and hence the diameter also tends to zero. This completes the proof.
\end{proof}

\subsubsection{Convergence to Morse flow lines}
\label{sec:proof_main_morse}
We prove Theorem \ref{theorem:main_morse}. Recall that $u_{n}$ has low energy and satisfies the adiabatic boundary conditions $u_{n}(s,0)\in L_{\epsilon_{n}\mathfrak{a}_{n}}$ and $u_{n}(s,1)\in L_{\epsilon_{n}\mathfrak{b}_{n}}$ where $\epsilon_{n}\to 0$ and $\mathfrak{b}_{n}-\mathfrak{a}_{n}=\d f_{n}$ converges to $\d f_{\infty}$ for $f_{\infty}$ \emph{Morse}.

In Theorem \ref{theorem:main_morse}, the domain of $u_{n}$ was presented as $[-R_{n}-\rho_{n},R_{n}+\rho_{n}]\times [0,1]$, and the ends of length $\rho_{n}$ are very tame, i.e., converge uniformly to points $x_{\pm}$. The necks which we cut along are supposed to lie in $[-R_{n},R_{n}]$. To avoid too much notation, we ignore the ends and simply consider the domain of $u_{n}$ as $[-r_{n},r_{n}]\times [0,1]$, leaving the details to the reader.


\begin{proof}[Theorem \ref{theorem:main_morse}]  
  The idea is to repeatedly apply Lemma \ref{lemma:finite-number-case}, iteratively constructing the limiting flow line. Let $[-r_{n},r_{n}]\times [0,1]$ be the domain of $u_{n}$. Suppose that $\ell_{n}:=\epsilon_{n}r_{n}$ converges to $+\infty$. The rescaled map $v_{n}(s,t)=u_{n}(\epsilon_{n}^{-1}s,\epsilon_{n}^{-1}t)$ is defined on $[-\ell_{n},\ell_{n}]\times[0,\epsilon_{n}]$.
  
  The first task is to construct the ``end'' strips $\Sigma_{n}^{\pm}$. Fix $\rho>0$ and consider the translated restriction:
  \begin{equation*}
    (s,t)\in [0,\rho]\times [0,\epsilon_{n}]\mapsto w_{n}^{-}(s,t)=v_{n}(-\ell_{n}+s,t).
  \end{equation*}
  Lemma \ref{lemma:finite-number-case} produces a subsequence so that $w_{n}^{-}(s,t)$ converges uniformly to a flow line $w_{\infty}(s)\in L$ for $g_{*}^{-1}\d f_{\infty}$ defined on $[0,\rho]$. In particular, $w_{n}(0,t)=v_{n}(-\ell_{n},t)$ converges uniformly to some point $x_{-}$. 

  Let $\gamma_{-}$ be the flow line starting at $x_{-}$ and ending at some critical point $y_{1}$. Uniqueness of flow lines starting at $x_{-}$ implies that $w_{n}$ converges uniformly to $\gamma_{-}|_{[0,\rho]}$.

  Pick a sequence $\rho_{k}\to \infty$, and repeatedly apply the preceding argument, obtaining a diagonal subsequence. As in Lemma \ref{lemma:little-trick}, for each $N \in \mathbb{N}$, let $k(N)$ be the maximal $k$ so that $\rho_{k}<\ell_{N}$ and 
  \begin{equation*}
    s\in [0,\rho_{k}]\text{ and }n\ge N\implies \mathrm{dist}(w_{n}(s,t),\gamma_{-}(s))<\frac{1}{k}.
  \end{equation*}
  For $k$ fixed, the above implication will eventually hold as $N\to\infty$, by the construction of the diagonal subsequence. This implies that $k(N)\to\infty$ as $N\to\infty$. Let $r_{n}^{-}=\rho_{k(n)}$, and define $\Sigma_{n}^{-}=[-\epsilon_{n}^{-1}\ell_{n},-\epsilon_{n}^{-1}(\ell_{n}-r_{n}^{-})]\times [0,1]$.

  Then $\Sigma_{n}^{-}$ will satisfy the first part of (iii). A similar argument constructs $\Sigma_{n}^{+}$. Denote $\epsilon_{n}^{-1}I_{n}^{\pm}\times [0,1]=\Sigma_{n}^{\pm}$.
  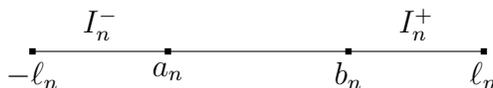
\begin{figure}[H]
    \centering
    \begin{tikzpicture}
      \path[every node/.style={draw,inner sep=1pt,fill}] (0,0)node{}--node(A)[pos=0.3]{}node[pos=0.7](B){}(6,0)node{};
      \node at (A) [below] {$a_{n}$};
      \node at (B) [below] {$b_{n}$};
      \draw (0,0)node[below]{$-\ell_{n}$}--node[pos=0.15,above]{$I_{n}^{-}$}node[pos=0.85,above]{$I_{n}^{+}$}(6,0)node[below]{$\ell_{n}$};
    \end{tikzpicture}
    \caption{The interval $I_{n}^{-}$ converges uniformly to the flow line $\gamma_{-}$ joining $x_{-}$ to some critical point, while $I_{n}^{+}$ converges uniformly to the flow line $\gamma_{+}$ joining some critical point to $x_{+}$.}
    \label{fig:first-step-done}
  \end{figure}
  Consider the complement $(a_{n},b_{n}):=[-\ell_{n},\ell_{n}]\setminus (I_{n}^{-}\cup I_{n}^{+})$. This interval has the property that $\mathrm{dist}(a_{n},-\ell_{n})\to \infty$ and $\mathrm{dist}(b_{n},\ell_{n})\to \infty$, and $v_{n}(a_{n},[0,\epsilon_{n}])$ and $v_{n}(b_{n},[0,\epsilon_{n}])$ converge uniformly to critical points.

  Our strategy has a few steps.

  Step 1 is to find a subinterval $I_{n}\subset (a_{n},b_{n})$ so that, for $I_{n}=s_{n}+J_{n}$ where $J_{n}$ is centered at $0$, the retranslated map $(s,t)\in J_{n}\times [0,\epsilon_{n}]\mapsto w_{n}(s,t)=v_{n}(s_{n}+s,t)$ converges uniformly to a \emph{nonconstant} flow line joining two critical points. If we cannot find such a nonconstant flow line, then we will show the \emph{Morse energy} (defined momentarily) of  $(a_{n},b_{n})$ converges to zero. 

  We iteratively repeat the argument to the two open intervals forming the complement $(a_{n},b_{n})\setminus I_{n}$. At each stage of the recursive process, we will either construct a closed interval $I_{n}'$ inside an open interval, and show that $v_{n}|I_{n}'$ converges to a non-constant flow line joining two critical points, or show that the open interval has Morse energy converging to zero. 

  Step 2 is to show that, whenever $I_{n}=s_{n}+J_{n}$, with $w_{n}(s,t)=v_{n}(s_{n}+s,t)$ (for $s\in J_{n}$), has $w_{n}$ converging uniformly to a non-constant flow line, then that flow line has a minimum positive amount of Morse energy $\hbar$.

  Thus the iteration starting of Step 1 will terminate after finitely many steps, and we will be left with a complementary region whose total Morse energy goes to zero.
  
  Step 3 is to show that sequences of open intervals $(a_{n},b_{n})$ have Morse energy convering to zero if and only if the \emph{diameter} of $v_{n}((a_{n},b_{n})\times [0,\epsilon_{n}])$ is also converging to zero. 

  The final step will be to define $\Sigma_{n}=\epsilon_{n}^{-1}I_{n}\times [0,1]$, and let $\sigma_{n}$ be the components of the complementary region $([-r_{n},r_{n}]\times [0,1])\setminus (\Sigma_{n}^{-}\cup\Sigma_{n}^{1}\cup \cdots \cup \Sigma_{n}^{k}\cup \Sigma_{n}^{+})$. This will complete the proof.
  
  Before we begin, define the Morse energy supported on $(a_{n},b_{n})\times [0,\epsilon_{n}]$ by:
  \begin{equation*}
    \mathrm{ME}(v_{n};(a_{n},b_{n}))=\frac{1}{\epsilon_{n}}\int_{0}^{\epsilon_{n}}\int_{a_{n}}^{b_{n}}v_{n}^{*}\omega\,\d s\d t=\frac{1}{\epsilon_{n}}\mathrm{E}(v_{n};(a_{n},b_{n})).
  \end{equation*}
  As in Proposition \ref{prop:low-e-strips}, the energy of $v_{n}$ supported on $(a_{n},b_{n})\times [0,\epsilon_{n}]$ is given by Stokes' theorem (see Figure \ref{fig:stokes}),
  \begin{equation*}
    \mathrm{E}(v_{n};(a_{n},b_{n}))=\epsilon_{n}(f_{n}(v_{n}(b_{n},\epsilon_{n}))-f_{n}(v_{n}(a_{n},\epsilon_{n})))+\text{error}_{n}.
  \end{equation*}
  The $\text{error}_{n}$ term is the integral of the primitive $\lambda_{n}$ over the two ends $\set{a_{n}}\times [0,\epsilon_{n}]$ and $\set{b_{n}}\times [0,\epsilon_{n}]$. Using $\lambda_{n}=-\lambda_{\mathrm{can}}+\epsilon_{n}\mathfrak{a}_{n}$ and the estimates:
  \begin{equation*}
    \abs{P_{n}}=O(\epsilon_{n})\text{ and }\abs{d v_{n}}=O(1),
  \end{equation*}
  which follow from the estimates \eqref{eq:my-estimates} (used in the proof of Lemma \ref{lemma:finite-number-case}), estimate:
  \begin{equation*}
    \abs{\int_{0}^{\epsilon}v_{n}^{*}\lambda_{n}}=\epsilon_{n}\max\abs{\lambda_{n}}\max\abs{\d v_{n}}=\epsilon_{n}O(\epsilon_{n})O(1)=o(\epsilon_{n}).
  \end{equation*}
  For this, use the fact that $\lambda_{\mathrm{can}}$ vanishes along the zero section (and hence is $O(\abs{P_{n}})$). Thus:
  \begin{equation*}
    \mathrm{ME}(v_{n};a_{n},b_{n})=f_{n}(v_{n}(b_{n},\epsilon_{n}))-f_{n}(v_{n}(a_{n},\epsilon_{n}))+o(1).
  \end{equation*}
  Since $v_{n}(a_{n},t)$ and $v_{n}(b_{n},t)$ are known to converge to critical points, say $y_{-}$ and $y_{+}$ (recall $a_{n}$ is the right boundary of $I_{n}^{-}$ and $b_{n}$ is the left boundary of $I_{n}^{+}$), we conclude that
  \begin{equation}\label{eq:morse-energy-limit}
    \lim_{n\to\infty}\mathrm{ME}(v_{n};a_{n},b_{n})=f(y_{+})-f(y_{-}).
  \end{equation}
  We will prove steps 1,2, and 3 from above.

  If the diameter of $v_{n}((a_{n},b_{n})\times [0,\epsilon_{n}])$ does not converge to zero, there is $s_{n}\in (a_{n},b_{n})$ so that $v_{n}(s_{n},0)$ converges to a non-critical point $p$ of $f$, after passing to a subsequence. This is because there are only finitely many critical points, and a minimum distance $\delta$ between any two critical points. For $r\le \delta$, any connected set of diameter $r$ must contain a non-critical point $p_{n}$ which remains $r/3$ far from all critical points (otherwise a set of diameter $r$ would be compressed into a ball of radius $r/3$, which is impossible due to the triangle inequality). We pick our subsequence so that $p_{n}$ converges to some point $p$.

  After passing to a further subsequence, the induced map $w_{n}(s,t)=v_{n}(s_{n}+s,t)$ converges on compact subsets to a flow line $\gamma:\R\to L$ passing with $\gamma(0)=p$. Since $v_{n}(a_{n},0)$ and $v_{n}(b_{n},0)$ converge to critical points, we conclude that $s_{n}-a_{n}$ and $b_{n}-s_{n}$ both diverge to $+\infty$ (i.e., $s_{n}$ remains far from the boundary of the interval under consideration).

  Similarly to how we have argued previously, pick any sequence $\rho_{k}\to \infty$, and for each $N\in \mathbb{N}$, define $k(N)$ to be the maximal $k$ satisfying
  \begin{enumerate}
  \item $\mathrm{dist}(w_{n}(s,t),\gamma(s))<1/k$ for $s,t\in [-\rho_{k},\rho_{k}]\times [0,\epsilon_{n}]$ and $n\ge N$,
  \item $s_{n}+[-\rho_{k},\rho_{k}]\subset (a_{n},b_{n})$ for $n\ge N$.
  \end{enumerate}
  By the above remarks, for any fixed $k$, the two conditions eventually hold for $N$ large enough, and so $k(N)\to \infty$. We then set $J_{n}=[-\rho_{k(n)},\rho_{k(n)}]$ and $I_{n}=s_{n}+J_{n}$. Then it is clear that the condition (ii) from the statement holds for $\Sigma_{n}=\epsilon_{n}^{-1}I_{n}\times [0,1]$. 

  Moreover, applying \eqref{eq:morse-energy-limit} to the interval $I_{n}$, we conclude that $\lim_{n\to\infty}ME(v_{n};I_{n})>0$, since $w_{n}$ converges to a \emph{non-constant} flow line.

  To summarize, if the diameter of $v_{n}((a_{n},b_{n})\times [0,\epsilon_{n}])$ does not converge to zero, then we there is a subinterval $I_{n}\subset (a_{n},b_{n})$ converging to a non-constant infinite flow line, and the Morse energy supported on $(a_{n},b_{n})$ converges to a positive number.

  If, on the other hand, the diameter \emph{does} converge to zero, then the Morse energy supported on $(a_{n},b_{n})$ will also converge to zero, as the limiting points $y_{\pm}$ will be the same. This completes the proof of Step 3.

  Write $I_{n}=[b_{n}',a_{n}']$, as shown in Figure \ref{fig:first-and-three-done}.
  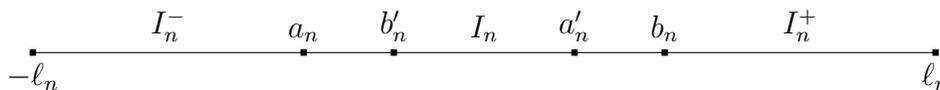
\begin{figure}[H]
    \centering
    \begin{tikzpicture}[xscale=2]
      \path[every node/.style={draw,inner sep=1pt,fill}] (0,0)node{}--node(A)[pos=0.3]{}node(Bp)[pos=0.4]{}node(Ap)[pos=0.6]{}node[pos=0.7](B){}(6,0)node{};
      \node at (A) [above] {$a_{n}$};
      \node at (B) [above] {$b_{n}$};
      \node at (Ap) [above] {$a^{\prime}_{n}$};
      \node at (Bp) [above] {$b^{\prime}_{n}$};
      \draw (0,0)node[below]{$-\ell_{n}$}--node[pos=0.5][above]{$I_{n}$}node[pos=0.15,above]{$I_{n}^{-}$}node[pos=0.85,above]{$I_{n}^{+}$}(6,0)node[below]{$\ell_{n}$};
    \end{tikzpicture}
    \caption{The interval $I_{n}$ converges uniformly to a \emph{nonconstant} flow line joining two critical points. We will then repeat the above arguments to $(a_{n},b_{n}')$ and $(a_{n}',b_{n})$.}
    \label{fig:first-and-three-done}
  \end{figure}
  Each interval $(a_{n},b_{n}')$ and $(a_{n}',b_{n})$ either has Morse energy going to zero, or we can find $I_{n}'$ contained within satisfying (i). In this fashion, the arguments given above can be iterated.

  For Step 2, it is clear each interval $I_{n}$ consumes a minimum amount $\hbar$ of Morse energy. Indeed one can take:
  \begin{equation*}
    \hbar := \inf\set{f(y_{+})-f(y_{-}):\text{ there exists a nonconstant flow line joining $y_{-}$ to $y_{+}$}}
  \end{equation*}
  which is positive as it is the infimum of a finite set of strictly positive numbers. This completes the proof.
\end{proof}


\subsubsection{Other cases of low energy strips}
\label{sec:other_low_energy_strips}
There are two other cases of low energy strips to consider. The first is when $u_{n}(s,0),u_{n}(s,1)\in L_{n}$ and $L_{n}$ converges to a Lagrangian $L$. In this case, one can apply Proposition \ref{prop:low-e-strips} with $\mathfrak{a}_{n}=\mathfrak{b}_{n}$. The following result classifies what happens when the first derivative tends to zero:
\begin{prop}\label{prop:low-e-strips-2}
  Let $K\subset (W,\omega)$ be a compact neighbourhood of a Lagrangian $L$, and let $J_{n}\to J$ be a $C^{\infty}$ convergent family of almost complex structures, and $L_{n}\to L$ be $C^{\infty}$ convergent family of Lagrangians. Suppose that $J$ is $\omega$-tame and $L$ is $J$-totally real. If $u_{n}:[-R_{n},R_{n}]\to K$ is a sequence of $J_{n}$-holomorphic curves with $u_{n}(s,0)\in L_{n}$ and $u_{n}(s,1)\in L_{n}$. If $\abs{\d u_{n}(s,t)}\to 0$ uniformly in $s,t$, as $n\to\infty$, then $u_{n}$ converges uniformly to a point.
\end{prop}
\begin{proof}
  The argument is similar to the proof of Proposition \ref{prop:low-e-cyl}. As in the proof of Proposition \ref{prop:low-e-strips}, let $\lambda_{n}$ be a primitive for $\omega$ which vanishes when restricted to $L_{n}$. For instance, we can fix a Weinstein neighbourhood of $L$ and take $\lambda_{n}=-\lambda_{\mathrm{can}}+\pr^{*}\beta_{n}$ where $L_{n}$ is the graph of the \emph{closed} one-form $\beta_{n}\to 0$. Note that $\lambda_{n}$ vanishes on $TW|_{L_{n}}$, not just $TL_{n}$. 

  Similarly to the proof of \ref{prop:low-e-cyl}, let
  \begin{equation*}
    E_{n}(r)=\int_{-r}^{r}\int_{0}^{1}u_{n}^{*}\omega=\int_{-r}^{r}\int_{0}^{1}\abs{\bd_{s}u}_{J,\omega}^{2}\d s\d t=\int_{0}^{1}\gamma_{r,n}^{*}\lambda_{n}-\int_{0}^{1}\gamma_{-r,n}^{*}\lambda_{n},
  \end{equation*}
  where $\abs{v}^{2}_{J,\omega}=\omega(v,J_{n}v)$ is strictly positive for $n$ sufficiently large. Note also that in our application of Stokes theorem, we use the fact that $u_{n}^{*}\lambda_{n}$ vanishes on \emph{both} horizontal boundary components.

  We also estimate $\gamma_{r,n}^{*}\lambda_{n}\le C\abs{\gamma'_{r,n}(t)}^{2}_{g}\d t\le C'\abs{\bd_{s}u}^{2}_{J,\omega}\d t$.

  One can then show that:
  \begin{equation*}
    E_{n}(r)\le CE_{n}'(r),
  \end{equation*}
  and the rest of the proof proceeds exactly as in the proof of \ref{prop:low-e-cyl}. 
\end{proof}

The other case is when the boundary conditions converge to different Lagrangians which have isolated intersections:
\begin{prop}\label{prop:low-e-strips-3}
  Let $K$ be a compact neighbourhood of $(W,\omega)$ containing two Lagrangians $L^{0},L^{1}$, with $L^{0}\cap L^{1}$ a finite set. Suppose that $J_{n}\to J$ is a convergent sequence of $\omega$-tame almost complex structures, and $u_{n}:[-R_{n},R_{n}]\times [0,1]\to K$ is sequence of $J_{n}$-holomorphic strips satisfying the boundary conditions:
  \begin{equation*}
    u_{n}(s,0)\in L_{n}^{0}\text{ and }u_{n}(s,1)\in L_{n}^{1},
  \end{equation*}
  where $L_{n}^{i}\to L^{i}$. Suppose moreover that:
  \begin{equation*}
    \limsup_{n\to \infty}\abs{\d u_{n}(-R_{n}+s,t)}+\abs{\d u_{n}(R_{n}-s,-t)}=0\text{ for $(s,t)$ in compact sets.}
  \end{equation*}
  Then there is a constant $\hbar>0$ depending on $L^{0},L^{1},J,\omega$ so that
  \begin{equation*}
    \limsup_{n\to\infty}E(u_{n})>0\implies \limsup_{n\to\infty}\sup_{s,t} \abs{\d u_{n}(s,t)}>0\implies \limsup_{n\to\infty}E(u_{n})\ge \hbar.
  \end{equation*}  
  Secondly, if $\abs{\d u_{n}(s,t)}\to 0$ uniformly in $s,t$ as $n\to\infty$, then $u_{n}$ converges uniformly to an intersection point $L^{0}\cap L^{1}$.
\end{prop}
\begin{proof}
  The proof relies on slightly different techniques from the proofs of Propositions \ref{prop:low-e-strips} and \ref{prop:low-e-strips-2}, and is a bit more topological in nature.

  Suppose that $\abs{\d u_{n}(s,t)}\to 0$ uniformly in $s,t$ as $n\to\infty$. Then Arzel\`a-Ascoli implies that for any sequence $s_{n}\in (-R_{n},R_{n})$ a subsequence $v_{n}(s,t)=u_{n}(s_{n}+s,t)$ converges uniformly on compact sets to a single point \emph{which must lie on} $L^{0}\cap L^{1}$. Since the set of intersection points is isolated, a standard subsequence argument implies that $u_{n}(s,t)$ converges uniformly to a single limit point, say $p$. This proves the second part of the proposition, and will be used to establish the first part.

  Near $p$, pick $\omega$-primitives $\lambda_{n}\to \lambda_{\infty}$ so that $\lambda_{n}|_{L^{0}_{n}}=0$ and $\lambda_{n}|_{L^{1}_{n}}=\d f_{n}$ with $f_{n}$ convergent. Stokes' theorem yields:
  \begin{equation*}
    E(u_{n},\omega)\le \abs{f_{n}(u_{n}(R_{n},1))-f_{n}(u_{n}(-R_{n},1))}+C\sup_{t}\abs{\d u_{n}(\pm R_{n},t)}.
  \end{equation*}
  Taking the limit $n\to\infty$, conclude that $\limsup_{n\to\infty}E(u_{n})=0$. This proves the contrapositive to the first implication in the proposition. To complete the proof, it remains to prove the implication involving $\hbar$. Define:
  \begin{equation*}
    \begin{aligned}
      \mathscr{U}_{0}(K,J)&=\set{\text{nonconstant $J$-holomorphic spheres $S^{2}\to K$}},\\
      \mathscr{U}_{1}(K,J,L_{0})&=\set{\text{nonconstant $J$-holomorphic disks $(D(1),\bd D(1))\to (K,L_{0})$}},\\
      \mathscr{U}_{2}(K,J,L_{1})&=\set{\text{nonconstant $J$-holomorphic disks $(D(1),\bd D(1))\to (K,L_{1})$}},\\
      \mathscr{U}_{3}(K,J,L_{0},L_{1})&=\set{\text{nonconstant $J$-holomorphic strips $\R\times [0,1]\to (K,L_{0},L_{1})$}},
    \end{aligned}
  \end{equation*}
  where in the last case we require the top boundary is mapped to $L_{1}$ and the bottom boundary to $L_{0}$. Then we set:
  \begin{equation*}
    \hbar(K,L,J,\omega)=\inf\set{E(u,\omega)\text{ such that }u\in \mathscr{U}_{0}\cup \mathscr{U}_{1}\cup \mathscr{U}_{2}\cup \mathscr{U}_{3}}.
  \end{equation*}
  It follows from the mean-value property that the infimum over $\mathscr{U}_{0}\cup \mathscr{U}_{1}\cup \mathscr{U}_{2}$ is bounded below. To bound the infimum over $\mathscr{U}_{3}$, argue by contradiction: if not, there is a sequence $u_{n}\in \mathscr{U}_{3}$ whose energy tended to zero. The mean-value theorem still applies in this setting, indeed, on any half-disk of sufficiently small radius $u_{n}$ is a half-disk with boundary on $L_{0}$ or $L_{1}$. Thus $\abs{\d u_{n}(s,t)}$ converges uniformly to $0$, and, as argued above, $u_{n}$ converges uniformly to a point $p\in L^{0}\cap L^{1}$. 

  Let $n$ be large enough so that $u_{n}$ remains entirely in a small ball $B$ around $p$. We require that $p$ is the only intersection point of $L_{0}$ and $L_{1}$ in the ball. By applying similar arguments to the translations $u_{n,k,\pm}(s,t)=u_{n}(s\pm k,t)$, we conclude that $u_{n,k,\pm}(s,t)$ must both converge to $p$ as $k\to\infty$, uniformly in $t$. Writing $\omega=\d\lambda$, where $\lambda|_{L_{0}}=0$ and $\lambda|_{L_{1}}=\d f$ (on $B$), Stokes' theorem implies that $E(u_{n},\omega)=0$. Thus $u_{n}$ was not non-constant, contradicting the definition of $\mathscr{U}_{3}$. Thus $\hbar>0$.

  Finally, if $\limsup_{n}\sup_{s,t}\abs{\d u_{n}(s,t)}>0$, pick $s_{n}$ so that a subsequence of $u_{n}(s_{n}+s,t)$ either (i) forms a sphere or disk bubble (as in \ref{lemma:basic_bubble}), or (ii) converges on compact sets to a curve in $\mathscr{U}_{3}$. In either case, $\limsup_{n}E(u_{n},\omega)\ge \hbar$, as desired. This completes the proof.
\end{proof}
\appendix

\section{On elliptic regularity}
\label{sec:proof-of-bootstrappingone}

The proof of Lemma \ref{lemma:bootstrappingone} requires three analytical prerequisites: the \emph{Sobolev embedding theorem}, the \emph{elliptic estimates for the Laplacian}, and \emph{quadratic estimates for $W^{k,2}$ spaces}.

\begin{lemma}[Sobolev embedding theorem]\label{lemma:prereq1}
  For every bounded Lipshitz domain $\Omega\subset \R^{2}$ there exists constants $c_{2}(\Omega)$ $c_{1}(\Omega)>0$ so that
  \begin{equation*}
    \norm{f}_{C^{0}(\Omega)}\le c_{1}\norm{f}_{W^{1,4}(\Omega)}\le c_{2}\norm{f}_{W^{2,2}(\Omega)}.
  \end{equation*}
\end{lemma}
\begin{proof}
  See \cite[Theorem B.1.11]{mcduffsalamon} for a more general result.
\end{proof}
\begin{lemma}[Elliptic estimates for the Laplacian]\label{lemma:prereq2}
  For every pair of domains $\Omega_{1},\Omega_{2}\subset \mathbb{H}$ with $\cl{\Omega}_{1}\subset \Omega_{2}$ there exists a constant $c(\ell,\Omega_{1},\Omega_{2})$ so that
  \begin{equation*}
    \norm{u}_{W^{k+2,2}(\Omega_{1})}\le c(\norm{\Delta u}_{W^{k,2}(\Omega_{2})}+\norm{u}_{W^{k+1,2}(\Omega_{2})}),
  \end{equation*}
  for all smooth functions $u:\Omega_{2}\to \R^{d}$ satisfying the Dirichlet boundary conditions $u(\R\cap \Omega_{2})=0$ or the Neumann boundary conditions $\bd_{t}u(\R\cap \Omega_{2})=0$.
\end{lemma}
\begin{proof}
  See \cite[Lemma C.2]{robbinsalamon} for a short proof. This also follows from the general $L^{p}$, $p>1$, results in \cite[Appendix B]{mcduffsalamon}.
\end{proof}
\begin{lemma}\label{lemma:prereq3}
  Let $\Omega$ be a bounded Lipshitz domain. There are constants $Q_{k}=Q_{k}(\Omega)$ so that for all smooth functions $f,g$ on $\Omega$ we have
  \begin{equation}\label{eq:baby-quadratic}
    \norm{fg}_{W^{1,2}(\Omega)}\le Q_{1}(\norm{f}_{W^{1,2}(\Omega)}\norm{g}_{C^{0}(\Omega)}+\norm{f}_{C^{0}(\Omega)}\norm{g}_{W^{1,2}(\Omega)}),
  \end{equation}
  and for $k\ge 2$ we have
  \begin{equation}\label{eq:quadratic}
    \norm{fg}_{W^{k,2}(\Omega)}\le Q_{k}\norm{f}_{W^{k,2}(\Omega)}\norm{g}_{W^{k,2}(\Omega)}.
  \end{equation}
\end{lemma}
\begin{proof}
  The estimate on the $W^{1,2}$ norm holds by observing that
  \begin{equation*}
    \nabla (fg)=(\nabla f)\cdot g+f\cdot \nabla g,
  \end{equation*}
  and the estimate $\norm{ab}_{L^{2}}\le \norm{a}_{L^{2}}\norm{b}_{C^{0}}$. The estimate on the $W^{2,2}$ spaces follows from:
  \begin{equation*}
    \norm{fg}_{W^{2,2}(\Omega)}\le \norm{fg}_{W^{1,2}(\Omega)}+\norm{\nabla f\cdot \nabla g}_{L^{2}}+\norm{\nabla \nabla f\cdot g}_{L^{2}}+\norm{f\cdot \nabla \nabla g}_{L^{2}}.
  \end{equation*}
  We claim that each term can be bounded by some constant times $\norm{f}_{W^{2,2}}\norm{g}_{W^{2,2}}$. The first term can be estimated using the quadratic estimate on the $W^{1,2}$ norm \eqref{eq:baby-quadratic}, together with the Sobolev embedding theorem for $C^{0}\subset W^{2,2}$. The last two terms can be estimated using $\norm{ab}_{L^{2}}\le \norm{a}_{L^{2}}\norm{b}_{C^{0}}$ and the Sobolev embedding theorem. The hard term to estimate is $\norm{\nabla f\cdot \nabla g}_{L^{2}}$. To do so, use the Sobolev embedding theorem for $W^{1,4}\subset W^{2,2}$, and the H\"older-type inequality:
  \begin{equation*}
    \norm{\nabla f\cdot \nabla g}_{L^{2}}\le \norm{\nabla f}_{L^{4}}\norm{\nabla g}_{L^{4}}.
  \end{equation*}
  The quadratic estimates for $k>2$ follow easily by induction, using:
  \begin{equation*}
    \norm{fg}_{W^{k,2}}\le \norm{fg}_{W^{k-1,2}}+\norm{\nabla f\cdot g}_{W^{k-1,2}}+\norm{f\nabla g}_{W^{k-1,2}}.
  \end{equation*}
  This completes the proof.  
\end{proof}
With these prerequisites the proof Lemma \ref{lemma:bootstrappingone} is a straightforward bootstrapping argument. See \cite[Lemma C.3]{robbinsalamon}.

\begin{proof}[of Lemma \ref{lemma:bootstrappingone}]
  In search of a contradiction, suppose a subsequence $u_{n}$ whose energy decays to zero, but $\abs{\nabla^{\ell}\d u_{n}(z_{n})}$ is bounded from below, say $\epsilon_{\ell}>0$. Replace $u_{n}$ by this subsequence.

  Write $z_{n}=s_{n}+it_{n}$. Passing to a subsequence, either $t_{n}$ converges to $0$, or $z_{n}+it_{n}$ remains a minimal distance $\delta>0$ from $\bd\mathbb{H}$. We will prove the case $t_{\infty}=0$, and leave the other (simpler) case to the reader.

  Consider the function $v_{n}(s+it)=u_{n}(s_{n}+s+it)$. Since $z_{n}=s_{n}+it_{n}$ and $t_{n}$ converges to $0$, eventually $v_{n}$ is defined on the half-disk $D(0,r/2)\cap \cl{\mathbb{H}}$.
  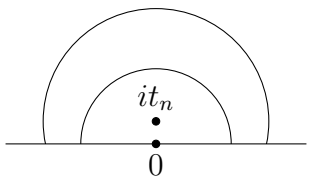
\begin{figure}[H]
    \centering
    \begin{tikzpicture}
      \begin{scope}
        \clip (-2,0)rectangle (2,1.9);
        \draw (0,0.3) circle (3/2) node (A){};
        \draw (0,0) circle (1);
      \end{scope}
      \draw (-2,0)--(2,0);
      \node[draw,circle,fill,inner sep=1pt] (Z)at (0,0) {};
      \node[draw,circle,fill,inner sep=1pt] at (A) {};
      \node at (A) [above] {$it_{n}$};
      \node at (Z) [below] {$0$};
    \end{tikzpicture}
    \caption{The half disk $D(0,r/2)\cap \cl{\mathbb{H}}$ is eventually contained in $D(it_{n},r)\cap \cl{\mathbb{H}}$.}
    \label{fig:half-disk}
  \end{figure}

  Since $u_{n}$ is valued in the compact neighbourhood $K$, so is $v_{n}$. Therefore we may pass to a subsequence where $v_{n}(0)\in L$ converges to a point $p\in L$. Choose now a coordinate chart $\varphi:\cl{U}\to \cl{B}\subset \R^{2d}$ centered at $p$ which identifies $L\cap \cl{U}$ with $(\R^{d}\times \set{0})\cap \cl{B}$ and so that the induced complex structure $\d\varphi\cdot J\cdot \d\varphi^{-1}$ is equal to $J_{0}$ along $\R^{d}\times \set{0}$.\footnote{To see that such a coordinate chart exists one can, e.g., pick the first $d$ coordinates $x_{1},\cdots,x_{d}$ for $L$ and then define the remaining coordinates $y_{1},\cdots,y_{d}$ by exponentiating the vector fields $J\bd_{x_{i}}$ (which are transverse to $L$ since $J$ is compatible with $\omega$).}

  The mean-value property for the energy density implies that the first derivative of $v_{n}$ is eventually bounded on its domain $D(0,r/2)\cap \cl{\mathbb{H}}$. In particular, there is a $\delta>0$ so that $v_{n}$ eventually maps $D(0,\delta)\cap\cl{\mathbb{H}}$ into $U$. Thus we may (eventually) define the $\R^{2d}$-valued function $w_{n}(z)=\varphi\circ v_{n}(z)$. 

  Then, abusing notation and letting $J:=\d\varphi\cdot J\cdot \d\varphi^{-1}$, conclude that $w_{n}$ satisfies the boundary value problem:
  \begin{equation*}\label{eq:bvp2}
    \left\{
      \begin{aligned}
        \bd_{s}w_{n}+J(w_{n})\cdot \bd_{t}w_{n}=0&\\
        w_{n}(s,0)\in \R^{d}\times \set{0}&
      \end{aligned}
    \right.
  \end{equation*}
  Decompose $w_{n}=(X_{n},Y_{n})$ into its real and imaginary parts, and compute $Y_{n}(s,0)=0$ and $0=\bd_{s}Y_{n}(s,0)=-\bd_{t}X_{n}(s,0)$. Thus $X_{n}$ satisfies Neumann boundary conditions and $Y_{n}$ satisfies Dirichlet boundary conditions. Therefore Lemma \ref{lemma:prereq2} implies that, for $k\ge 2$, $w_{n}$ satisfies the elliptic estimates:
  \begin{equation}\label{elliptic-estimates}
    \norm{w_{n}}_{W^{k,2}(\Omega_{k})}\le c_{k}(\norm{\Delta w_{n}}_{W^{k-2,2}(\Omega_{k-1})}+\norm{w_{n}}_{W^{k-1,2}(\Omega_{k-1})}),
  \end{equation}
  where we define $\Omega_{k}:=D(0,\delta/k)\cap \cl{\mathbb{H}}$. In order to use \eqref{elliptic-estimates}, we compute
  \begin{equation}\label{eq:delta-wn}
    \begin{aligned}
      &(\bd_{s}-J(w_{n})\bd_{t})(\bd_{s}w_{n}+J(w_{n})\bd_{t}w_{n})=0\\
      \implies &\Delta w_{n}=\bd_{t}[J(w_{n})]\bd_{s}w_{n}-\bd_{s}[J(w_{n})]\bd_{t}w_{n}.
    \end{aligned}
  \end{equation}
  The strategy is to use \eqref{elliptic-estimates} and \eqref{eq:delta-wn} to bootstrap the the $C^{1}$ bounds from the mean-value property to $W^{k,2}$ bounds on the half-disk $\Omega_{k}$. To be more precise, we will prove:
  \begin{equation}\label{eq:desired2}
    \sup_{n}\norm{w_{n}}_{W^{k,2}(\Omega_{k})}=0,
  \end{equation}
  by induction on $k$. It then follows from the Sobolev embedding theorem that the case $k=\ell+3$ will contradict our assumption that $\abs{\nabla^{\ell}\d u_{n}(z_{n})}\ge \epsilon_{\ell}$.

  Because of the mean-value property, the assumption that the energy tends to zero implies that the derivative of $w_{n}$ is converging to $0$ on its domain $\Omega_{1}$. Thus equation \eqref{eq:delta-wn} implies that $\Delta w_{n}$ is convergent to $0$ in $L^{2}(\Omega_{1})$. Therefore the elliptic estimates \eqref{elliptic-estimates} imply that \eqref{eq:desired2} holds with $k=2$.

  It is well-known that
  \begin{equation}
    \label{eq:smooth-compo}
    \sup_{n}\norm{w_{n}}_{W^{k,2}(\mathscr{D}(r))}<\infty\implies \sup_{n}\norm{J(w_{n})}_{W^{k,2}(\mathscr{D}(r))}<\infty,
  \end{equation}
  for all $k\ge 0$, since $J$ is smooth.
  
  Apply the quadratic estimate \eqref{eq:baby-quadratic} to the equation \eqref{eq:delta-wn} to conclude $$\limsup_{n}\norm{\Delta w_{n}}_{W^{1,2}(\Omega_{2})}=\limsup_{n}\norm{\bd_{t}[J(w_{n})]\bd_{s}w_{n}-\bd_{s}[J(w_{n})]\bd_{t}w_{n}}_{W^{1,2}(\Omega_{2}}=0.$$ This uses the fact that $J(w_{n})$ is uniformly bounded in $W^{2,2}(\Omega_{2})$ and $C^{1}$, and $w_{n}$ is converging to $0$ in $W^{2,2}(\Omega_{2})$ and $C^{1}$. Then the elliptic estimates \eqref{elliptic-estimates} imply that the desired result \eqref{eq:desired2} holds with $k=3$.

  Now conclude from \eqref{eq:smooth-compo} and the higher quadratic estimates \eqref{eq:quadratic} that
  \begin{equation*}
    \limsup_{n}\norm{\Delta w_{n}}_{W^{2,2}(\Omega_{2})}=\limsup_{n}\norm{\bd_{t}[J(w_{n})]\bd_{s}w_{n}-\bd_{s}[J(w_{n})]\bd_{t}w_{n}}_{W^{2,2}(\Omega_{2})}=0.
  \end{equation*}
  Then the elliptic estimates \eqref{elliptic-estimates} prove \eqref{eq:desired2} in the case $k=4$. The argument repeats, without any further modification, to conclude \eqref{eq:desired2} for all $k$.

  The proof is almost over. Recall the assumption that $$\limsup_{n\to\infty}\abs{\nabla^{\ell}\d u_{n}(z_{n})}\ge \epsilon_{\ell}.$$ Since $\varphi\circ u_{n}(z)=w_{n}(z-s_{n})$, $w_{n}$ is also bounded below\footnote{Since $\varphi$ is a diffeomorphism between compact domains, it distorts the $C^{\ell+1}$ size by a bounded amount.} in the $C^{\ell+1}$ norm near $z_{n}-s_{n}=it_{n}$.

  Since $t_{n}$ converges to $0$, $it_{n}$ eventually enters the disk $\Omega_{\ell+3}$. However, the $C^{\ell+1}(\Omega_{\ell+3})$ norm is bounded by the $W^{\ell+3,2}(\Omega_{\ell+3})$ norm (by the Sobolev embedding theorem). Then \eqref{eq:desired2} contradicts the fact that the $C^{\ell+1}$ size of $w_{n}$ is bounded below. This contradiction completes the proof.
\end{proof}

The next result we prove in this appendix is another elliptic type estimate.

\begin{lemma}\label{lemma:baby-ellreg}
  Let $\Phi_n:D(1)\to \C$ be a sequence of smooth maps satisfying

  (1) $\sup_{n,z}\abs{\d\Phi_n(z)}<\infty$, (2) $\Phi_n(0)=0$, and (3) $\Phi_n$ is approximately holomorphic in the sense that
  \begin{equation}
    \label{eq:appx-holo}
    \bd_s\Phi_n(z)+i\bd_t\Phi_n(z)=A_n(z)\cdot \nabla \Phi_n(z)+B_{n}(z),
  \end{equation}
  where $A_n,B_{n}\to 0$ in the $C^\infty_{\mathrm{loc}}$ topology. Then the $k$th derivative of $\Phi_n$ is bounded on any smaller disk $D(\rho)$ ($\rho<1$) as $n\to\infty$. 
\end{lemma}
\begin{proof}
  The proof is very similar to the proof of the preceding lemma, and so we only sketch the result. One shows that $\norm{\Phi_{n}}_{W^{k,2}(D(\rho))}$ as $n\to\infty$ for all $k$ and $\rho<1$ using the elliptic estimates for $\Delta$. Indeed, by applying $\bd := \bd_{s}-i\bd_{t}$ to both sides of the above equation, we conclude that
  \begin{equation*}
    \Delta \Phi_{n}=\bd (A_{n}\cdot \nabla \Phi_{n})+\bd B_{n}.
  \end{equation*}
  The $C^{1}$ bound on $\Phi_{n}$ implies that $\Delta\Phi_{n}$ is bounded in $L^{2}$, and hence $\Phi_{n}$ and $\bd \Phi_{n}$ are bounded in $W^{1,2}$. Then the equation implies $\Delta\Phi_{n}$ is bounded in $W^{1,2}$, and hence $\Phi_{n}$ and $\bd\Phi_{n}$ are bounded in $W^{2,2}$, etc. To perform these estimates, one should use
  \begin{equation*}
    \norm{\bd (A_{n}\cdot \nabla \Phi_{n})}_{L^{2}(\Omega')}\le \epsilon_{n}\norm{\Phi_{n}}_{W^{2,2}(\Omega)},
  \end{equation*}
  where $\epsilon_{n}\to 0$. This completes the proof.  
\end{proof}
\section{Removal of singularities}
\label{sec:removal-of-singularities}
In this section we show how the $C^{1}$ exponential estimates from Lemma \ref{lemma:c1estimates} imply the \emph{removal of singularities} result for $J$ holomorphic curves $u:[0,\infty)\times [0,1]\to (K,L)$ with finite energy (a result we have used multiple times in this paper). Note that we work with a single Lagrangian $L$. A similar argument works for holomorphic curves $[0,\infty)\times \R/\Z\to K$.

The idea is to apply Lemma \ref{lemma:c1estimates} with $\mathfrak{a}_{n}=\mathfrak{b}_{n}=0$ and define a sequence of strips by
\begin{equation*}
  u_{n}:[-R_{n}-1,R_{n}+1]\times [0,1]\to (K,L)\text{ given by }u_{n}(s,t)=u(s+s_{n},t),
\end{equation*}
with arbitrary sequences $s_{n}\to\infty$ with $2R_{n}=s_{n}$. Since $\mathfrak{a}_{n}=\mathfrak{b}_{n}$, we are free to pick an arbitrary sequence $\epsilon_{n}\to 0$.

It is easy to show that the energy of $u_{n}$ is tending to $0$, because of our assumption that $R_{n}=s_{n}-R_{n}$ is the left endpoint of the strip. As a consequence, we can apply Lemma \ref{lemma:c1estimates} to conclude the following $C^{1}$ exponential bound on $P_n=\tilde{P}_n$ (for $n$ sufficiently large)
\begin{equation*}
  \abs{P_{n}(s+s_{n},t)}+\abs{\nabla_{s} P_n(s+s_{n},t)}+\abs{\nabla_{t} P_n(s+s_{n},t)}\le \kappa_{n}(e^{-d(R_{n}+s)}+e^{-d(R_{n}-s)}+\epsilon_{n}),
\end{equation*}
where $\kappa_{n}=\kappa_{n}(u_{n})$ converges to $0$. Crucially for us, $\kappa_{n}$ is independent of the arbitrary sequence $\epsilon_{n}$, and hence we can take the limit $\epsilon_{n}\to 0$ in the above estimate and conclude
\begin{equation*}
  \abs{P_{n}(s+s_{n},t)}+\abs{\nabla_{s} P_n(s+s_{n},t)}+\abs{\nabla_{t} P_n(s+s_{n},t)}\le \kappa_{n}(e^{-d(r_{n}+s)}+e^{-d(r_{n}-s)}),
\end{equation*}

Recalling equations \eqref{eq:perturbJ}, we similarly conclude that $\abs{\bd_{s}Q}$ and $\abs{\bd_{t}Q}$ decay exponentially in the same manner. This implies that we can bound
\begin{equation}\label{eq:subsequence-bound}
  \abs{\d u(s_{n},t)}\le a_{n}e^{-dR_{n}}=a_{n}e^{-(d/2)s_{n}}
\end{equation}
for $a_{n}\to 0$. Since $s_{n}$ was arbitrary, this implies
\begin{equation*}
  \lim_{s\to\infty}e^{(d/2) s}\abs{\d u(s,t)}=0.
\end{equation*}
In other words, if the above limit was not zero, then we could find some sequence $s_{n}\to \infty$ for which it was bounded below, contradicting the above argument.

With these preparations in place, we now introduce the function $w:\Omega^{\times}\to (K,L)$ defined by
\begin{equation*}
  w(e^{-i\pi(s+it)})=u(s,t),
\end{equation*}
where $\Omega=D(1)\cap (-\cl{\mathbb{H}})$ and $\Omega^{\times}$ is obtained by removing the origin from $\Omega$.
\begin{claim}
  The map $w$ admits a holomorphic extension to $\Omega\to (K,L)$.
\end{claim}
\begin{proof}
  First we observe that the exponential bound on $u$ implies
  \begin{equation*}
    \lim_{n\to\infty}\mathrm{diam}(u([n,\infty)\times [0,1]))=0,
  \end{equation*}
  and hence $w$ admits a continuous extension to $\Omega$.

  The strategy is to prove that $\abs{\d w}^{p}$ is integrable over $\Omega$ for $p>2$. This will enable us to appeal to the ``$W^{1,p}\implies \mathrm{smooth}$'' non-linear elliptic regularity theory to conclude that $w$ is smooth.

  We compute
  \begin{equation*}
    \abs{\d w(e^{-\pi(s+it)})}=e^{\pi s}\abs{\d u(s,t)},
  \end{equation*}
  and hence
  \begin{equation*}
    \begin{aligned}
      \int_{\Omega^{\times}} \abs{\d w(x,y)}^{p}\d x\wedge \d y
      &=\int_{0}^{\infty}\int_{0}^{1}e^{p\pi s}\abs{\d u(s,t)}^{p}\varphi^{*}\d x\wedge \d y,\\
      &=\int_{0}^{\infty}\int_{0}^{1}e^{p\pi s}\abs{\d u(s,t)}^{p}e^{2\pi s}\d x\wedge \d y<\infty,
    \end{aligned}
  \end{equation*}
  where $\varphi(s,t)=e^{-\pi(s+it)}$. Now, using the fact that $\abs{\d u(s,t)}^{p}$ decays exponentially with rate $e^{-p(d/2)s}$, we conclude that the right hand side is integrable for \emph{some} value of $p>2$. As explained in \cite[Appendix B]{mcduffsalamon}, the $W^{1,p}$ regularity of $w$ implies that it extends holomorphically to $\Omega$, as desired.
\end{proof}

\section{Hofer's bubbling lemma}
\label{sec:hofer-bubbling}
Suppose that $(\Sigma_{n},\mathfrak{e}_{n},\Theta_{n},u_{n})$ has tame ends (in the sense of \S\ref{sec:tame_sequences}) and $(\Sigma_{n},\Theta_{n},\mathfrak{e}_{n})$ converges with marked points to $(\Sigma_{\infty},\Theta_{\infty})$ via a map $\psi_{n}:\Sigma_{n}\to \Sigma_{\infty}$. The next lemma shows that we can add finitely many points $\Gamma$ to $\Sigma_{\infty}$ to make the derivative bounded on compact subsets of $\Sigma_{\infty}\setminus(\Theta_{\infty}\cup \Gamma)$. 

\begin{lemma}\label{lemma:basic_bubble}
  There exists a finite set $\Gamma\subset \Sigma_{\infty}\setminus \Theta_{\infty}$ so that $w_{n}=u_{n}\circ \psi^{-1}_{n}$ has bounded first derivative on compact subsets of $\Sigma_{\infty}\setminus (\theta_{\infty}\cup \Gamma)$, after passing to a subsequence.
\end{lemma}
\begin{proof}
  The construction of $\Gamma$ is iterative, beginning with $\Gamma=\emptyset$. At each stage of the iterative process, if there exists a compact set $K$ disjoint from $\Gamma\cup \Pi_{\infty}$ and a sequence $z_{n}\in K$ so that $\abs{\d w_{n}(z_{n})}$ is unbounded, then we will add a point to $\Gamma$ (the point we add will be a limit point of the sequence $z_{n}$), pass to a subsequence of $w_{n}$, and repeat the iteration. On the other hand, if the derivative of $w_{n}$ is bounded on each compact set disjoint from $\Gamma\cup \Pi_{\infty}$, then the iterative process terminates, and the proof is complete.

  We now explain how to add the points to $\Gamma$. Suppose that $K$, $z_{n}$ are as above. By passing to a subsequence, we may suppose that $z_{n}$ converges to a limit $\zeta$ and $\abs{\d w_{n}(z_{n})}$ diverges to $+\infty$. Clearly $\zeta$ is disjoint from $\Gamma\cup \Pi_{\infty}$. We then set $\Gamma:=\Gamma\cup \set{\zeta}$.

  Clearly this iterative construction produces a set $\Gamma$ and a subsequence of $w_{n}$ with the property that, for each $\zeta\in \Gamma$, there exists $z_{n}^{\zeta}\to \zeta$ so that $\abs{\d w_{n}(z_{n}^{\zeta})}\to +\infty$.

  To complete the proof, it suffices to show that the process must terminate. Morally, the idea is that each point $\Gamma$ will be the location where a \emph{bubble} forms. A \emph{bubble} is a non-constant holomorphic map $\cl{\mathbb{H}}\to K$ or $\mathbb{C}\to K$. It follows from the mean-value property that bubbles consume $\omega$-energy at least $\hbar>0$, where $\hbar$ is small enough that
  \begin{equation}\label{eq:comparison}
    \hbar\le \epsilon_{0}\frac{\omega(v,J_{n}v)}{g(v,v)}\text{ for all $v\in TK$}
  \end{equation}
  for some fixed Riemann metric $g$. Here $\epsilon_{0}=\epsilon_{0}(K,J,L,g)$ is the constant guaranteed by Lemma \ref{lemma:mvp}. Note that this implies that: $$\int u_{n}^{*}\omega<\hbar\implies \int \abs{\bd_{s}u}^{2}\d s\d t<\epsilon_{0}\implies \text{mean-value theorem applies}$$ for all $J_{n}$-holomorphic maps $u_{n}:\Omega\to K$ for $\Omega\subset \C$. 

  Suppose that $N\hbar$ is greater than the limit supremum of the energy of $u_{n}$. We will show that $\abs{\Gamma}<N$, arguing by contradiction. Thus, suppose there are $N$ distinct points $\zeta_{1},\dots,\zeta_{N}$ in $\Gamma$.

  Around each point $\zeta_{k}$ we apply the results of \S\ref{sec:coordinate_disks} to obtain a convergent family of coordinate disks $c_{n}^{k}:D(1)\to \Sigma_{\infty}^{\mathrm{int}}\setminus \Theta_{\infty}$ or half-disks $(D(1)\cap \cl{\mathbb{H}},D(1)\cap \mathbb{R})\to (\Sigma_{\infty}\setminus \Theta_{\infty},\bd\Sigma_{\infty})$. We may suppose that all the coordinate disks are disjoint. Let us focus on the half-disk case, as the full-disk case is easier. As usual, we write $\Omega(r):=D(r)\cap \cl{\mathbb{H}}$ and $\bd\Omega(r)=D(1)\cap \mathbb{R}$. 
  
  Since $u_{n}$ is holomorphic and $\varphi_{n}\circ c_{n}^{k}$ is holomorphic, we conclude that $w_{n}\circ c_{n}^{k}$ is a holomorphic map $(\Omega(1),\bd \Omega(1))\to (K,L_{n},J_{n})$. It is clear that $w_{n}^{k}:=w_{n}\circ c_{n}^{k}$ has an unbounded first derivative as $n\to\infty$. In particular, we can find a sequence of points $z_{n}'\in \Omega(1/2)$ so that $\abs{\d w_{n}^{k}(z_{n}')}\to \infty$.

  One technical result needed is known as ``Hofer's lemma:''
  \begin{lemma}[Hofer's Lemma]\label{lemma:hofers-lemma}
    Let $d:X\to [0,\infty)$ be a continuous function on a complete metric space; and let $\epsilon'>0$ and $x'\in X$. One can find $0<\epsilon\le \epsilon'$ and $x\in X$ so that

    (i) $\mathrm{dist}(x,x')<2\epsilon'$,
    
    (ii) $d(y)\le 2 d(x)$ for all $y\in D(x,\epsilon)$.
    
    (iii) $\epsilon d(x)\ge \epsilon' d(x')$,
  \end{lemma}
  Hofer's lemma was introduced in \cite[Lemma 3.3]{hofer92} (moreover, they show that the lemma gives a \emph{characterization} of completeness).
  \begin{proof}[of Lemma \ref{lemma:hofers-lemma}]
    Let $\epsilon_{n}=2^{-n}\epsilon'$, and define a (potentially terminating) sequence $x_{n}$ as follows: let $x_{0}=x'$, and choose $x_{n+1}\in D(x_{n},\epsilon_{n})$ so that $d(x_{n+1})> 2d(x_{n})$. If no such $x_{n+1}$ exists (i.e., the sequence terminates at $x_{n}$), then we conclude that, for all $y\in D(x_{n},\epsilon_{n})$ we have $d(y)\le 2d(x_{n})$, so (ii) is satisfied with $x=x_{n}$, $\epsilon=\epsilon_{n}$. By construction, we have
    \begin{equation*}
      \epsilon_{n}d(x_{n})\ge 2\epsilon_{n}d(x_{n-1})=\epsilon_{n-1}d(x_{n-1})\ge \cdots \ge \epsilon_{0}d(x_{0})=\epsilon'd(x'),
    \end{equation*}
    so (iii) would also be satisfied. Since $\mathrm{dist}(x_{0},x_{n})\le \epsilon_{0}+\epsilon_{1}+\cdots+\epsilon_{n}\le 2\epsilon_{0}$, we conclude (i) also holds.

    Thus the proof of the lemma is reduced to proving that the above recursion terminates. In search of a contradiction suppose it does not converge. Then the sequence $x_{n}$ converges, however, $d(x_{n})$ is unbounded since $d(x_{n})>2d(x_{n-1})$. This is impossible, and so we complete the proof.
  \end{proof}

  Now returning to the proof, consider $k=1$, let $R_{n}':=\abs{\d w_{n}^{1}(z_{n}')}$, and pick $0<\epsilon_{n}'<1/6$ so that $\lim_{n\to\infty}\epsilon_{n}'=0$ but $\lim_{n\to\infty}\epsilon_{n}'R_{n}'=+\infty$.

  Introduce the function $d_{n}(z)=\abs{\d w^{1}_{n}(z)}$, and apply Hofer's lemma with $\epsilon'=\epsilon_{n}'$ and $x'=z_{n}'$ to conclude $\epsilon_{n}\le \epsilon_{n}'$ and $z_{n}$ so that
  \begin{enumerate}
  \item $\mathrm{dist}(z_{n},z_{n}')<2\epsilon_{n}'$,
  \item\label{item:hc2} $\abs{\d u_{n}(y)}\le 2\abs{\d w^{1}_{n}(z_{n})}$ for $y\in D(z_{n},\epsilon_{n})\cap \cl{\mathbb{H}}$,
  \item\label{item:hc3} $\epsilon_{n}\abs{\d w^{1}_{n}(z_{n})}\ge \epsilon_{n}'R_{n}'$.
  \end{enumerate}
  The reader may complain that $d_{n}$ is not defined on a complete metric space, but it is easy to see that every point and ball considered in the recursive proof of Hofer's lemma will remain entirely in $D(z_{n}',3\epsilon_{n}')$. Since we chose $\epsilon_{n}'<1/6$, we see that we can cut off $d_{n}$ outside of $\mathscr{D}(z_{n}',3\epsilon_{n}')\subset \Omega(1)$ (and obtain a continuous function defined on all of $\cl{\mathbb{H}}$) without affecting our conclusions.
  
  We abbreviate $R_{n}:=\abs{\d w_{n}^{1}(z_{n})}$. Note that by item \ref{item:hc3} $R_{n}$ is still diverging to $\infty$.

  The idea now is to rescale the domains of $w^{1}_{n}$ by the factor of $R_{n}^{-1}$; we introduce
  \begin{equation*}
    z\in D(0,\epsilon_{n}R_{n})\cap g_{n}^{-1}\cl{\mathbb{H}}\mapsto v_{n}^{1}(z):=w^{1}_{n}(z_{n}+R_{n}^{-1}z),
  \end{equation*}
  where $g_{n}(z)=z_{n}+R_{n}^{-1}z$.
  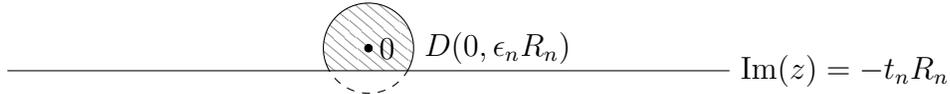
\begin{figure}[H]
    \centering
    \begin{tikzpicture}[scale=0.6]
      \begin{scope}
        \clip (-2,-0.5) rectangle (2,1);
        \draw[fill,pattern={Lines[angle=-40]},pattern color=black!40!white] (0,0) circle (1) node[draw,circle,fill=black,inner sep=1pt](X) {};
      \end{scope}
      \begin{scope}
        \clip (-2,-0.5) rectangle (2,-1.1);
        \draw[dashed](0,0) circle (1);
      \end{scope}
      \draw (-8,-0.5)--(8,-0.5) node [right] {$\mathrm{Im}(z)=-t_{n}R_{n}$};

      \node at (X) [right] {$0$};
      \node at (1,0) [right] {$D(0,\epsilon_{n}R_{n})$};
    \end{tikzpicture}
    \caption{The domain of the rescaled map $v^{1}_n$ is the shaded region. Depending on the limit of $t_{n}R_{n}$, the domain is either expanding to cover the entire complex plane $\C$, or an upper half-plane.}
    \label{fig:domain-rescaled-map}
  \end{figure}
  Pass to a further subsequence so that $t_{n}R_{n}$ converges in $[0,\infty]$. There are two cases to consider: if $t_{n}R_{n}$ converges to a finite number, then the domains of $v_{n}^{1}$ converge to an upper half-plane. On the other hand, if $t_{n}R_{n}$ diverges to $\infty$, then the domains of $v_{n}^{1}$ converge to $\C$.
  
  By ``converge to a half plane'' we mean that there exists a half plane $H$ so that any compact set in $H$ is eventually contained in the domain of $v_{n}^{1}$.

  It is straightforward to conclude that:
  \begin{equation}\label{eq:gradient-bound-147}
    \abs{\d v_{n}^{1}(0)}=1\text{ and }\abs{\d v_{n}^{1}(z)}\le 2
  \end{equation}
  for all $z$ in the domain of $v_{n}^{1}$ (using \ref{item:hc2}). The Arzel\`a-Ascoli theorem implies that $v_{n}^{1}$ converges in $C^{0}_{\mathrm{loc}}$ to a continuous function $v_{\infty}^{1}:\Omega\to \R\times Y$ with where $\Omega$ is either a half-plane $H$ or $\C$. If $\Omega$ is a half-plane, then the aforementioned $C^{0}_{\mathrm{loc}}$ convergence implies that $v_{\infty}^{1}$ maps the boundary onto $L$.

  Corollary \ref{corollary:apriori-estimates-2} implies the higher derivatives of $v_{n}^{1}$ are uniformly bounded on compact sets. These $C^{k}_{\mathrm{loc}}$ bounds upgrade the conclusion of the Arzel\`a-Ascoli theorem to conclude that the limit map $v_{\infty}^{1}$ is smooth and $v_{n}^{1}$ converges in $C^{\infty}_{\mathrm{loc}}$ to $v_{\infty}^{1}$. In particular, $v^{1}_{\infty}$ is holomorphic, and $\abs{\d v^{1}_{\infty}(0)}=1$. It follows that $v^{1}_{\infty}$ is non-constant. The energy of $v^{1}_{\infty}$ provides a lower bound for the limiting energy of $v^{1}_{n}$. Thus $\limsup E(w^{1}_{n})\ge \hbar$, as the energy of $w^{1}_{n}$ is at least the energy of $v^{1}_{n}$ since the domain of $v^{1}_{n}$ is a conformal reparametrization of part of the domain of $w^{1}_{n}$.

  By passing to a further subsequence, we conclude that a rescaled version of $w_{n}\circ c^{2}_{n}$ also converges to a bubble. By repeatedly performing this argument for $k=3,\dots,N$, and recalling that the images of $c^{k}_{n}$ are disjoint, we conclude that
  \begin{equation*}
    \limsup_{n\to\infty}\int w_{n}^{*}\omega\ge N\hbar, 
  \end{equation*}
  contradicting our assumption on $N$. This completes the proof.
\end{proof}

\bibliography{citations}
\bibliographystyle{alpha}
\end{document}